\documentclass[10pt,nocut]{article}

\usepackage{./packages}
\usepackage{./notations}
\usepackage{subfig}
\def\MMSE{{{\rm MMSE}}}
\def\MSE{{{\rm MSE}}}
\def\DMSE{{{\rm DMSE}}}

\def\Tr{{{\rm Tr}}}
\def\Var{{{\rm Var}}}
\def\Cov{{{\rm Cov}}}

\usepackage{titlesec}
\makeatletter
\newcommand{\setappendix}{Appendix~\thesection:~~}
\newcommand{\setsection}{\thesection~~}
\titleformat{\section}{\bfseries\LARGE}{%
	\ifnum\pdfstrcmp{\@currenvir}{appendices}=0
	\setappendix
	\else
	\setsection
\fi}{0em}{}
\makeatother
\usepackage[titletoc]{appendix}

\begin{document}
\title{Fundamental limits of low-rank matrix estimation:\\ the non-symmetric case}
\author{Léo Miolane\footnote{%
Département d’informatique de l’ENS, École normale supérieure, CNRS, PSL Research University, 75005
Paris, France \& Inria. Email: leo.miolane@inria.fr}}
\date{}
\maketitle

\begin{abstract}
	We consider the high-dimensional inference problem where the signal is a low-rank matrix which is corrupted by an additive Gaussian noise. 
	Given a probabilistic model for the low-rank matrix, we compute the limit in the large dimension setting for the mutual information between the signal and the observations, as well as the matrix minimum mean squared error, while the rank of the signal remains constant. 
	This allows to locate the information-theoretic threshold for this estimation problem, i.e.\ the critical value of the signal intensity below which it is impossible to recover the low-rank matrix.
\end{abstract}

\section{Introduction}

Estimating a low-rank matrix from a noisy observation is a fundamental problem in statistical inference with applications in machine learning, signal processing or information theory. 
It encompass numerous classical statistical problems from PCA, sparse PCA to high-dimensional Gaussian mixture clustering.
Consider a signal matrix $\bbf{U} \bbf{V}^{\intercal}$ where $\bbf{U}$ and $\bbf{V}$ are two $n \times k$ and $m \times k$ independent matrices. We will be interested in the low-rank, high-dimensional setting, i.e.\ $k$ will remain fixed as $n,m \to \infty$ and $m/n \to \alpha>0$. Given a noisy observation $\bbf{Y}$ of the matrix $\bbf{U} \bbf{V}^{\intercal}$ we would like to reconstruct the signal.
We consider here additive white Gaussian noise $\bbf{Z}$ (where $Z_{i,j} \iid \mathcal{N}(0,1)$):
\begin{equation}\label{eq:problem}
	\bbf{Y} = \sqrt{\frac{\lambda}{n}} \, \bbf{U}\bbf{V}^{\intercal} +\bbf{Z} \,,
\end{equation}
where $\lambda$ captures the strength of the signal.
This model is often called ``spiked'' Wishart model (or spiked covariance model) and was introduced in statistics by Johnstone~\cite{johnstone2001distribution}.
In this paper, we aim at computing the best achievable performance (in term of mean squared error) for the estimation of the low-rank signal. We prove limiting expressions for the mutual information $I((\bbf{U},\bbf{V});\bbf{Y})$ and the minimum mean squared error (MMSE), as conjectured in~\cite{DBLP:conf/allerton/LesieurKZ15}. This allows us to compute the information-theoretic threshold for this estimation problem.
More precisely, we derive a critical value $\lambda_c$ such that when $\lambda < \lambda_c$ no algorithm can retrieve the signal better than a ``random guess'' whereas for $\lambda > \lambda_c$ the signal can be estimated more accurately.
As mentioned above, high-dimensional Gaussian mixture clustering can be seen as a particular instance of the matrix factorization problem~\eqref{eq:problem} (see~\cite{lesieur2016uv,banks2016information}). The present work justify therefore the non-rigorous derivation of the information-theoretic threshold for Gaussian mixture clustering from~\cite{lesieur2016uv}.
\\

Random matrix models like~\eqref{eq:problem} has received much attention in random matrix theory. In 1976 Edwards and Jones~\cite{edwards1976eigenvalue} observed using the non-rigorous ``replica'' method: ``there is a critical finite value [for $\lambda$] above which a single eigenvalue [of $\bbf{Y}/\sqrt{n}$] splits off from the semi-circular continuum of eigenvalues''. This phase transition phenomenon for the largest eigenvalue of perturbed random matrices has then been rigorously understood in the seminal work of Baik, Ben Arous and Péché~\cite{baik2005phase} and following papers~\cite{feral2007largest,benaych2012singular}. 
Suppose for instance that $\bbf{U}$ and $\bbf{V}$ are vectors with i.i.d.\ coefficients with zero mean and unit variance. Results from~\cite{benaych2012singular} give then
\begin{itemize}
	\item[--] if $\lambda\leq 1$, the top singular value of $\bbf{Y}/\sqrt{n}$ converges a.s.\ to $2$ as $n\to \infty$. Let $\bbf{\hat{u}}$ and $\bbf{\hat{v}}$ be the respectively the left and right unit singular vectors of $\bbf{Y}/\sqrt{n}$ associated with this top singular value. Then $\bbf{\hat{u}}$ and $\bbf{\hat{v}}$ have trivial correlation with the planted solution: $\frac{1}{n} \bbf{\hat{u}}^{\intercal} \bbf{U} \to 0$ and $\frac{1}{n} \bbf{\hat{v}}^{\intercal} \bbf{V} \to 0$.
	\item[--] if $\lambda>1$, the top eigenvalue of $\bbf{Y}/\sqrt{n}$ converges a.s.\ to $\sqrt{\lambda}+1/\sqrt{\lambda}>2$ as $n \to \infty$. Let $\bbf{\hat{u}}$ and $\bbf{\hat{v}}$ be the respectively the left and right unit singular vectors of $\bbf{Y}/\sqrt{n}$ associated with this top singular value. Then $\bbf{\hat{u}}$ and $\bbf{\hat{v}}$ achieve a non-trivial correlation with the solution: $(\frac{1}{n} \bbf{\hat{u}}^{\intercal} \bbf{U})^2 \to 1 - 1/\lambda$ and $(\frac{1}{n}\bbf{\hat{v}}^{\intercal} \bbf{V})^2 \to 1 - 1/\lambda$.
\end{itemize}
This means that when $\lambda$ goes below $1$, the singular vector associated with the top singular value becomes suddenly uninformative. The question then arises: is it still possible to build a non trivial estimator of the signal when $\lambda \leq 1$ ? How does the optimal performance depends on $\lambda$ and the priors $P_U$ and $P_V$ on the entries of $\bbf{U}$ and $\bbf{V}$?
\\

To answer this question, one has to analyze the performance of the optimal estimator (in term of mean squared error). This estimator is known to be the posterior mean of the signal given the observations. However computing such an estimator leads to untractable expressions and exponential-time algorithms. This motivated the study of efficient message passing algorithms for solving the matrix factorization problem~\eqref{eq:problem}. 
Rangan and Fletcher~\cite{rangan2012iterative} proposed an Approximate Message Passing (AMP) algorithm (based on the previous work of~\cite{donoho2009message}) to estimate the low-rank signal.
Deshpande and Montanari~\cite{deshpande2014information} considered then the case of Bernoulli $\Ber(\epsilon)$ priors and showed that AMP was optimal for $\epsilon$ above a certain critical value $\epsilon_* > 0$. Interestingly, Lesieur et al.~\cite{DBLP:conf/isit/LesieurKZ15} conjectured using non-rigorous methods from statistical physics that the estimation problem may become hard for $\epsilon \leq \epsilon_*$: it would still be possible to recover the signal partially, but not with AMP or any polynomial-time algorithm. 
Consequently, a careful analysis of AMP algorithm as in~\cite{deshpande2014information} would fail to derive information-theoretic threshold in the presence of such hard phase.
Lesieur et al.\ also conjectured in~\cite{DBLP:conf/allerton/LesieurKZ15} limiting expression for the mutual information and the MMSE.
This conjecture was recently proved for the symmetric ($\bbf{U}=\bbf{V}$) case by~\cite{barbier2016mutual,lelarge2016fundamental}.

A completely different proof technique based on second moment computations and contiguity has been used to derive upper and lower bounds for the information-theoretic threshold. See the recent works~\cite{banks2016information,perry2016optimality,perry2016statistical} and the references therein. These bounds are however not expected to be tight in the regime considered in this paper.
\\

In this paper we extend and deepen the ideas of~\cite{lelarge2016fundamental} to prove the limiting expressions for the mutual information and the MMSE conjectured in~\cite{DBLP:conf/allerton/LesieurKZ15}. It builds on the mathematical approach of the Sherrington-Kirkpatrick (SK) model: see the books of Talagrand~\cite{talagrand2010meanfield1} and Panchenko~\cite{panchenko2013SK}.
Our estimation problem is indeed equivalent to a bipartite spin glass model that is closely related to SK model studied in the groundbreaking book of Mézard, Parisi and Virasoro~\cite{mezard1987spin}.
The methods developed in~\cite{mezard1987spin} have then been widely applied to other spin glass models, and in particular models arising from Bayesian estimation problems. 
This class of models enjoys specific properties due to the presence of the planted (hidden) solution of the estimation problem and to the fact that the parameters of the inference channel (noise, priors...) are supposed to be known by the statistician. In the statistical physics jargon, the system is on the ``Nishimori line'' (see~\cite{nishimori2001statistical,iba1999nishimori,korada2009exact}), a region of the phase diagram where no ``replica symmetry breaking'' occurs. These properties will play a crucial role in our proofs. 
They imply that important quantities will concentrate around their means: the system will then be characterized using only few parameters.
For a detailed introduction to the connections between statistical physics and statistical inference, see~\cite{zdeborova2016statistical}.
Bipartite spin glasses are also of special interest because they are related to Hopfield model~\cite{hopfield1982neural}.
The bipartite SK model has been investigated in~\cite{barra2010replica,barra2011equilibrium}, but the study relies on an additional hypothesis, namely the ``replica-symmetric'' assumption which will be verified for our ``planted'' model. 
\\

\noindent\textbf{Acknowledgments.} The author is grateful to M. Lelarge for numerous comments and feedback and to L. Zdeborov\'a and F. Krzakala for pointing out interesting papers.

\section{Main results}

\subsection{Rank-one matrix estimation} \label{sec:rs_formula}
For simplicity, we first focus on the rank-one case ($k=1$). The extension to finite-rank is then presented in Section~\ref{sec:multidim}.
Let $P_U$ and $P_V$ be two probability distributions on $\R$ with finite second moment and such that $\Var_{P_U}(U) , \Var_{P_V}(V) > 0$. 
Let $\lambda > 0$  and consider independent vectors $(U_i)_{1 \leq i \leq n} \iid P_U$ and $(V_j)_{1 \leq j \leq m} \iid P_V$. We observe
\begin{equation} \label{eq:model}
	Y_{i,j} = \sqrt{\frac{\lambda}{n}} U_i V_j + Z_{i,j}, \quad \text{for} \ 1 \leq i \leq n \ \text{and} \ 1 \leq j \leq m \,,
\end{equation}
where $Z_{i,j}$ are i.i.d.\ standard normal random variables that account for noise. In the following, $\E$ will denote the expectation with respect to the variables $(\bbf{U},\bbf{V})$ and $\bbf{Z}$. 

We will be interested in the high-dimensional limit where $n,m \to \infty$ while $m/n \to \alpha >0$.
Our main quantity of interest is the minimal mean squared error for the estimation of the matrix $\bbf{U} \bbf{V}^{\intercal}$ given the observation of the matrix $\bbf{Y}$:
\begin{align*}
	\MMSE_n(\lambda) 
	&= \min_{\hat{\theta}} \left\{ \frac{1}{n m } \sum_{i=1}^n \sum_{j=1}^m \E\left[ \left(U_i V_j - \hat{\theta}_{i,j}(\bbf{Y}) \right)^2\right] \right\} \label{eq:def_mmse_min_intro}
	=\frac{1}{n m} \sum_{i,j} \E \left[ \left(U_i V_j - \E\left[U_i V_j |\bbf{Y} \right]\right)^2\right],
\end{align*}
where the minimum is taken over all estimators $\hat{\theta}$ (i.e.\ measurable functions of the observations $\bbf{Y}$).
In order to get an upper bound on the matrix minimum mean squared error, 
we consider the ``dummy'' estimator given by
$\hat{\theta}_{i,j} = \E[U_i V_j]$ for all $i,j$. This estimator does not depend on the observations $\bbf{Y}$ and achieves a mean squared error
$$
\DMSE = \lim_{n,m \to \infty} \frac{1}{n m} \sum_{i,j} \E \left[ \left(U_i V_j - \E [U_i V_j] \right)^2\right] 
= \E[U^2] \E[V^2] - (\E U)^2 (\E V)^2 \,.
$$

Our goal is to locate the information-theoretic threshold for the estimation problem~\eqref{eq:model}, i.e.\ the value of $\lambda$ below which it not possible to estimate the matrix better than a dummy estimator, when $n \to \infty$.
We need therefore to compute the limit of $\MMSE_n(\lambda)$ as $n \to \infty$, for any value of $\lambda$. We will see in the sequel that this reduces to the computation of the limit of the mutual information $\frac{1}{n} I\big((\bbf{U},\bbf{V}); \bbf{Y} \big)$.

\subsection{Connection with statistical physics}

We will now connect our statistical estimation problem~\eqref{eq:model} with statistical physics concepts, namely the notions of Hamiltonian, free energy, replicas and overlap.
It will be convenient to express the posterior distribution of $(\bbf{U},\bbf{V})$ given $\bbf{Y}$ in a ``Boltzmann'' form.
We define the Hamiltonian 
\begin{equation} \label{eq:hamiltonian}
	H_n(\bbf{u},\bbf{v}) = \sum_{i=1}^n \sum_{j=1}^m \sqrt{\frac{\lambda}{n}} u_i v_j Z_{i,j} + \frac{\lambda}{n} u_i U_i v_j V_j - \frac{\lambda}{2n} u_i^2 v_j^2
	, \ \ \text{for} \  (\bbf{u},\bbf{v}) \in \R^n \times \R^m.
\end{equation}
The posterior distribution of $(\bbf{U},\bbf{V})$ given $\bbf{Y}$ is then
\begin{equation} \label{eq:posterior}
	P\big((\bbf{u},\bbf{v}) \big| \bbf{Y}\big) = \frac{1}{\mathcal{Z}_n} P_U^{\otimes n}(\bbf{u}) P_V^{\otimes m}(\bbf{v}) \exp\left(\sum_{i,j} \sqrt{\frac{\lambda}{n}} Y_{i,j} u_i v_j - \frac{\lambda}{2n}u_i^2 v_j^2\right)= \frac{1}{\mathcal{Z}_n} P_U^{\otimes n}(\bbf{u}) P_V^{\otimes m}(\bbf{v}) e^{H_n(\bbf{u},\bbf{v})} \,,
\end{equation}
where $\mathcal{Z}_n = \!\!\int \! dP_U^{\otimes n}(\bbf{u}) dP_V^{\otimes m}(\bbf{v}) e^{H_n(\bbf{u},\bbf{v})}$ is the appropriate normalization. The free energy of this model is defined as
$$
F_n = \frac{1}{n} \E \log \mathcal{Z}_n = \frac{1}{n} \E \log \left( \int_{\bbf{u},\bbf{v}}dP_U^{\otimes n}(\bbf{u}) dP_V^{\otimes m}(\bbf{v}) e^{H_n(\bbf{u},\bbf{v})} \right).
$$
In statistical physics, the free energy is a fundamental quantity that encodes a lot of information about the system. For instance, its derivative with respect to the inverse temperature corresponds to the average energy. In our context of statistical inference, the free energy contains a lot of relevant information about our estimation problem. In particular, we will see that it corresponds (up to an affine transformation) to the mutual information $\frac{1}{n} I\big((\bbf{U},\bbf{V}); \bbf{Y} \big)$ of the observation channel. Moreover, its derivative with respect to the signal-to-noise ratio $\lambda$ (which plays the role of the inverse temperature) is linked to the minimum mean-square error of our problem by the ``I-MMSE Theorem'', see~\cite{guo2005mutual}. The asymptotic behavior of the mutual information and the $\MMSE$ will therefore be linked to the limit of the free energy. 
\\

We introduce now central notions of the study of spin glasses: \textit{Gibbs measure}, \textit{replica} and \textit{overlap}.
In our context we define the \textit{Gibbs measure} $\langle \cdot \rangle_n$ as the posterior distribution~\eqref{eq:posterior}. $\langle \cdot \rangle_n$ is thus a random probability measure (depending on $\bbf{Y}$) on $\R^n \times \R^m$. 
We will write, for $k\geq 1$ (provided that the expectation on the right is well-defined)
$$
\big\langle f(\bbf{u}^{(1)},\bbf{v}^{(1)}, \dots, \bbf{u}^{(k)},\bbf{v}^{(k)}) \big\rangle_n
=
\frac{1}{\mathcal{Z}_n^k} \int
f(\bbf{u}^{(1)}, \dots, \bbf{v}^{(k)})
dP_U^{\otimes n}(\bbf{u}^{(1)}) \dots dP_V^{\otimes m}(\bbf{v}^{(k)})
\exp \Big(\sum_{l=1}^k H_n(\bbf{u}^{(l)},\bbf{v}^{(l)}) \Big) \,,
$$
the expectation of a function $f$ applied to $k$ i.i.d.\ samples (conditionally to $\bbf{Y}$) $(\bbf{u}^{(1)},\bbf{v}^{(1)}), \dots, (\bbf{u}^{(k)},\bbf{v}^{(k)})$ from $\langle \cdot \rangle_n$. 
Such samples will be called \textit{replicas}.
For $\bbf{x}^{(1)},\bbf{x}^{(2)} \in \R^N$ we define the \textit{overlap} between $\bbf{x}^{(1)}$ and $\bbf{x}^{(2)}$ as the rescaled scalar product
\begin{equation} \label{eq:overlap}
	\bbf{x}^{(1)}.\bbf{x}^{(2)} = \frac{1}{N} \sum_{i=1}^N x^{(1)}_i x^{(2)}_i.
\end{equation}

Before moving to the asymptotic analysis of the inference problem~\eqref{eq:model}, we need to state a fundamental identity (which is in fact nothing more than Bayes rule) that will be used repeatedly. It was used by Nishimori (see for instance~\cite{nishimori2001statistical}) and extensively used in the context of Bayesian inference, see~\cite{iba1999nishimori,korada2010tight,zdeborova2016statistical}.
It express the fact that the planted configuration $(\bbf{U},\bbf{V})$ behaves like a sample $(\bbf{u},\bbf{v})$ from the posterior distribution $\P((\bbf{U},\bbf{V})=.|\bbf{Y})$.

\begin{proposition}[Nishimori identity] \label{prop:nishimori}
	Let $(\bbf{X},\bbf{Y})$ be a couple of random variables on a polish space. Let $k \geq 1$ and let $\bbf{x}^{(1)}, \dots, \bbf{x}^{(k)}$ be $k$ i.i.d.\ samples (given $\bbf{Y}$) from the distribution $\P(\bbf{X}=. | \bbf{Y})$, independently of every other random variables. Let us denote $\langle \cdot \rangle$ the expectation with respect to $\P(\bbf{X}=. | \bbf{Y})$ and $\E$ the expectation with respect to $(\bbf{X},\bbf{Y})$. Then, for all continuous bounded function $f$
	$$
	\E \langle f(\bbf{Y},\bbf{x}^{(1)}, \dots, \bbf{x}^{(k)}) \rangle
	=
	\E \langle f(\bbf{Y},\bbf{x}^{(1)}, \dots, \bbf{x}^{(k-1)}, \bbf{X}) \rangle \,.
	$$
\end{proposition}

\begin{proof}
	It is equivalent to sample the couple $(\bbf{X},\bbf{Y})$ according to its joint distribution or to sample first $\bbf{Y}$ according to its marginal distribution and then to sample $\bbf{X}$ conditionally to $\bbf{Y}$ from its conditional distribution $\P(\bbf{X}=.|\bbf{Y})$. Thus the $(k+1)$-tuple $(\bbf{Y},\bbf{x}^{(1)}, \dots,\bbf{x}^{(k)})$ has the same law than $(\bbf{Y},\bbf{x}^{(1)},\dots,\bbf{x}^{(k-1)},\bbf{X})$.
\end{proof}
\\

We will now illustrate the concepts of Gibbs distribution, replicas and the Nishimori identity by computing the derivative of the free energy $F_n$ with respect to the signal $\lambda$. The arguments used in this computation will be used repeatedly in the proofs of this paper.
$\lambda \mapsto F_n$ is differentiable over $(0, + \infty)$ and for $\lambda >0$
$$
F_n'(\lambda) = \frac{1}{n} \E \Big\langle \sum_{i,j} 
\frac{1}{2 \sqrt{\lambda n}} u_i v_j Z_{i,j} + \frac{1}{n} u_i U_i v_j V_j - \frac{1}{2n} u_i^2 v_j^2 \Big\rangle_n,
$$
where $(\bbf{u},\bbf{v})$ is a replica sampled from the Gibbs distribution $\langle \cdot \rangle_n$. Let $1 \leq i,j \leq n$. Using the Gaussian integration by parts, we have
$$
\E \left[
	Z_{i,j} \big\langle u_i v_j \big\rangle_n
\right]
=
\sqrt{\frac{\lambda}{n}} \E \left[
	\big\langle u_i^2 v_j^2 \big\rangle_n
\right]
- 
\sqrt{\frac{\lambda}{n}} \E \left[
	\big\langle u_i v_j \big\rangle_n^2
\right].
$$
Let $(\bbf{u}^{(1)},\bbf{v}^{(1)})$ and  $(\bbf{u}^{(2)},\bbf{v}^{(2)})$ be two independent replicas from the Gibbs distribution $\langle \cdot \rangle_n$. 
We have then $\E \langle u_i v_j \rangle_n^2 = \E \langle u_i^{(1)} v_j^{(1)} u_i^{(2)} v_j^{(2)} \rangle_n$.
We use now the Nishimori identity (Proposition~\ref{prop:nishimori}) to obtain $\E \langle u_i^{(1)} v_j^{(1)} u_i^{(2)} v_j^{(2)} \rangle_n = \E \langle u_i v_j U_i V_j \rangle_n$. Consequently $ \E \left[ Z_{i,j} \big\langle u_i v_j \big\rangle_n \right] = \sqrt{\frac{\lambda}{n}} \E \big\langle u_i^2 v_j^2 \big\rangle_n - \sqrt{\frac{\lambda}{n}} \E \langle u_i v_j U_i V_j \rangle_n$ and finally
\begin{equation} \label{eq:f_n_p}
	F_n'(\lambda) = \frac{1}{2n^2} \E \Big\langle \sum_{i,j} 
		u_i U_i v_j V_j \Big\rangle_n = \frac{m}{2n} \E \Big\langle (\bbf{u}.\bbf{U})(\bbf{v}.\bbf{V}) \Big\rangle_n \,.
\end{equation}

\subsection{Effective scalar channel} \label{sec:scalar_channel}

As we will see in Theorem~\ref{th:rs_formula}, the limit of $F_n$ is linked to a simple 1-dimensional inference problem.
Let $P_X$ be a probability distribution with finite second moment.
Let $\gamma \geq 0$ and consider the following observation channel:
\begin{equation} \label{eq:scalar_channel}
	Y = \sqrt{\gamma} X + Z \,,
\end{equation}
where the signal $X \sim P_X$ and the noise $Z \sim \mathcal{N}(0,1)$ are independent random variables.
Note that the posterior distribution of $X$ knowing $Y$ is then given by
\begin{equation} \label{eq:posterior_scalar}
dP(X=x|Y) = \frac{1}{\cZ(Y)} dP_X(x)e^{Y\sqrt{\gamma}x-\frac{\gamma x^2}{2}}
\end{equation}
where $\cZ(Y)$ is the normalization: $\cZ(Y) = \int dP_X(x)e^{Y\sqrt{\gamma}x-\frac{\gamma x^2}{2}} =\int dP_X(x)e^{\gamma x X+\sqrt{\gamma}x Z -\frac{\gamma x^2}{2}}$.
We define
\begin{align}
	F_{P_X}: \gamma \in \R_+ \mapsto \E\big[X \E[X|Y]\big] =
	\E \left[ \frac{1}{\mathcal{Z}(Y)} \int  x X e^{\sqrt{\gamma} Y x - \frac{\gamma}{2} x^2}d P_X(x) \right] \,,
\label{eq:def_fx}
\end{align}
We define also the free energy of the channel as $\E [\log \cZ(Y)]$, which is a function of the signal-to-noise ratio $\gamma$:
\begin{equation} \label{eq:def_psi}
	\psi_{P_X}: \gamma \in \R_+ \mapsto \E \log \left( \int dP_X(x) \exp\big(\sqrt{\gamma} x Z + \gamma x X - \frac{\gamma}{2} x^2\big) 
\right) \,.
\end{equation}

The main properties of the functions $\psi_{P_X}$ and $F_{P_X}$ are presented in Appendix~\ref{sec:scalar_channel_proofs}.
In the sequel we will consider the scalar channel~\eqref{eq:scalar_channel} for $P_X = P_U$ or $P_X=P_V$.
We will be interested in values of the signal intensity $(\gamma_1, \gamma_2)$ that satisfy fixed some point equations. 

\begin{definition} \label{def:gamma1}
	We define the set $\Gamma(\lambda,\alpha)$ as
	\begin{equation} \label{eq:def_gamma}
		\Gamma(\lambda,\alpha) = \left\{
			(q_u,q_v) \in \R_+^2 \ \big| \
			q_u = F_{P_U}(\lambda \alpha q_v)  \ \text{and} \
			q_v = F_{P_V}(\lambda q_u)
		\right\} \,.
	\end{equation}
\end{definition}
A simple application of Brouwer's fixed point Theorem (see Proposition~\ref{prop:gamma_non_empty} in Appendix~\ref{sec:gamma_non_empty}) gives that $\Gamma(\lambda,\alpha) \neq \emptyset$.



\subsection{The Replica-Symmetric formula and its consequences}

The limit of $F_n$ is expressed using the following function
$$
\mathcal{F}:(\lambda,\alpha,q_u,q_v) \mapsto 
\psi_{P_U}(\lambda \alpha q_v) + \alpha \psi_{P_V}(\lambda q_u)
- \frac{\lambda \alpha}{2} q_u q_v \,.
$$
$\mathcal{F}$ corresponds to the free energy of the two scalar channels~\eqref{eq:scalar_channel} associated to $P_U$ and $P_V$, minus the term $\frac{\lambda \alpha}{2} q_u q_v$.
The Replica-Symmetric formula states that the free energy $F_n$ converges to the supremum of $\mathcal{F}$ over $\Gamma(\lambda,\alpha)$.

\begin{theorem}[Replica-Symmetric formula] \label{th:rs_formula}
	For all $\lambda,\alpha >0$,
	\begin{equation} \label{eq:lim_fn}
		F_n(\lambda) \xrightarrow[n \to \infty]{} \sup_{(q_u,q_v) \in \Gamma(\lambda,\alpha)} \mathcal{F}(\lambda,\alpha,q_u,q_v) = \sup_{q_v \geq 0} \inf_{q_u \geq 0} \mathcal{F}(\lambda,\alpha,q_u,q_v) \,.
\end{equation}
	Moreover, these extrema are achieved over the same couples $(q_u,q_v) \in \Gamma(\lambda,\alpha)$.
\end{theorem}
Theorem~\ref{th:rs_formula} is (with Theorem~\ref{th:rs_formula_multidim} that generalizes the result to any multidimensional input distribution) the main result of this paper and is proved in Section~\ref{sec:proof_rs}. This proves a conjecture from~\cite{DBLP:conf/allerton/LesieurKZ15}, in particular $\mathcal{F}$ corresponds to the ``Bethe free energy'' (\cite{DBLP:conf/allerton/LesieurKZ15}, Equation 47). The Replica-Symmetric formula allows to compute the limit of the mutual information for the inference channel~\eqref{eq:model}.

\begin{corollary}[Limit of the mutual information]\label{cor:limit_i}
	$$
	\frac{1}{n} I((\bbf{U},\bbf{V});\bbf{Y}) \xrightarrow[n \to \infty]{} \frac{\alpha \lambda}{2} \E[U^2] \E[V^2] - \sup_{(q_u,q_v) \in \Gamma(\lambda,\alpha)} \mathcal{F}(\lambda,\alpha,q_u,q_v) \,.
	$$
\end{corollary}
\begin{proof}
	The joint distribution $P_{(\bbf{U},\bbf{V};\bbf{Y})}$ is absolutely continuous with respect to the product $P_{(\bbf{U},\bbf{V})} \otimes P_{\bbf{Y}}$ with Radon-Nikodym derivative:
	$$
	\frac{dP_{(\bbf{U},\bbf{V};\bbf{Y})}}{dP_{(\bbf{U},\bbf{V})} \!\otimes\! P_{\bbf{Y}}}
	(\bbf{U},\bbf{V};\bbf{Y})
	=
	\frac{\exp\Big(-\frac{1}{2}\|\bbf{Y} - \sqrt{\frac{\lambda}{n}} \bbf{U}^{\intercal} \bbf{V}\|^2\Big)}%
	{\int dP_U^{\otimes n}(\bbf{u}) dP_V^{\otimes m}(\bbf{v}) \exp\Big(-\frac{1}{2}\|\bbf{Y} - \sqrt{\frac{\lambda}{n}} \bbf{u}^{\intercal} \bbf{v}\|^2\Big)} \,.
	$$
	Therefore the mutual information is equal to
	$$
	I\big(\bbf{U},\bbf{V};\bbf{Y}\big)
	=
	\E \log \Big(
	\frac{dP_{(\bbf{U},\bbf{V};\bbf{Y})}}{dP_{(\bbf{U},\bbf{V})} \!\otimes\! P_{\bbf{Y}}}
(\bbf{U},\bbf{V};\bbf{Y}) \Big)
= - nF_n(\lambda) + \frac{\lambda m}{2} \E_{P_U}[U^2]\E_{P_V}[V^2]\,.
	$$
\end{proof}
\\

Theorem~\ref{th:rs_formula} allows also to compute the limit of the $\MMSE$:

\begin{proposition}[Limit of the $\MMSE$] \label{prop:csq}
	Let 
	$$
	D_{\alpha} = \Big\{ \lambda > 0 \ \Big| \ \mathcal{F}(\lambda,\alpha,\cdot,\cdot) \ \text{has a unique maximizer $(q_u^*(\lambda),q_v^*(\lambda))$ on} \ \Gamma(\lambda,\alpha) \Big\}.
	$$
	Then $D_{\alpha}$ is equal to $(0,+\infty)$ minus a countable set and for all $\lambda \in D_{\alpha}$ (and thus almost every $\lambda>0$)
	\begin{align}
		&\MMSE_n(\lambda) \xrightarrow[n \to \infty]{} \E[U^2] \E[V^2] - q_u^*(\lambda) q_v^*(\lambda)\,. \label{eq:lim_mmse}
	\end{align}
\end{proposition}

Again, this was conjectured in~\cite{DBLP:conf/allerton/LesieurKZ15}: the performance of the Bayes-optimal estimator (i.e.\ the MMSE) corresponds to the fixed point of the state-evolution equations~\eqref{eq:def_gamma} which has the greatest Bethe free energy $\mathcal{F}$.
Before proving Proposition~\ref{prop:csq}, let us deduce the information-theoretic threshold for our matrix estimation problem.
Let us define 
\begin{equation} \label{eq:lambda_c}
	\lambda_c(\alpha) = \sup \left\{\lambda \in D_{\alpha} \ \middle| \ q_u^*(\lambda) q_v^*(\lambda) = (\E U)^2 (\E V)^2 \right\}.
\end{equation}
If the set of the left-hand side is empty, one define $\lambda_c(\alpha) = 0$.
Proposition~\ref{prop:csq} gives that $\lambda_c(\alpha)$ is the information-theoretic threshold for the estimation of $\bbf{U} \bbf{V}^{\intercal}$ given $\bbf{Y}$:
\begin{itemize}
	\item[--] If $\lambda < \lambda_c(\alpha)$, then $\MMSE_n(\lambda) \xrightarrow[n \to \infty]{} \DMSE$. It is not possible to reconstruct the signal $\bbf{U} \bbf{V}^{\intercal}$ better than a ``dummy'' estimator.
	\item[--] If $\lambda > \lambda_c(\alpha)$, then $\lim\limits_{n \to \infty} \MMSE_n(\lambda) < \DMSE$. It is possible to reconstruct the signal $\bbf{U} \bbf{V}^{\intercal}$ better than a ``dummy'' estimator.
\end{itemize}

\begin{proof}[of Proposition~\ref{prop:csq}]
	Let $\lambda > 0$.
	Compute, using the Nishimori identity (Proposition~\ref{prop:nishimori}),
	\begin{align*}
		\MMSE_n(\lambda) &= \frac{1}{n m} \E\Big[ \sum_{i,j} (U_i V_j - \langle u_i v_j \rangle_n)^2 \Big]
		= \frac{1}{n m} \E\Big[ \sum_{i,j} (U_i V_j)^2 + \langle u_i v_j \rangle_n^2 - 2 \langle u_i U_i v_j V_j \rangle_n \Big]
		\\
		&= \E[U^2] \E[V^2] - \E \Big\langle (\bbf{u}.\bbf{U})(\bbf{v}.\bbf{V}) \Big\rangle_n
		= \E[U^2] \E[V^2] - \frac{2n}{m} F_n'(\lambda) \,,
	\end{align*}
	where we used Equation~\eqref{eq:f_n_p} in the last equality.
	$\MMSE_n$ is a non-increasing function of the signal-to-noise ratio $\lambda$. Consequently, $F_n'$ is non-decreasing: $\lambda \mapsto F_n(\lambda)$ is convex. 
	Define the function 
	\begin{equation} \label{eq:Phi_ch}
		\Phi_{\alpha}: \lambda \mapsto \lim_{n \to \infty} F_n(\lambda) = \sup_{q_v \geq 0}\left\{ \psi_{P_U}(\lambda \alpha q_v) + \inf_{q_u\geq 0} \Big\{  \alpha \psi_{P_V}(\lambda q_u) - \frac{\lambda \alpha q_u q_v}{2} \Big\} \right\} \,.
	\end{equation}
	The value of the infimum over $q_u$ does not depend on $\lambda$ and $\lambda \mapsto \psi_{P_U}(\lambda \alpha q_v)$ is differentiable with derivative equal to $\frac{\alpha q_v}{2} F_{P_U}(\lambda \alpha q_v)$. 
	The supremum over $q_v$ is achieved over a compact set (by Theorem~\ref{th:rs_formula}, because $\Gamma(\lambda,\alpha)$ is compact), thus an envelope theorem (Corollary~4 from~\cite{milgrom2002envelope}) gives that $\Phi_{\alpha}$ is differentiable at $\lambda>0$
	if and only if 
	$$
	\left\{ \frac{\alpha q_v}{2} F_{P_U}(\lambda \alpha q_v) \ \middle| \ q_v \ \text{maximizes the right-hand side of~\eqref{eq:Phi_ch}} \right\} 
	$$
	is a singleton. By strict monotonicity of $F_{P_U}$ (see Lemma~\ref{lem:general_convex}), one see that $\Phi_{\alpha}$ is differentiable at $\lambda$ if and only if there is only one couple $(q_u,q_v) \in \Gamma(\lambda,\alpha)$ that achieves the extrema in~\eqref{eq:lim_fn}. Therefore, the set of point at which $\Phi_{\alpha}$ is differentiable is exactly $D_{\alpha}$ and for all $\lambda \in D_{\alpha}$:
	\begin{equation} \label{eq:der_Phi}
		\Phi_{\alpha}'(\lambda) = \frac{\alpha q_v^*(\lambda) F_{P_U}(\lambda \alpha q_v^*)}{2} = \frac{ \alpha q_u^*(\lambda) q_v^*(\lambda)}{2} \,.
	\end{equation}
	$\Phi_{\alpha}$ is convex (as a limit of convex functions) and is thus differentiable everywhere except a countable set. This proves the first assertion. By definition, $F_n(\lambda) \to \Phi_{\alpha}(\lambda)$ for all $\lambda >0$.
	By convexity of $F_n$ and $\Phi_{\alpha}$, a standard analysis lemma gives that for all $\lambda \in D_{\alpha}$, $F_n'(\lambda) \xrightarrow[n \to \infty]{} \Phi_{\alpha}'(\lambda)$. The lemma follows.
\end{proof}

\subsection{Numerical experiments}\label{sec:num}

In this section we illustrate our results with numerical experiment: we compute the MMSE for different priors and noise levels. For simplicity, we will considers priors for $U$ and $V$ with zero mean, unit variance and with 2 values in their support:
$$
U \sim p_u \delta\left(\sqrt{\frac{1-p_u}{p_u}}\right)
+ (1-p_u) \delta\left(-\sqrt{\frac{p_u}{1-p_u}}\right)
\quad
\text{and}
\quad
V \sim p_v \delta\left(\sqrt{\frac{1-p_v}{p_v}}\right)
+ (1-p_v) \delta\left(-\sqrt{\frac{p_v}{1-p_v}}\right) \,,
$$
where $0<p_u,p_v<1$ characterize the asymmetry of the priors.
We will first compare the performance of PCA with the MMSE. The results of~\cite{benaych2012singular} mentioned in the introduction give:
$$
\MSE^{\text{PCA}}_n \xrightarrow[n \to \infty]{} 
\begin{cases}
	1 & \text{if} \ \lambda \leq 1 \,,\\
	\frac{1}{\lambda}(1- \frac{1}{\lambda}) & \text{otherwise.}
\end{cases}
$$
\begin{figure}[!h]
	\centering
	\subfloat[Symmetric priors: $p_u=p_v=1/2$]{\includegraphics[width = 0.47\textwidth]{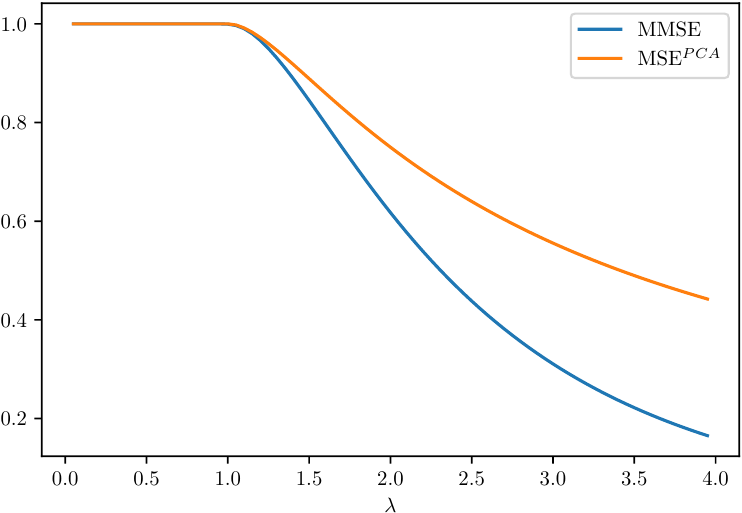}}		
	\quad
	\subfloat[Asymmetric priors: $p_u=p_v=0.1$]{\includegraphics[width = 0.47\textwidth]{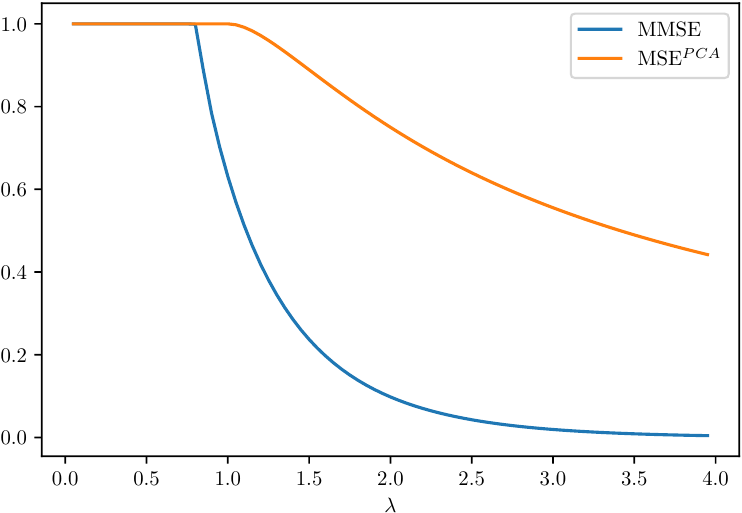}}		
	\caption{Sub-optimality of PCA for symmetric and asymmetric priors.}
	\label{fig:mmse_pca}
\end{figure}
For the symmetric case ($p_u = p_v = 1/2$), we see that PCA achieves a non-trivial performance as soon it is information-theoretically possible to estimate the signal ($\lambda > 1$). It is however sub-optimal. In the asymmetric case, the information-theoretic threshold $\lambda_c$ is strictly below $1$. Thus, for $\lambda_c < \lambda < 1$, it is theoretically possible to achieve a non trivial performance but PCA fails. It is conjectured that any polynomial-time algorithm would fail in this regime (see for instance~\cite{lesieur2017constrained}).




\subsection{Algorithmic interpretation: Approximate Message Passing (AMP)} \label{sec:amp}

Approximate Message Passing (AMP) algorithms, introduced in~\cite{donoho2009message}, have been widely used to study the matrix factorization problem~\eqref{eq:model}. They have been used in ~\cite{rangan2012iterative} for the rank-one case and then in~\cite{matsushita2013low} for finite-rank matrix estimation. 
For detailed review and developments about the study of matrix factorization with message-passing algorithms, see~\cite{lesieur2017constrained}.

We introduce briefly the AMP algorithm and comment its connections with the results of the previous sections. More details about the algorithm and numerical experiments can be found in~\cite{deshpande2014information} and~\cite{lesieur2016uv}.
We would not give any proof about AMP, but the results presented here can be deduced from~\cite{javanmard2013state} and the previously mentioned articles.

The scalar channel presented in Section~\ref{sec:scalar_channel} holds a key role in the AMP algorithm. Suppose that we observed $Y_u$ and $Y_v$ that are noisy observation of respectively $U$ and $V$ through the scalar channel~\eqref{eq:scalar_channel} with signal intensities $\gamma_u$ and $\gamma_v$. The best predictions that we can make (in term of mean squared error) for $U$ and $V$ are respectively 
\begin{align}
	f_u(Y_u,\gamma_u) &=  \E[U | Y_u]  \,, \\
	f_v(Y_v,\gamma_v) &=  \E[V | Y_v] \,.
\end{align}
The performance of these estimators is measured by $F_{P_U}$ and $F_{P_V}$. Indeed $F_{P_U}(\gamma_u) = \E[U f_u(Y_u,\gamma_u)]$ and $F_{P_V}(\gamma_v)=\E[V f_v(Y_v,\gamma_v)]$ measure the correlations between the estimators $f_u$ and $f_v$ and the true values $U$ and $V$. Define $(q_u^0, q_v^0) = (0,0)$ and $(q_u^t,q_v^t)_{t \geq 1}$ through the following recursion, called ``state evolution'':
\begin{equation} \label{eq:state_evolution}
	\begin{cases}
		q_u^{t+1} &= F_{P_U}(\lambda \alpha q_v^t) \,,\\
		q_v^{t+1} &= F_{P_V}(\lambda q_u^t)\,.
	\end{cases}
\end{equation}

The AMP algorithm initializes two estimates of $\bbf{U}$ and $\bbf{V}$ by setting $(\hat{u}^0_i)_{1 \leq i \leq n} \iid P_U$, $(\hat{v}^0_j)_{1 \leq j \leq m} \iid P_V$, and follows the recursion
\begin{align*}
	\left\{
	\begin{array}{lcl}
		\bbf{u}^{t+1} &=& \bbf{Y} \bbf{\hat{v}}^t - \delta_{u}^t \bbf{\hat{u}}^t
		\\
		\bbf{v}^{t+1} &=& \bbf{Y}^{\intercal} \bbf{\hat{u}}^t - \delta_{v}^t \bbf{\hat{v}}^t
	\end{array}
	\right.
	\qquad
	\left\{
	\begin{array}{ccc}
		\bbf{\hat{u}}^{t+1} &=& f_u(\bbf{u}^{t+1},\lambda \alpha q_v^t)
		\\
		\bbf{\hat{v}}^{t+1} &=& f_v(\bbf{v}^{t+1}, \lambda q_u^t)
	\end{array}
	\right.
\end{align*}
where we extend the functions $f_u$ and $f_v$ to vector inputs by applying them coordinate by coordinate. The scalars $\delta_{u}^t$ and $\delta_{v}^t$ are linked to the partial derivatives of the functions $f_u(\cdot, \lambda \alpha q_v^t)$ and $f_v(\cdot,\lambda q_u^t)$:
\begin{align*}
	\delta_{u}^t &= \frac{1}{n} \sum_{i=1}^n \frac{\partial f_u}{\partial u} (\hat{u}_i, \lambda \alpha q_v^t)
	, \qquad 
	\delta_{v}^t = \frac{1}{n} \sum_{j=1}^m \frac{\partial f_v}{\partial v}(\hat{v}_j,\lambda q_u^t) \,.
\end{align*}
After $t$ iterations, the algorithm outputs $\bbf{\hat{u}}^t (\bbf{\hat{v}}^t)^{\intercal}$.
The AMP algorithm is particularly interesting because its evolution can be rigorously tracked (see~\cite{javanmard2013state}). For $t \geq 1$, we have almost-surely
$$
\bbf{U}.\bbf{\hat{u}}^t \xrightarrow[n \to \infty]{} q_u^t
, \quad
\bbf{V}.\bbf{\hat{v}}^t \xrightarrow[n \to \infty]{} q_v^t
\quad \text{and} \quad 
\MSE(\bbf{\hat{u}}^t (\bbf{\hat{v}}^t)^{\intercal}) \xrightarrow[n \to \infty]{} \E[U^2] \E[V^2] - q_u^t q_v^t \,.
$$
The state evolution~\eqref{eq:state_evolution} characterizes therefore the behavior of the AMP algorithm. We see that if $(q_u^t,q_v^t)_{t \geq 0}$ converges to the fixed point $(q_u^*,q_v^*) \in \Gamma(\lambda,\alpha)$ that maximizes $\cF(\lambda,\alpha,\cdot,\cdot)$, then the AMP algorithm is an optimal, polynomial-time algorithm. 
The AMP algorithm is conjectured to be the most efficient polynomial-time algorithm, even in the regime where it does not converges to the optimal fixed point.

\subsection{Extension to rank-$k$ matrix estimation} \label{sec:multidim}

We extend in this section the main results of Section~\ref{sec:rs_formula} multidimensional input distributions.
\\

Let $k \geq 1$. Let $P_U$ and $P_V$ be two probability distributions on $\R^k$ with finite second moment and such that $\Cov_{P_U}(\bbf{U})$ and $\Cov_{P_V}(\bbf{V})$ are inversible.
Consider i.i.d.\ random variables $(\bbf{U}_i)_{1 \leq i \leq n} \iid P_U$ and $(\bbf{V}_j)_{1 \leq j \leq m} \iid P_V$. 
We will study the regime where $n,m \to \infty$ and $m/n \to \alpha >0$.
We observe
\begin{equation} \label{eq:model_multidim}
	Y_{i,j} = \sqrt{\frac{\lambda}{n}} \bbf{U}_i^{\intercal} \bbf{V}_j + Z_{i,j}
	\quad \text{for} \  1 \leq i \leq n , \ 1 \leq j \leq m \,,
\end{equation}
where $Z_{i,j}$ are i.i.d.\ standard normal random variables. 
Similarly to Section~\ref{sec:rs_formula} we define the minimal mean squared error for the estimation of the matrix $\bbf{U} \bbf{V}^{\intercal}$ given the observation of the matrix $\bbf{Y}$:
\begin{align*}
	\MMSE_n(\lambda) 
	&= \min_{\hat{\theta}} \left\{ \frac{1}{n m} \sum_{\substack{1 \leq i \leq n \\1 \leq j \leq m}} \E\left[ \left(\bbf{U}_i^{\intercal} \bbf{V}_j - \hat{\theta}_{i,j}(\bbf{Y}) \right)^2\right] \right\} \label{eq:def_mmse_min_intro}
	=\frac{1}{nm} \sum_{\substack{1 \leq i \leq n \\1 \leq j \leq m}} \E \left[ \left(\bbf{U}_i^{\intercal} \bbf{V}_j - \E\left[\bbf{U}_i^{\intercal} \bbf{V}_j |\bbf{Y} \right]\right)^2\right],
\end{align*}
where the minimum is taken over all estimators $\hat{\theta}$ (i.e.\ measurable functions of the observations $\bbf{Y}$).
Define the Hamiltonian 
\begin{equation} \label{eq:hamiltonian_multidim}
	H_n(\bbf{u},\bbf{v}) = \sum_{\substack{1 \leq i \leq n \\1 \leq j \leq m}} \sqrt{\frac{\lambda}{n}} \bbf{u}_i^{\intercal} \bbf{v}_j Z_{i,j} + \frac{\lambda}{n} \bbf{u}_i^{\intercal} \bbf{v}_j \bbf{U}_i^{\intercal} \bbf{V}_j - \frac{\lambda}{2n} (\bbf{u}_i^{\intercal} \bbf{v}_j)^2
	, \quad \text{for } (\bbf{u},\bbf{v}) \in (\R^k)^n \times (\R^k)^m.
\end{equation}
The free energy is then
$$
F_n 
= \frac{1}{n} \E \log \left( \int dP_U^{\otimes n}(\bbf{u}) dP_V^{\otimes m}(\bbf{v}) e^{H_n(\bbf{u},\bbf{v})} \right) \,.
$$
We can generalize the definition of the functions $F_{P_U}$ and $F_{P_V}$ in Section~\ref{sec:scalar_channel} to the multidimensional case, and define (see Appendix~\ref{sec:scalar_channel_proofs}) for $P_X = P_U, P_V$: 
$$
F_{P_X} : \bbf{q} \in S_k^+
\mapsto
\E \left[
	\frac{\int \bbf{X}\bbf{x}^{\intercal}  \, e^{\bbf{x}^{\intercal} \bbf{q}^{1/2} \bbf{Z} + \bbf{x}^{\intercal} \bbf{q} \bbf{X} - \frac{1}{2} \bbf{x}^{\intercal} \bbf{q} \bbf{x}} d P_X(\bbf{x})}
{\int  e^{\bbf{x}^{\intercal} \bbf{q}^{1/2} \bbf{Z} + \bbf{x}^{\intercal} \bbf{q} \bbf{X} - \frac{1}{2} \bbf{x}^{\intercal} \bbf{q} \bbf{x}} d P_X(\bbf{x})}
	\right] \,,
$$
where the expectation is taken with respect $(\bbf{X},\bbf{Z} )\sim P_X \otimes \cN(\bbf{0},\bbf{I}_k)$ and
$S_k^+$ is the set of $k \times k$ positive semidefinite matrices. We define also
$$
\psi_{P_X} (\bbf{q}) = 
\E \log \int  e^{\bbf{x}^{\intercal} \bbf{q}^{1/2} \bbf{Z} + \bbf{x}^{\intercal} \bbf{q} \bbf{X} - \frac{1}{2} \bbf{x}^{\intercal} \bbf{q} \bbf{x}} d P_X(\bbf{x}) \,.
$$
The main properties of the functions $\psi_{P_X}$ and $F_{P_X}$ are presented in Appendix~\ref{sec:scalar_channel_proofs}.
\begin{definition}\label{def:gammak}
	We define the set $\Gamma(\lambda,\alpha)$ as
	\begin{equation} \label{eq:def_gamma_multidim}
		\Gamma(\lambda,\alpha) = \left\{
			(\bbf{q}_u,\bbf{q}_v) \in (S_k^+)^2 \ \big| \
			\bbf{q}_u = F_{P_U}(\lambda \alpha \bbf{q}_v) \text{ and }
			\bbf{q}_v = F_{P_V}(\lambda \bbf{q}_u)
		\right\} \,.
	\end{equation}
\end{definition}
An application of Brouwer's fixed point Theorem (see Proposition~\ref{prop:gamma_non_empty} in Appendix~\ref{sec:gamma_non_empty}) gives that $\Gamma(\lambda,\alpha) \neq \emptyset$.
Similarly to the unidimensional case we will express the limit of $F_n$ using the functions $\psi_{P_U}$ and $\psi_{P_V}$. Let
$$
\begin{array}{rccl}
	\mathcal{F}: & (\R_+^*)^2 \times (S_k^+)^2 &\to& \R \\
						  &(\lambda,\alpha,\bbf{q}_u,\bbf{q}_v) &\mapsto& \psi_{P_U}(\lambda \alpha \bbf{q}_v) + \alpha \psi_{P_V}(\lambda \bbf{q}_u) - \frac{1}{2}\lambda \alpha \Tr\big[\bbf{q}_u \bbf{q}_v\big] .
\end{array}
$$

\begin{theorem} \label{th:rs_formula_multidim}
	For all $\lambda,\alpha >0$,
	\begin{equation}\label{eq:rs_multi}
		F_n(\lambda) \xrightarrow[n \to \infty]{} 
	\sup_{(\bbf{q}_u,\bbf{q}_v) \in \Gamma(\lambda,\alpha)} \mathcal{F}(\lambda,\alpha,\bbf{q}_u,\bbf{q}_v) 
	=\sup_{\bbf{q}_v \in S_k^+} \inf_{\bbf{q}_u \in S_k^+} \mathcal{F}(\lambda,\alpha,\bbf{q}_u,\bbf{q}_v) \,.
\end{equation}
	Moreover, these extrema are achieved over the same couples $(\bbf{q}_u,\bbf{q}_v) \in \Gamma(\lambda,\alpha)$.
\end{theorem}
Theorem~\ref{th:rs_formula_multidim} will be proved in Section~\ref{sec:proof_rs_multidim}. 

\begin{proposition}[Limit of the $\MMSE$] \label{prop:csq_multidim}
	For all $\alpha>0$ we have for almost all $\lambda >0$ that all the optimal couples $(\bbf{q}_u,\bbf{q}_v)$ of~\eqref{eq:rs_multi} have the same scalar product $\Tr[\bbf{q}_u \bbf{q}_v] = Q(\lambda,\alpha)$ and
	\begin{align}
		\MMSE_n(\lambda) \xrightarrow[n \to \infty]{} \Tr\big[\E_{P_U}[\bbf{U}\bbf{U}^{\intercal}] \E_{P_V}[\bbf{V}\bbf{V}^{\intercal}]\big] - Q(\lambda,\alpha) \label{eq:lim_mmse_multidim} \,.
	\end{align}
\end{proposition}
The proof of Proposition~\ref{prop:csq_multidim} is a simple extension of the proof of Proposition~\ref{prop:csq} and is therefore left to the reader.

\section{Proof technique} \label{sec:proof_techniques}

Our proof technique is closely related to~\cite{lelarge2016fundamental} that deals with symmetric matrices. It adapts two techniques that originated from the study of the SK model:
\begin{itemize}
	\item[--] A lower bound on limit of the free energy follows from an application of Guerra's interpolation technique for the SK model (see~\cite{guerra2003broken} or~\cite{panchenko2013SK}). 
	\item[--] The converse upper bound is proved (as in~\cite{lelarge2016fundamental}) via cavity computations, inspired from the ``Aizenman-Sims-Starr scheme'' (see~\cite{aizenman2003extended} or~\cite{panchenko2013SK}). 
\end{itemize}
The transposition of these arguments to the context of Bayesian inference is made possible by the obvious but fundamental Nishimori identity (Proposition~\ref{prop:nishimori}) which states that the planted configuration $(\bbf{U},\bbf{V})$ behaves like a sample $(\bbf{u},\bbf{v})$ from the posterior distribution $\P\big((\bbf{U},\bbf{V})= \cdot \ | \ \bbf{Y}\big)$.

However, our inference model differs from the SK model in a crucial point. Under a small perturbation of the model~\eqref{eq:model}, the overlap between the planted solution $(\bbf{U},\bbf{V})$ and a sample $(\bbf{u},\bbf{v})$ from the posterior distribution concentrates around its mean (such behavior is called ``Replica-Symmetric'' in statistical physics). This property is verified for a wide class of inference problems and is a major difference with the SK model, where the overlap concentrates only at high temperature.

Our model differs also from the SK model and the low-rank symmetric matrix estimation by some lack of convexity. The Hamiltonian of the SK model is a Gaussian process $(H_n(\sigma))_{\sigma}$ indexed by the configurations $\sigma \in \{-1,1\}^n$ whose covariance structure is given by $\E [H_n(\sigma^{(1)}) H_n(\sigma^{(2)})] = n R_{1,2}^2$, where $R_{1,2} = \frac{1}{n} \sum_{i=1}^n \sigma^{(1)}_i \sigma^{(2)}_i$ is the overlap between the configurations $\sigma^{(1)}$ and $\sigma^{(2)}$. The covariance is thus a convex function of the overlap $R_{1,2}$. This property is fundamental and allows to use the powerful Guerra's interpolation technique~\cite{guerra2003broken} to derive bounds on the free energy. The low-rank symmetric matrix estimation setting ($\bbf{Y} = \sqrt{\lambda/n} \bbf{X} \bbf{X}^{\intercal} + \bbf{Z}$) enjoys analogous convexity properties and Guerra's interpolation scheme allows to obtain tight bounds on the free energy as proved in~\cite{krzakala2016mutual} and~\cite{korada2009exact}. However, these convexity properties does not holds in our case of non-symmetric matrix estimation and Guerra's interpolation strategy can not be directly applied.

For this reason one has to investigate further the overlap distribution to by-pass this lack of convexity. We mentioned above that the overlaps concentrates around their means. We will show in Section~\ref{sec:tap} that these mean values satisfies asymptotically fixed point equations. These equations are related to the TAP equations for the SK model (see~\cite{thouless1977solution},~\cite{talagrand2010meanfield1}) and are called ``state evolution equations'' in the study of Approximate Message Passing (AMP) algorithms (see Section~\ref{sec:amp} and~\cite{javanmard2013state}). Combining these state evolution equations to the classical Guerra's interpolation scheme allows to derive a tight lower bound.

\section{A decorrelation principle} \label{sec:overlap_concentration}

We present here a general concentration result for the overlap between two replicas (i.e.\ a sample from a posterior distribution), for a large class of inference problems. 
This result will hold under some small perturbation of the inference model, which will correspond to some (small) side-information given to the statistician.
This is the analog of the Ghirlanda-Guerra identities (see~\cite{ghirlanda1998general}) for the SK model: the proof will thus be closely related to the derivation of the Ghirlanda-Guerra identities from~\cite{panchenko2013SK}. In the context of Bayesian inference, a similar result was proved in~\cite{korada2010tight} for the case of CDMA systems with binary inputs.
\\

Let $P$ be a probability distribution on $\R^n$ with bounded support $S \subset [-K,K]^n$, for some $K > 0$.
Let $\bbf{X} = (X_i)_{1 \leq i \leq n} \sim P$. Let $\bbf{Y}$ be a random vector (that could correspond to some noisy observations of $\bbf{X}$) in $\R^N$.

Suppose that the distribution of $\bbf{X}$ given $\bbf{Y}$ takes the following form
$$
\P(\bbf{X} \in A \ | \ \bbf{Y}) = \frac{1}{\mathcal{Z}_n(\bbf{Y})} \int_{\bbf{x} \in A} dP(\bbf{x}) e^{H_n(\bbf{x},\bbf{Y})}, \quad \text{for all Borel set } A \subset \R^n,
$$
where $H_n$ is a measurable function on $\R^{n} \times \R^N$ that can be equal to $-\infty$ (in which case, we use the convention $\exp(- \infty) = 0$) and such that the normalizing constant $\mathcal{Z}_n(\bbf{Y}) = \int_{\bbf{x} \in S} dP(\bbf{x}) e^{H_n(\bbf{x},\bbf{Y})}$ verifies $\E |\log (\mathcal{Z}_n(\bbf{Y}))| < \infty$. We can thus define the free energy
$$
F_n = \frac{1}{n} \E \log \mathcal{Z}_n(\bbf{Y}) = \frac{1}{n} \E \log \left( \int_{\bbf{x} \in S} dP(\bbf{x}) e^{H_n(\bbf{x},\bbf{Y})} \right).
$$
From now we will simply write $H_n(\bbf{x})$ instead of $H_n(\bbf{x},\bbf{Y})$.
Let us consider a small ``perturbation'' of our model: suppose that we have some extra side-information on $\bbf{X}$ that takes the form:
\begin{equation} \label{eq:pert_scalar}
	Y'_i = a \sqrt{s_n} X_i + Z_i, \quad \text{for } 1 \leq i \leq n,
\end{equation}
where $Z_i \iid \cN(0,1)$, $(s_n) \in (0,1)^{\N}$ and $a>0$.
The posterior distribution of $\bbf{X}$ given $\bbf{Y},\bbf{Y}'$ is now $\P(\bbf{x} \, | \, \bbf{Y},\bbf{Y}') = \frac{1}{\cZ_{n,a}^{\text{(pert)}}}P(\bbf{x}) \exp\big(H_{n,a}^{\text{(pert)}}(\bbf{x})\big)$, where
$
H_{n,a}^{\text{(pert)}}(\bbf{x}) = H_{n,a}(\bbf{x}) + h_{n,a}(\bbf{x})
$
and
$$
h_{n,a}(\bbf{x}) = \sum_{i=1}^n a \sqrt{s_n} Z_i x_i + a^2 s_n x_i X_i -\frac{1}{2} a^2 s_n x_i^2 \,.
$$
$\cZ_{n,a}^{\text{(pert)}}$ is the appropriate normalization. We will denote by $\langle \cdot \rangle_{n,a}$ the expectation with respect to the posterior distribution of $\bbf{X}$ given $\bbf{Y},\bbf{Y}'$. We will write:
$$
\big\langle f(\bbf{x}^{(1)},\dots,\bbf{x}^{(k)}) \big\rangle_{n,a}
=
\frac{1}{(\cZ_{n,a}^{\text{(pert)}})^k} \int dP(\bbf{x}^{(1)}) \dots dP(\bbf{x}^{(k)})
f(\bbf{x}^{(1)},\dots,\bbf{x}^{(k)})
\exp\Big(\sum_{l=1}^k H_{n,a}^{\text{(pert)}}(\bbf{x}^{(l)})\Big) \,,
$$
for all $k \geq 1$ and all function $f$ for which this integral is well defined.
The perturbed free energy is
$$
F_{n,a}^{\text{(pert)}} = \frac{1}{n} \E \log \cZ_{n,a}^{\text{(pert)}} =  \frac{1}{n} \E \log \left( \int_{\bbf{x} \in S} dP(\bbf{x}) e^{H_{n,a}^{\text{(pert)}}(\bbf{x})} \right) \,.
$$

The next Lemma tells us that if $s_n \to 0$, then the perturbation does not affect the limit of the free-energy:
\begin{lemma} \label{lem:free_energy_pert}
	We have for all $a \geq 0$, $\big| F_n - F_{n,a}^{\text{(pert)}} \big| \leq 2 K^2 a^2 s_n$.
\end{lemma}
\begin{proof}
	Let $\langle \cdot \rangle_n$ denote the (random) measure on $\R^n$ defined as
	$
	\big\langle f(\bbf{x}) \big\rangle_n = \frac{\int_{\bbf{x} \in S} dP(\bbf{x}) f(\bbf{x}) \exp(H_n(\bbf{x}))}{\int_{\bbf{x} \in S} dP(\bbf{x}) \exp(H_n(\bbf{x}))}
	$
	for every continuous bounded function $f$. We have $F_{n,a}^{\text{(pert)}} - F_n = \frac{1}{n} \E \log \big\langle e^{h_{n,a}(\bbf{x})} \big\rangle_n$.
	Thus, using Jensen's inequality twice
	$$
	\frac{1}{n} \E \left\langle \E_{Z} h_{n,a}(\bbf{x}) \right\rangle_n  = \frac{1}{n} \E \left\langle h_{n,a}(\bbf{x}) \right\rangle_n \leq F_{n,a}^{\text{(pert)}} - F_n \leq \frac{1}{n} \E \log \E_{Z} \left\langle e^{h_{n,a}(\bbf{x})} \right\rangle_n = \frac{1}{n} \E \log \left\langle \E_{Z} e^{h_{n,a}(\bbf{x})} \right\rangle_n \,,
	$$
	where $\E_{Z}$ denotes the expectation with respect to the variables $(Z_i)_{1 \leq i \leq n}$ only. We have, for all $\bbf{x} \in S$, $|\E_{Z} h_{n,a}(\bbf{x})| \leq 2n K^2 a^2 s_n$ and $|\E_{Z} e^{h_{n,a}(\bbf{x})}| \leq e^{n K^2 a^2 s_n}$. We conclude: 
	$$
	- 2K^2 a^2 s_n \leq F_{n,a}^{\text{(pert)}} - F_n \leq K^2 a^2 s_n \,.
	$$
\end{proof}

Let us define
$$
\phi: a \mapsto \frac{1}{n s_n} \log \left( \int_{\bbf{x} \in S} dP(\bbf{x}) e^{H^{\text{(pert)}}_{n,a}(\bbf{x})} \right).
$$
Define also $v_n(s_n) = \sup_{1/2 \leq a \leq 3} \E | \phi(a) - \E \phi(a) |$. The following result shows that, in the perturbed system (under some conditions on $v_n$ and $s_n$) the overlap between two replicas concentrates asymptotically around its expected value.

\begin{theorem}[Overlap concentration] \label{th:overlap_concentration}
	Suppose that
	$$
	\begin{cases}
		v_n(s_n) &\xrightarrow[n \to \infty]{} 0\,, \\
		n s_n &\xrightarrow[n \to \infty]{} + \infty \,.
	\end{cases}
	$$
	Then we have
	$$
	\int_1^2 \E \Big\langle \big(\bbf{x}^{(1)}.\bbf{x}^{(2)} - \E\langle \bbf{x}^{(1)}.\bbf{x}^{(2)} \rangle_{n,a}\big)^2 \Big\rangle_{\!n,a} da \xrightarrow[n \to \infty]{} 0 \,.
	$$
\end{theorem}

Theorem~\ref{th:overlap_concentration} is the analog of Theorem~3.2 (the Ghirlanda-Guerra identities, see~\cite{ghirlanda1998general}) from~\cite{panchenko2013SK} 
and is proved analogously in Appendix~\ref{sec:proof_overlap_concentration}.

\section{Proof of Theorem~\ref{th:rs_formula}} \label{sec:proof_rs}

The proof of Theorem~\ref{th:rs_formula} is divided in four steps. 
In Section~\ref{sec:overlap_concentration_1} we apply the Theorem~\ref{th:overlap_concentration} above to our matrix estimation problem to show that the overlaps concentrates around their expectations.
In Section~\ref{sec:tap}, we show that the overlaps satisfy asymptotically some fixed point equations. In Section~\ref{sec:interpolation}, we prove a lower bound for the limit of $F_n$. In Section~\ref{sec:aizenman}, we use similar arguments as in~\cite{lelarge2016fundamental} to obtain an upper bound on the limit, which will be revealed to be tight in Section~\ref{sec:final_part}.
\\

We will only prove the first equality in Theorem~\ref{th:rs_formula}, since the second follows from the ``sup-inf'' formula Proposition~\ref{prop:min_max} in Appendix~\ref{sec:min_max}.
In order to simplify the proof we are going to prove Theorem~\ref{th:rs_formula} in the case where $P_U$ and $P_V$ have finite (and thus bounded) support $S \subset [-K,K]$. The general case can be deduced from this case by approximating $P_U$ and $P_V$ by mixtures of Diracs as in~\cite{lelarge2016fundamental}, Section~6.2.2.
Since the dependency in $\lambda$ can be incorporated in the vector $\bbf{U}$ (and therefore in the prior $P_U$), we can restrict ourselves to the case $\lambda = 1$.
For simplicity, we are going to consider the case where $n=m$, i.e.\ $\alpha = 1$. The proof for general $\alpha$ can be directly deduced from the proof for $\alpha = 1$. Indeed, assume that $\alpha \in (0,1)$ (the case $\alpha >1$ follows simply by symmetry). Let $B_1, \dots, B_n \iid \Ber(\alpha)$, independently of everything else. Since $\frac{1}{n} \sum B_i$ concentrates tightly around $\alpha$, is is easy to show that the free energy $F_n$ is equal (up to a vanishing term) to the free energy of the observation channel
$$
Y_{i,j} = B_i \Big( \frac{1}{\sqrt{n}} U_i V_j + Z_{i,j} \Big) \quad \text{for} \ 1 \leq i,j \leq n \,,
$$
which is
$$
\widetilde{F}_n = \frac{1}{n} \E \log \int dP_{U}^{\otimes n}(\bbf{u}) dP^{\otimes m}_V(\bbf{v})
\exp \Big(\sum_{i,j=1}^n B_i \big( n^{-1/2} Z_{i,j} u_i v_j +\frac{1}{n} u_i v_j U_i V_j - \frac{1}{2n}u_i^2 v_j^2 \big)\Big) \,.
$$
Therefore, the case $\alpha < 1$ will only add some Bernoulli random variables $B_i$ in the proof for $\alpha = 1$ without changing the arguments. 
\\

For reasons mentioned above, we will suppose in this section to be in the case $\alpha = 1$ and $\lambda =1$, and remove all dependencies in this variables: we will simply write $\Gamma$ instead of $\Gamma(\lambda,\alpha)$ and $\cF(q_u,q_v)$ instead of $\cF(\lambda,\alpha,q_u,q_v)$.
We will use the notation $P_0 = P_U \otimes P_V$.

\subsection{Overlap concentration}\label{sec:overlap_concentration_1}

In this section we apply the results of the previous section to our model~\eqref{eq:model}.
We will need to consider an inference model that is slightly more general than~\eqref{eq:model}. Let $(U_i,V_i)_{1 \leq i \leq n} \iid P_0$ and suppose that we observe
\begin{align}
	Y_{i,j} &= \sqrt{\frac{t}{n}} U_i V_j + Z_{i,j}, &\text{for} \  1 \leq i,j \leq n\,, \label{eq:channel_m}\\
	Y^{(u)}_{i} &= \sqrt{q_u} U_i + Z^{(u)}_i, \qquad 
	Y^{(v)}_{i} = \sqrt{q_v} V_i + Z^{(v)}_i, &\text{for} \  1 \leq i \leq n\,, \label{eq:channel_s}
	\\
	Y^{(u)\prime}_{i} &= a_u\sqrt{s_n} U_i + z^{(u)}_i, \qquad 
	Y^{(v)\prime}_{i} = a_v \sqrt{s_n} V_i + z^{(v)}_i, &\text{for} \  1 \leq i \leq n\,, \label{eq:perturbation}
\end{align}
where $Z^{(u)}_i, Z^{(v)}_i,z^{(u)}_i,z^{(v)}_i \iid \mathcal{N}(0,1)$ are independent of everything else, $t \in [0,1]$, $q_u,q_v \geq 0$, $a_u,a_v \in [0,1]$ and $(s_n) \in (0,1)^{\N}$.
The observations~\eqref{eq:channel_m} corresponds to the original matrix estimation problem.
One can associate to these observations the Hamiltonian $H_n$:
\begin{equation} \label{eq:def_h_n} \\
	H_n(\bbf{u},\bbf{v}) = \sum_{1 \leq i,j \leq n} \sqrt{\frac{t}{n}} u_i v_j Z_{i,j} + \frac{t}{n} u_i U_i v_j V_j - \frac{t}{2n} u_i^2 v_j^2, \quad \text{for } \bbf{u},\bbf{v} \in S^n \,.
\end{equation}
Similarly, one can associate to the observations~\eqref{eq:channel_s} the Hamiltonian:
\begin{equation} \label{eq:def_h_s}
	H_n^{(s)}(\bbf{u},\bbf{v}) = \sum_{i=1}^n
	\sqrt{q_u} u_i Z^{(u)}_{i} +  q_u u_i U_i - \frac{q_u}{2} u_i^2
	+ \sqrt{q_v} v_i Z^{(v)}_{i} + q_v v_i V_i - \frac{q_v}{2} v_i^2 \,.
\end{equation}
The observations~\eqref{eq:perturbation} correspond to a small amount of side-information that will allow us to prove some concentration result for the overlaps as in Section~\ref{sec:overlap_concentration}. The corresponding Hamiltonians read
\begin{align*}
	H_{n,u}^{\text{(pert)}}(\bbf{u}) &= \sum_{i=1}^n \sqrt{s_n}a_u z^{(u)}_i u_i + s_n a_u^2 u_i U_i - \frac{s_n a_u^2}{2} u_i^2 , \quad \text{for } \bbf{u} \in S^n \,,\\
	H_{n,v}^{\text{(pert)}}(\bbf{v}) &= \sum_{i=1}^n \sqrt{s_n}a_v z^{(v)}_i v_i + s_n a_v^2 v_i V_i - \frac{s_n a_v^2}{2} v_i^2 , \quad \text{for } \bbf{v} \in S^n.
\end{align*}
We write, for $\bbf{u},\bbf{v} \in S^n$, $H_n^{\text{(pert)}}(\bbf{u},\bbf{v}) = H_{n,u}^{\text{(pert)}}(\bbf{u}) + H_{n,v}^{\text{(pert)}}(\bbf{v})$ and define the ``total'' Hamiltonian as $H_n^{\text{(tot)}} = H_n + H_n^{(s)} + H_n^{\text{(pert)}}$.
The posterior distribution of $(\bbf{U},\bbf{V})$ given $\mathsf{Y}=(\bbf{Y},\bbf{Y}^{(u)},\bbf{Y}^{(v)},\bbf{Y}^{(u)\prime},\bbf{Y}^{(v)\prime})$ reads
\begin{equation}
	P\big((\bbf{u},\bbf{v}) \, \big| \, \mathsf{Y}\big) = \frac{1}{\cZ_n^{\text{(tot)}}} P_0^{\otimes n}(\bbf{u},\bbf{v}) e^{H_n^{\text{(tot)}}(\bbf{u},\bbf{v})}\,,
\end{equation}
where $\cZ_n^{\text{(tot)}}$ is the appropriate normalization.
Let $\langle \cdot \rangle_{n,a}$ be the associated Gibbs measure on $(S^n)^2$:
\begin{equation}\label{eq:def_gibbs_proof}
\big\langle f(\bbf{u},\bbf{v}) \big\rangle_{n,a} = \frac{ \sum_{\bbf{u},\bbf{v} \in S^n} P^{\otimes n}_0(\bbf{u},\bbf{v}) f(\bbf{u},\bbf{v}) e^{H_n^{\text{(tot)}}(\bbf{u},\bbf{v})}}{ \sum_{\bbf{u},\bbf{v} \in S^n} P_0^{\otimes}(\bbf{u},\bbf{v}) e^{H_n^{\text{(tot)}}(\bbf{u},\bbf{v})}}
, \quad \text{for any function } f \text{ on } (S^n)^2.
\end{equation}

An application of the decorrelation principle of Section~\ref{sec:overlap_concentration} (see Appendix~\ref{sec:proof_overlap_uv} for a proof) gives
\begin{proposition} \label{prop:overlap_uv}
	Define $s_n = n^{-1/4}$, then
	\begin{align}
		&\int_1^2 \int_1^2 \E \Big\langle \big(\bbf{u}^{(1)}.\bbf{u}^{(2)} - \E\langle \bbf{u}^{(1)}.\bbf{u}^{(2)} \rangle_{n,a}\big)^2 \Big\rangle_{n,a} da_u da_v \xrightarrow[n \to \infty]{} 0\,, \label{eq:overlap_u} \\
		&\int_1^2 \int_1^2 \E \Big\langle \big(\bbf{v}^{(1)}.\bbf{v}^{(2)} - \E\langle \bbf{v}^{(1)}.\bbf{v}^{(2)} \rangle_{n,a}\big)^2 \Big\rangle_{n,a} da_u da_v \xrightarrow[n \to \infty]{} 0\,. \label{eq:overlap_v}
	\end{align}
\end{proposition}

In the following, $s_n$ will be equal to $n^{-1/4}$. It will also be convenient to consider $a_u$ and $a_v$ as random variables.
Suppose that $(a_u,a_v) \sim \mathcal{U}([1,2]^2)$ and denote $\E_{a}$ the expectation with respect to $(a_u,a_v)$. We can then rewrite the result of Proposition~\ref{prop:overlap_uv} as
\begin{align}
	&\E_a \E \Big\langle (\bbf{u}^{(1)}.\bbf{u}^{(2)} - \E\langle \bbf{u}^{(1)}.\bbf{u}^{(2)} \rangle_{n,a})^2 \Big\rangle_{n,a} \xrightarrow[n \to \infty]{} 0 \,,\label{eq:overlap_u2} \\
	&\E_a \E \Big\langle \big(\bbf{v}^{(1)}.\bbf{v}^{(2)} - \E\langle \bbf{v}^{(1)}.\bbf{v}^{(2)} \rangle_{n,a}\big)^2 \Big\rangle_{n,a} \xrightarrow[n \to \infty]{} 0 \,.\label{eq:overlap_v2}
\end{align}

\subsection{Fixed point equations} \label{sec:tap}

We have seen (in Proposition~\ref{prop:overlap_uv}) that the overlaps $\bbf{u}^{(1)}.\bbf{u}^{(2)}$ and $\bbf{v}^{(1)}.\bbf{v}^{(2)}$ concentrates asymptotically around their expectations. In this section, we show that these expected values satisfy fixed point equations, in the $n \to \infty$ limit. The analysis is an adaptation of the derivation of the TAP equations for the SK model, see~\cite{talagrand2010meanfield1}.

To obtain these fixed point equations, we are going to do what physicists call ``cavity computations'': we compare the system with $2n$ variables to the system with $2n+2$ variables to study the influence of the ``first'' $2n$ variables on the $2$ ``last'' variables we add.
\\

Let $(\bbf{\tilde{u}},\bbf{\tilde{v}}) \in S^{n+1} \times S^{n+1}$ and decompose $\bbf{\tilde{u}} = (\bbf{u},u')$, $\bbf{\tilde{v}}=(\bbf{v},v')$ where $\bbf{u},\bbf{v} \in S^{n}$ and $u',v' \in S$. We will use the short notations $U' = U_{n+1}$ and $V' = V_{n+1}$. We decompose the Hamiltonian
\begin{align*}
	H_{n+1}(\bbf{\tilde{u}},\bbf{\tilde{v}}) &= H_n'(\bbf{u},\bbf{v}) + h_u(\bbf{v},u') + h_v(\bbf{u},v') + \delta(u',v') \,,
\end{align*}
where
\begin{align*}
	H_n'(\bbf{u},\bbf{v}) &= \sum_{1 \leq i,j \leq n} \frac{\sqrt{t}}{\sqrt{n+1}} u_i v_j Z_{i,j} + \frac{t}{n+1} u_i U_i v_j V_j - \frac{t u_i^2 v_j^2}{2(n+1)} \,,  \\
	h_u(\bbf{v},u') &=\sum_{j=1}^n  \frac{\sqrt{t}}{\sqrt{n+1}}u' v_j Z_{n+1,j} + u' U' \frac{t}{n+1} v_j V_j - \frac{t u'^2 v_j^2 }{2(n+1)} \,, \\
	h_v(\bbf{u},v') &=\sum_{i=1}^n  \frac{\sqrt{t}}{\sqrt{n+1}}v' u_i Z_{i,n+1} + v' V' \frac{t}{n+1} u_i U_i - \frac{t v'^2 u_i^2 }{2(n+1)} \,,\\
	\delta(u',v') &=  \frac{\sqrt{t} u' v'}{\sqrt{n+1}} Z_{n+1,n+1} + \frac{t}{n+1} u' U' v' V' - t \frac{u'^2 v'^2}{2(n+1)} \,.
\end{align*}
Similarly, one can decompose the Hamiltonians $H_n^{(s)}$ and $H_{n+1}^{\text{(pert)}}$
\begin{align*}
	H_{n+1}^{(s)}(\bbf{\tilde{u}},\bbf{\tilde{v}}) &= H_n^{(s)}(\bbf{u},\bbf{v}) + g_u(u') + g_v(v') \,, \\
	H_{n+1}^{\text{(pert)}}(\bbf{\tilde{u}},\bbf{\tilde{v}}) &= H_n^{\text{(pert)} \prime}(\bbf{u},\bbf{v}) + \sqrt{s_{n+1}}a_u z^{(u)}_{n+1} u' + s_{n+1} a_u^2 u' U' - \frac{s_{n+1} a_u^2}{2} u^{\prime 2}  \\
																																												  & \qquad \qquad \qquad \ \ +  \sqrt{s_{n+1}}a_v z^{(v)}_{n+1} v' + s_{n+1} a_v^2 v' V' - \frac{s_{n+1} a_v^2}{2} v^{\prime 2} \,,
\end{align*}
where 
\begin{align*}
	H_n^{\text{(pert)} \prime}(\bbf{u},\bbf{v}) &= \sum_{i=1}^n \sqrt{s_{n+1}}a_u z^{(u)}_i u_i + s_{n+1} a_u^2 u_i U_i - \frac{s_{n+1} a_u^2}{2} u_i^2 
	+ \sum_{i=1}^n \sqrt{s_{n+1}}a_v z^{(v)}_i v_i + s_{n+1} a_v^2 v_i V_i - \frac{s_{n+1} a_v^2}{2} v_i^2 \,, \\
	g_u(u') &= \sqrt{ q_u} u' Z^{(u)}_{n+1} +   q_u u' U' - \frac{ q_u}{2} u'^2 \,, \\
	g_v(v') &= \sqrt{ q_v} v' Z^{(v)}_{n+1} +   q_v v' V' - \frac{ q_v}{2} v'^2  \,.
\end{align*}
Let us now define $H_n^{\text{(tot)}\prime} =H_n' +H_n^{(s)}+ H_n^{\text{(pert)} \prime}$ and $\langle \cdot \rangle'_{n,a}$ the Gibbs measure on $(S^n)^2$ corresponding to the Hamiltonian $H_n^{\text{(tot)} \prime}$. 
An easy adaptation of Proposition~\ref{prop:overlap_uv} gives that the overlaps under the Gibbs measure $\langle \cdot \rangle'_{n,a}$ concentrate around their expectations:
\begin{equation}\label{eq:concentration_prime}
	\E_a \E \Big\langle \big(\bbf{u}^{(1)}.\bbf{u}^{(2)} - \E\langle \bbf{u}^{(1)}.\bbf{u}^{(2)} \rangle'_{n,a}\big)^2 \Big\rangle'_{n,a} 
	+ \ \E_a \E \Big\langle \big(\bbf{v}^{(1)}.\bbf{v}^{(2)} - \E\langle \bbf{v}^{(1)}.\bbf{v}^{(2)} \rangle'_{n,a}\big)^2 \Big\rangle'_{n,a} \xrightarrow[n \to \infty]{} 0 \,.
\end{equation}
Define
\begin{align}
	y(u',v',\bbf{u},\bbf{v}) &= H_{n+1}^{\text{(tot)}}\big((\bbf{u},u'),(\bbf{v},v')\big) - H_n^{\text{(tot)}\prime}(\bbf{u},\bbf{v}) \nonumber
	\\
	&= h_u(\bbf{v},u') + h_v(\bbf{u},v') + \delta(u',v') +g_u(u') + g_v(v')\label{eq:def_y}
	\\
	&+ \sqrt{s_{n+1}}a_u z^{(u)}_{n+1} u' + s_{n+1} a_u^2 u' U' - \frac{s_{n+1} a_u^2}{2} u^{\prime 2} 
	+ \sqrt{s_{n+1}}a_v z^{(v)}_{n+1} v' + s_{n+1} a_v^2 v' V' - \frac{s_{n+1} a_v^2}{2} v^{\prime 2} \,. \nonumber
\end{align}
Define the random variables
\begin{align}
	Q_u' &= \Big\langle \frac{1}{n} \sum_{i=1}^n u^{(1)}_i u^{(2)}_i \Big\rangle'_{n,a}  = \frac{1}{n} \sum_{i=1}^n \langle u_i \rangle_{n,a}^{\prime \, 2} \,, \label{eq:def_q_u} \\
	Q_v' &= \Big\langle \frac{1}{n} \sum_{i=1}^n v^{(1)}_i v^{(2)}_i \Big\rangle'_{n,a} = \frac{1}{n} \sum_{i=1}^n \langle v_i \rangle_{n,a}^{\prime \, 2} \,, \label{eq:def_q_v} 
\end{align}
where $(\bbf{u}^{(1)},\bbf{v}^{(1)})$ and $(\bbf{u}^{(2)},\bbf{v}^{(2)})$ are two independent replicas sampled from $\langle \cdot \rangle'_{n,a}$.
Let $\phi: S^4 \to \R$ and define
\begin{align*}
	A &= \Big\langle \sum_{u',v' \in S} P_0(u',v') \phi(u',v',U',V') e^{y(u',v',\bbf{u},\bbf{v})} \Big\rangle'_{n,a}, \text{ where } (\bbf{u},\bbf{v}) \text{ is a replica from } \langle \cdot \rangle'_{n,a}, \\
	B &= \sum_{u',v' \in S} P_0(u',v') \phi(u',v',U',V') \exp(h_0(u',v')) \,,
\end{align*}
(recall the short notation $U'=U_{n+1}$ and $V'=V_{n+1}$) where 
\begin{align*}
	h_0(u',v') &= u' \Big(\sum_{j=1}^n \frac{\sqrt{t}Z_{n+1,j}}{\sqrt{n+1}} \langle v_j \rangle'_{n,a}\Big) + u' U' \frac{t n}{n+1} Q_v' - u'^2 \frac{tn}{2(n+1)} Q_v' \\
			   &+ v' \Big( \sum_{i=1}^n \frac{\sqrt{t}Z_{i,n+1} }{\sqrt{n+1}} \langle u_i \rangle'_{n,a} \Big) + v' V' \frac{t n }{n+1} Q_u' - v'^2 \frac{t n }{2(n+1)} Q_u' 
	 +g_u(u') + g_v(v') \,.
\end{align*}
Recall that $\E_a$ denotes the expectation with respect to the perturbation $a_u, a_v \iid \mathcal{U}([0,1])$.
\begin{lemma} \label{lem:l2}
	\begin{align*}
		\E_a \E ( A - B )^2 \xrightarrow[n \to \infty]{} 0 \,.
	\end{align*}
\end{lemma}

\begin{proof}
	It suffices to prove $\E_a (\E A^2 - \E B^2) \xrightarrow[n \to \infty]{} 0$ and $\E_a (\E AB - \E B^2) \xrightarrow[n \to \infty]{} 0$.
	The proof follows exactly the same steps than Lemma~28 from~\cite{lelarge2016fundamental}, so we omit it for the sake of brevity.
\end{proof}
\\

Let $\langle \cdot \rangle_{n+1,a}$ be the Gibbs measure on $(S^{n+1})^2$ associated with the Hamiltonian $H_{n+1}^{\text{(tot)}} = H_{n+1} + H_{n+1}^{(s)} + H_{n+1}^{\text{(pert)}}$ as defined by~\eqref{eq:def_gibbs_proof}.
\begin{lemma} \label{lem:tap}
	\begin{align*}
		\E_a \E \left| \big\langle \phi(u_{n+1},v_{n+1},U_{n+1},V_{n+1}) \big\rangle_{n+1,a} - \frac{\sum_{u',v' \in S} P_0(u',v') \phi(u',v',U',V') \exp(h_0(u',v'))}{\sum_{u',v' \in S} P_0(u',v') \exp(h_0(u',v'))} \right| \xrightarrow[n \to \infty]{} 0\,.
	\end{align*}
\end{lemma}

\begin{proof}
	By the definition of $y$ (see Equation~\eqref{eq:def_y}) we have
	\begin{equation} \label{eq:cavity_p}
		\Big\langle \phi(u_{n+1},v_{n+1},U_{n+1},V_{n+1}) \Big\rangle_{n+1,a} = \frac{\left\langle \sum_{u',v' \in S}P_0(u',v')  \phi(u',v',U',V') \exp(y(u',v',\bbf{u},\bbf{v})) \right\rangle'_{n,a}}{\left\langle \sum_{u',v' \in S} P_0(u',v') \exp(y(u',v',\bbf{u},\bbf{v})) \right\rangle'_{n,a}} \,.
	\end{equation}
	Define
	\begin{align*}
		R &=\sum_{u',v' \in S} P_0(u',v') \phi(u',v',U',V') \exp(h_0(u',v')) \ , \
		S =\sum_{u',v' \in S} P_0(u',v') \exp(h_0(u',v')) \,, \\
		R'&=\Big\langle \sum_{u',v' \in S} P_0(u',v') \phi(u',v',U',V') \exp(y(u',v',\bbf{u},\bbf{v})) \Big\rangle'_{n,a} \ , \
		S'=\Big\langle \sum_{u',v' \in S} P_0(u',v') \exp(y(u',v',\bbf{u},\bbf{v})) \Big\rangle'_{n,a} \,.
	\end{align*}
	We have to prove that $\E_a \E \big| \frac{R'}{S'} - \frac{R}{S} \big| \xrightarrow[n \to \infty]{} 0$.
	By equation~\eqref{eq:cavity_p} $|\frac{R'}{S'}| \leq \|\phi\|_{\infty}$, therefore
	\begin{align*}
		\Big| \frac{R'}{S'} - \frac{R}{S} \Big| = \Big| \frac{R'}{S'} \frac{S-S'}{S} + \frac{R'-R}{S} \Big| \leq \frac{1}{|S|} (\|\phi\|_{\infty} +1 ) ( |R-R'| + |S-S'| ) \,.
	\end{align*}
	Using the Cauchy-Schwarz inequality,
	\begin{align*}
		\E \Big| \frac{R'}{S'} - \frac{R}{S} \Big| \leq (\|\phi\|_{\infty} +1) \sqrt{\E S^{-2}} \sqrt{\E (R-R')^2 + \E (S-S')^2 }.
	\end{align*}
	By Jensen's inequality, one have
	\begin{align*}
		\E S^{-2} \leq \E \left[ \sum_{u',v' \in S} P_0(u',v') \exp(-2 h_0(u',v')) \right]
		= \sum_{u',v' \in S} P_0(u',v') \E \left[ \exp(-2 h_0(u',v')) \right] \leq e^{10  K^4 + 3  (q_u +q_v) K^2} \,.
	\end{align*}
	We apply Lemma~\ref{lem:l2} twice (with $\phi = 1$ and ``$\phi=\phi$'') to obtain
	$
	\E_a \sqrt{\E (R-R')^2 + \E (S-S')^2 } \xrightarrow[n \to \infty]{} 0
	$ which concludes the proof.
\end{proof}

\begin{lemma} \label{lem:gibbs_cavity}
	$$
	\E_a \E | \langle u_1 \rangle_{n+1,a}  - \langle u_1 \rangle'_{n,a} | \xrightarrow[n \to \infty]{} 0 \,.
	$$
\end{lemma}
\begin{proof}
	By the definition of $y$ (see equation~\eqref{eq:def_y}) we have
	\begin{equation} \label{eq:u_p}
		\frac{\big\langle (u_1 - \langle u_1 \rangle'_{n,a}) \sum_{u',v' \in S} P_0(u',v') e^{y(u',v',\bbf{u},\bbf{v})} \big\rangle'_{n,a}}{\big\langle \sum_{u',v' \in S} P_0(u',v') e^{y(u',v',\bbf{u},\bbf{v})} \big\rangle'_{n,a}} = \langle u_1 \rangle_{n+1,a} - \langle u_1 \rangle'_{n,a}\,.
	\end{equation}
	We denote by $\E'$ the expectation with respect to the variables $(Z_{i,n+1})_{1 \leq i \leq n+1}$, $(Z_{n+1,j})_{1 \leq j \leq n+1}$, $Z_{n+1}^{(u)}$, $Z_{n+1}^{(v)}$, $z_{n+1}^{(u)}$ and $z_{n+1}^{(v)}$.
	We first notice that, using Jensen's inequality,
	\begin{align*}
		\E \left[ \left(\left\langle \sum_{u',v' \in S} P_0(u',v') e^{y(u',v',\bbf{u},\bbf{v})} \right\rangle'_{n,a}\right)^{-2} \right]&\leq 
		\E\left\langle \sum_{u',v' \in S} P_0(u',v') e^{-2 y(u',v',\bbf{u},\bbf{v})} \right\rangle'_{n,a}
		\\
		&\leq 
		\E\left\langle \sum_{u',v' \in S} P_0(u',v') \E' \left[e^{-2 y(u',v',\bbf{u},\bbf{v})}\right] \right\rangle'_{n,a} \,.
	\end{align*}
	The bounded support assumption on $P_0$ gives then that, for all $(u',v')\in S^2$ and $\bbf{u},\bbf{v} \in S^n$, $\E' \left[e^{-2 y(u',v',\bbf{u},\bbf{v})}\right] \leq C_1$ for some constant $C_1$. Therefore $\E \left[ \Big(\big\langle \sum_{u',v' \in S} P_0(u',v') e^{y(u',v',\bbf{u},\bbf{v})} \big\rangle'_{n,a}\Big)^{-2} \right] \leq C_1$. The Cauchy-Schwarz inequality applied to the left hand side of equation~\eqref{eq:u_p} shows that it suffices to prove
	$$
	\E_a \E \left[ \left(\left\langle (u_1 - \langle u_1 \rangle'_{n,a}) \sum_{u',v' \in S} P_0(u',v') e^{y(u',v',\bbf{u},\bbf{v})} \right\rangle'_{\!\!n,a}\right)^2 \right] \xrightarrow[n \to \infty]{} 0 
	$$
	to obtain the lemma. Compute
	{\small
		\begin{align*}
			&\E \Big(\Big\langle (u_1 - \langle u_1 \rangle'_{n,a}) \sum_{u',v' \in S} P_0(u',v') e^{y(u',v',\bbf{u},\bbf{v})} \Big\rangle'_{\!\!n,a}\Big)^2 
			\\
			&=\E \left\langle \!\!(u_1^{(1)} - \langle u_1^{(1)} \rangle'_{n,a})(u_1^{(2)} - \langle u_1^{(2)} \rangle'_{n,a}) \!\!\!\! \sum_{u_1',v_1',u_2',v_2'\in S} \!\!\!\!  P_0(u_1',v_1') P_0(u_2',v_2') \exp(y(u_1',v_1',\bbf{u}^{(1)},\bbf{v}^{(1)})+ y(u_2',v_2',\bbf{u}^{(2)},\bbf{v}^{(2)})) \right\rangle'_{\!\!n,a}\\
			&= \E \left\langle \!\frac{1}{n} \sum_{i=1}^n (u^{(1)}_i - \langle u^{(1)}_i \rangle'_{n,a}) (u^{(2)}_i - \langle u^{(2)}_i \rangle'_{n,a}) \!\!\!\! \sum_{u_1',v_1',u_2',v_2'}  \!\!\!\! P_0(u_1',v_1') P_0(u_2',v_2') \exp(y(u_1',v_1',\bbf{u}^{(1)},\bbf{v}^{(1)})+ y(u_2',v_2',\bbf{u}^{(2)},\bbf{v}^{(2)})) \!\right\rangle'_{\!\!n,a}\\
			&\leq 
			\left( \E \left\langle \left( \frac{1}{n} \sum_{i=1}^n (u^{(1)}_i - \langle u^{(1)}_i \rangle'_{n,a}) (u^{(2)}_i - \langle u^{(2)}_i \rangle'_{n,a}) \right)^2 \right\rangle'_{\!\!n,a} \right)^{1/2}
			\\
			&\left( \E \left\langle \sum_{u_1',v_1',u_2',v_2' \in S}  P_0(u_1',v_1') P_0(u_2',v_2') \exp(2y(u_1',v_1',\bbf{u}^{(1)},\bbf{v}^{(1)})+ 2 y(u_2',v_2',\bbf{u}^{(2)},\bbf{v}^{(2)})) \right\rangle'_{\!\!n,a}  \right)^{1/2}
		\end{align*}
	}
	by Cauchy-Schwarz's inequality. The bounded support assumption on $P_0$ implies that, there exists a constant $C_2$ such that, for all $u_1',v_1',u_2',v_2' \in S$ and $\bbf{u}^{(1)},\bbf{v}^{(1)},\bbf{u}^{(2)},\bbf{v}^{(2)} \in S^n$ we have
	\begin{align*}
		\E' \left[\exp\big( 2y(u_1',u_2',\bbf{u}^{(1)},\bbf{v}^{(1)}) +  2y(u_2',v_2',\bbf{u}^{(2)},\bbf{v}^{(2)})\big)\right] \leq C_2 \,.
	\end{align*}
	Thus
	{\small
		$$
		\E_a \E \Big(\Big\langle (u_1 - \langle u_1 \rangle'_{n,a}) \!\!\! \sum_{u',v' \in S} P_0(u',v') e^{y(u',v',\bbf{u},\bbf{v})} \Big\rangle'_{\!\!n,a}\Big)^2 
		\leq C_2^{1/2} 
		\left( \E_a \E \left\langle \left( \frac{1}{n} \sum_{i=1}^n (u^{(1)}_i - \langle u^{(1)}_i \rangle'_{n,a}) (u^{(2)}_i - \langle u^{(2)}_i \rangle'_{n,a}) \right)^2 \right\rangle'_{\!\!n,a} \right)^{1/2}.
		$$
	}
	And the right hand side goes to $0$ as $n \to \infty$ by~\eqref{eq:concentration_prime}.
	This concludes the proof.
\end{proof}

\begin{corollary} \label{cor:overlap_p}
	\begin{align}
		&\E_a \Big| \E \Big\langle \frac{1}{n+1} \sum_{i=1}^{n+1} u^{(1)}_i u^{(2)}_i \Big\rangle_{n+1,a} - \E[Q_u'] \Big| \xrightarrow[n \to \infty]{} 0 \,, \label{eq:overlap_u_p}
		\\
		&\E_a \Big| \E \Big\langle \frac{1}{n+1} \sum_{i=1}^{n+1} v^{(1)}_i v^{(2)}_i \Big\rangle_{n+1,a} - \E[Q_v']  \Big| \xrightarrow[n \to \infty]{} 0 \,. \label{eq:overlap_v_p}
	\end{align}
\end{corollary}

\begin{proof}
	We only need to prove~\eqref{eq:overlap_u_p},~\eqref{eq:overlap_v_p} is then obtained by symmetry.
	By the preceding lemma $\E_a \E | \langle u_1 \rangle_{n+1,a}^2 - \langle u_1 \rangle_{n,a}^{\prime \, 2} | \to 0$. Thus
	$$
	\E_a \Big| \E \left[ \frac{1}{n} \sum_{i=1}^{n} \langle u_i \rangle_{n+1,a}^2 \right] - \E \left[  \frac{1}{n} \sum_{i=1}^{n} \langle u_i \rangle_{n,a}^{\prime \, 2} \right]  \Big| \xrightarrow[n \to \infty]{} 0.
	$$
	The variables $u_i$ are bounded, so $\big|\frac{1}{n+1} \sum_{i=1}^{n+1} \langle u_i \rangle_{n+1,a}^2 - \frac{1}{n} \sum_{i=1}^{n} \langle u_i \rangle_{n+1,a}^2\big| = O(n^{-1})$, hence the result.
\end{proof}
\\

Let $Z_1,Z_2 \iid \cN(0,1)$ and define for $\gamma_1, \gamma_2 \geq 0$
\begin{equation*}
	F_{\phi}(\gamma_1, \gamma_2) \! = 
	\!\E \!\!\left[ \frac{\sum\limits_{u',v' \in S} P_0(u',v') \phi(u',v',U',V') \exp(\sqrt{\gamma_1} Z_1 u' + \gamma_1 u'U' - \frac{1}{2}\gamma_1 u'^2 + \sqrt{\gamma_2} Z_2 v' + \gamma_2 v'V' - \frac{1}{2} \gamma_2 v'^2)}
	{\sum\limits_{u',v' \in S}P_0(u',v') \exp(\sqrt{\gamma_1} Z_1 u' + \gamma_1 u'U' - \frac{1}{2}\gamma_1 u'^2 + \sqrt{\gamma_2} Z_2 v' + \gamma_2 v'V' - \frac{1}{2} \gamma_2 v'^2)} \right]
\end{equation*}

\begin{proposition} \label{prop:tap_F}
	$$
	\E_a \left| \E \big\langle \phi(u_1,v_1,U_1,V_1) \big\rangle_{n+1,a} - F_{\phi}\left(
		 t \E \big\langle \bbf{v}^{(1)}. \bbf{v}^{(2)} \big\rangle_{n+1,a}  +  q_u
		,
		 t \E \big\langle \bbf{u}^{(1)}. \bbf{u}^{(2)} \big\rangle_{n+1,a} +  q_v
	\right) \right| \xrightarrow[n \to \infty]{} 0 \,.
	$$
\end{proposition}

\begin{proof}
	Define for $u',v' \in S$
	\begin{align*}
		h_1(u',v') &= \sqrt{ (\frac{n}{n+1} tQ_v' + q_u)} Z_1 u' +  (\frac{n}{n+1} tQ_v' + q_u) u'U' - \frac{1}{2} (\frac{n}{n+1} tQ_v' + q_u) u'^2 \\
							   &+ \sqrt{ (\frac{n}{n+1} tQ_u' + q_v)} Z_2 v' +  (\frac{n}{n+1} tQ_u' + q_v) v'V' - \frac{1}{2} (\frac{n}{n+1} tQ_u' + q_v) v'^2 \,.
	\end{align*}
	Notice that $(U',V',(h_0(u',v'))_{u',v' \in S}) = (U',V',(h_1(u',v'))_{u',v' \in S})$ in law. Indeed, conditionally to $(U',V')$ and $(\langle u_i \rangle'_{n,a}, \langle v_i \rangle'_{n,a})_{1 \leq i \leq n}$, $(h_0(u',v'))_{u',v' \in S}$ and $(h_1(u',v'))_{u',v' \in S}$ are two Gaussian processes with the same covariance structure.
	Consequently, 
	\begin{align*}
		\E \left[ \frac{\sum_{u',v' \in S} P_0(u',v') \phi(u',v',U',V') \exp(h_0(u',v'))}{\sum_{u',v' \in S} P_0(u',v') \exp(h_0(u',v'))} \right]
		&=
		\E \left[ \frac{\sum_{u',v' \in S} P_0(u',v') \phi(u',v',U',V') \exp(h_1(u',v'))}{\sum_{u',v' \in S} P_0(u',v') \exp(h_1(u',v'))} \right]
		\\
		&=
		\E  \left[ F_{\phi}\left( \frac{n}{n+1} tQ_v' +  q_u,  \frac{n}{n+1} tQ_u' +  q_v \right) \right] \,.
	\end{align*}
	Thus, using Lemma~\ref{lem:tap} we obtain
	$$
	\E_a \left| 
	\E \big\langle \phi(u_1,v_1,U_1,V_1) \big\rangle_{n+1,a} -
	\E \left[ F_{\phi} \left(  \frac{n}{n+1} tQ_v' +  q_u,  \frac{n}{n+1} tQ_u' +  q_v \right) \right]
	\right| \xrightarrow[n \to \infty]{} 0.
	$$
	The function $F_{\phi}$ is $\mathcal{C}^1$ and therefore Lipschitz on the compact set $C = [ q_u - K^2 ,  q_u + K^2 ] \times [ q_v - K^2 ,  q_v + K^2 ]$. We note $L_0$ its Lipschitz constant. $( \frac{n}{n+1} tQ_v' +  q_u ,  \frac{n}{n+1} tQ_u' +  q_v)$ belongs to $C$ with probability $1$, therefore
	\begin{align*}
		&\left| \E \left[ F_{\phi} \left( \frac{n}{n+1} tQ_v' +  q_u,  \frac{n}{n+1} tQ_u' +  q_v \right)\right]
		- F_{\phi} \left( t \E \big\langle \bbf{v}^{(1)}.\bbf{v}^{(2)} \big\rangle_{n+1,a} +  q_u,  t \E \big\langle \bbf{u}^{(1)}.\bbf{u}^{(2)} \big\rangle_{n+1,a} +  q_v \right) \right|
		\\
		& \qquad \qquad \leq L_0  \E \left(
		\left(\frac{n}{n+1} Q_u' - \E \big\langle \bbf{u}^{(1)}. \bbf{u}^{(2)} \big\rangle_{n+1,a}\right)^2
		+\left(\frac{n}{n+1} Q_v' - \E \big\langle \bbf{v}^{(1)}. \bbf{v}^{(2)} \big\rangle_{n+1,a}\right)^2
	\right)^{1/2}.
\end{align*}
The expectation of the right hand side with respect to $a_u$ and $a_v$ goes to zero as $n \to \infty$ because the overlaps under $\langle \cdot \rangle'_{n,a}$ concentrate around their expectations (see Equation~\ref{eq:concentration_prime}), and because of Corollary~\ref{cor:overlap_p}.
This concludes the proof.
\end{proof}
\\

We remark that the function $F_{P_U}$ defined as in~\eqref{eq:def_fx} corresponds to $F_{\phi}$ obtained for the choice $\phi(u_1,v_1,u_2,v_2) = u_1 u_2$.
Similarly, $F_{P_V}$ (defined as in~\eqref{eq:def_fx}) is the function $F_{\phi}$ obtained for $\phi(u_1,v_1,u_2,v_2) = v_1 v_2$. Proposition~\ref{prop:tap_F} implies then that the overlaps satisfy asymptotically two fixed point equations.

\begin{corollary} \label{cor:fixed_point}
	\begin{align*}
		&\E_a \Big| \E \langle \bbf{u}^{(1)}.\bbf{u}^{(2)} \rangle_{n,a} - F_{P_U}(  t \E \langle \bbf{v}^{(1)}.\bbf{v}^{(2)} \rangle_{n,a} +  q_u) \Big| \xrightarrow[n \to \infty]{} 0 \,,\\
		&\E_a \Big| \E \langle \bbf{v}^{(1)}.\bbf{v}^{(2)} \rangle_{n,a} - F_{P_V}(  t \E \langle \bbf{u}^{(1)}.\bbf{u}^{(2)} \rangle_{n,a} +  q_v) \Big| \xrightarrow[n \to \infty]{} 0 \,.
	\end{align*}
\end{corollary}

\subsection{The lower bound: interpolation method} \label{sec:interpolation}
The lower bound is proved using Guerra's interpolation technique~\cite{guerra2003broken}, originally developed for the SK model. In the context of bipartite spin glasses, this interpolation scheme has been used in~\cite{barra2010replica} under a ``replica symmetric'' assumption.

\begin{proposition} \label{prop:interpolation}
	$$
	\liminf_{n \to \infty} F_n \geq \sup_{(q_1,q_2) \in \Gamma}\mathcal{F}(q_1,q_2) \,.
	$$
\end{proposition}
\begin{proof}
	Let $(q_1,q_2) \in \Gamma$.
	Define, for $t \in [0,1]$, the Hamiltonians
	\begin{align*}
		H_{n,t}(\bbf{u},\bbf{v}) &= \sum_{1 \leq i,j \leq n} \sqrt{\frac{ t}{n}} u_i v_j Z_{i,j} + \frac{ t}{n} u_i U_i v_j V_j - \frac{ t}{2n} u_i^2 v_j^2 \,, \\
		H_{n,t}^{(s)}(\bbf{u},\bbf{v}) &= \sum_{i=1}^n \sqrt{ (1-t) q_2} Z^{(u)}_i u_i +  (1-t) q_2 u_i U_i - \frac{ (1-t) q_2}{2} u_i^2 \\
																				 &+ \sum_{i=1}^n \sqrt{ (1-t)q_1} Z^{(v)}_i v_i +  (1-t) q_1 v_i V_i - \frac{ (1-t) q_1}{2} v_i^2 \,,
	\end{align*}
	for $\bbf{u},\bbf{v} \in S^n$, and $H_{n,t}^{\text{(tot)}} = H_{n,t} + H_{n,t}^{(s)} + H_n^{\text{(pert)}}$, where $H_n^{\text{(pert)}}$ is defined  in Section~\ref{sec:overlap_concentration_1}. Let $\langle \cdot \rangle_t$ denotes the Gibbs measure corresponding to the Hamiltonian $H_{n,t}^{\text{(tot)}}$. Define
	$$
	\phi: t \in [0,1] \mapsto \frac{1}{n} \E_a \E \left[ \log \left( \sum_{\bbf{u},\bbf{v} \in S^n} P^{\otimes n}_0(\bbf{u},\bbf{v}) \exp(H_{n,t}^{\text{(tot)}}(\bbf{u},\bbf{v})) \right) \right] .
	$$
	Let $t \in (0,1)$ be fixed. Using Gaussian integration by parts and the Nishimori identity as we did to prove~\eqref{eq:f_n_p} we compute
	\begin{align}
		\phi'(t) &= 
		\E_a 
		\left[
			\frac{1}{2} \E \Big\langle 
				(\bbf{u}^{(1)}.\bbf{u}^{(2)})(\bbf{v}^{(1)}.\bbf{v}^{(2)}) - q_2 (\bbf{u}^{(1)}.\bbf{u}^{(2)}) - q_1 (\bbf{v}^{(1)}.\bbf{v}^{(2)})
			\Big\rangle_t
		\right] \nonumber
		\\
		&=
		\E_a 
		\left[
			\frac{1}{2} \E \Big\langle 
				(\bbf{u}^{(1)}.\bbf{u}^{(2)} - q_1)(\bbf{v}^{(1)}.\bbf{v}^{(2)} - q_2)
		\Big\rangle_t \right] \nonumber
		- \frac{1}{2} q_2 q_1 \\
		&=
		\E_a \left[ \frac{1}{2} 
		(\E \langle \bbf{u}^{(1)}.\bbf{u}^{(2)} \rangle_t - q_1)(\E \langle \bbf{v}^{(1)}.\bbf{v}^{(2)} \rangle_t - q_2) \right]
		- \frac{1}{2} q_2 q_1 + o_{n}(1) \,, \label{eq:phi_p}
	\end{align}
	where $o_n(1)$ denotes a quantity that goes to $0$ as $n \to \infty$, because of the concentration of the overlaps (Proposition~\ref{prop:overlap_uv}). We will show that the first term of the right-hand side of~\eqref{eq:phi_p} is asymptotically non-negative. This will follow from the fact that the overlaps $\E \langle \bbf{u}^{(1)}.\bbf{u}^{(2)} \rangle_t$ and $\E \langle \bbf{v}^{(1)}.\bbf{v}^{(2)} \rangle_t$ verify the fixed points equations of Corollary~\ref{cor:fixed_point}. 
	Since $(q_1,q_2) \in \Gamma$, we have $q_1 = F_{P_U}( q_2)$. By Corollary~\ref{cor:fixed_point}
	\begin{align*}
		&\E_a \Big| 
		\E \langle \bbf{u}^{(1)}.\bbf{u}^{(2)} \rangle_t 
		- F_{P_U}\left(t  \E \langle \bbf{v}^{(1)} \bbf{v}^{(2)}\rangle_t + (1-t)  q_2\right) 
			\Big| 
			\xrightarrow[n \to \infty]{} 0 \,.
		\end{align*}
	Thus
	\begin{align*}
		&\E_a \left[ (\E \langle \bbf{u}^{(1)}.\bbf{u}^{(2)} \rangle_t - q_1)(\E \langle \bbf{v}^{(1)}.\bbf{v}^{(2)} \rangle_t - q_2) \right] \\
		&\qquad\qquad = \E_a \left[
		\left(F_{P_U}\left(  t \E \langle \bbf{v}^{(1)}.\bbf{v}^{(2)} \rangle_t + (1-t)  q_2 \right) - F_{P_U}( q_2) \right)
		\left(\E \langle \bbf{v}^{(1)}.\bbf{v}^{(2)} \rangle_t- q_2\right)
	\right]
	+ o_{n}(1)
	\\
	&\qquad\qquad \geq o_n(1) \,,
\end{align*}
because $F_{P_U}$ is non-decreasing (Lemma~\ref{lem:general_convex}). Consequently, by Equation~\eqref{eq:phi_p}, $\liminf_{n \to \infty} \phi'(t) \geq - \frac{1}{2} q_2 q_1$.
Using Fatou's lemma
\begin{equation} \label{eq:fatou1}
	\liminf_{n \to \infty} \big( \phi(1) - \phi(0) \big) = \liminf_{n \to \infty} \int_0^1 \phi'(t) dt \geq \int_0^1 \liminf_{n \to \infty} \phi'(t) dt \geq - \frac{1}{2} q_2 q_1\,.
\end{equation}
We have $\phi(1) = \frac{1}{n} \E_a \E \log \left( \sum_{\bbf{u},\bbf{v} \in S^n} P_0^{\otimes n}(\bbf{u},\bbf{v}) e^{H_n^{\text{(tot)}}(\bbf{u},\bbf{v})} \right)$. Lemma~\ref{lem:free_energy_pert} gives us then that $|\phi(1) - F_n | \xrightarrow[n \to \infty]{} 0$, thus
$\liminf\limits_{n \to \infty} \phi(1) = \liminf\limits_{n \to \infty} F_n$. For the same reasons than in the proof of Lemma~\ref{lem:free_energy_pert}, the perturbation term inside $\phi(0)$ will be, in the limit, negligible:
\begin{align*}
	&\limsup_{n \to \infty} \phi(0) 
	\\
	&= \limsup_{n \to \infty} \frac{1}{n} \E \log \left( \sum_{\bbf{u},\bbf{v} \in S^n} P_0^{\otimes n}(\bbf{u},\bbf{v}) \exp(\sum_{i=1}^n \sqrt{  q_2} Z^{(u)}_i u_i +   q_2 u_i U_i - \frac{  q_2}{2} u_i^2 
	+ \sqrt{ q_1} Z^{(v)}_i v_i +   q_1 v_i V_i - \frac{  q_1}{2} v_i^2 )\right)
	\\
	&= \E \log \left( \sum_{u,v \in S} P_0(u,v) \exp(\sqrt{  q_2} Z_1 u +   q_2 u U - \frac{  q_2}{2} u^2 
	+ \sqrt{ q_1} Z_2 v +   q_1 v V - \frac{  q_1}{2} v^2 )\right)
	= \mathcal{F}(q_1,q_2) + \frac{1}{2} q_1 q_2 ,
\end{align*}
and we conclude using equation~\eqref{eq:fatou1}.
\end{proof}

\subsection{Aizenman - Sims - Starr scheme} \label{sec:aizenman}

We prove in this section an upper bound on the limit of the free energy.
We consider the observation system (\ref{eq:channel_m}-\ref{eq:channel_s}-\ref{eq:perturbation}) in the special case $q_u=q_v=0$ (so $H_n^{(s)} =0$) and $t=1$.

\begin{proposition} \label{prop:aizenman}
	\begin{equation} \label{eq:upper_bound}
		\limsup_{n \to \infty} F_n \leq \limsup_{n \to \infty} \E_a \left[ \mathcal{F}\left(\E \langle \bbf{u}^{(1)}.\bbf{u}^{(2)} \rangle_{n,a}, \E \langle \bbf{v}^{(1)}.\bbf{v}^{(2)} \rangle_{n,a} \right) \right].
	\end{equation}
\end{proposition}

By definition of $y$ (Equation~\ref{eq:def_y}), we have $H_{n+1}^{\text{(tot)}} = y + H_n^{\text{(tot)} \prime}$. One can also express $H_n^{\text{(tot)}}$ in term of $H_n^{\text{(tot)}\prime}$.
Let $(\tilde{Z}_{i,j})_{i,j} \iid \mathcal{N}(0,1)$ and $(\tilde{z}^{(u)}_i)_{1 \leq i \leq n}, (\tilde{z}^{(v)}_i)_{1 \leq i \leq n} \iid \mathcal{N}(0,1)$ independent standard Gaussian random variables, independent of everything else. Define
$$
w(\bbf{u},\bbf{v}) = \sum_{1 \leq i,j \leq n} \frac{1}{\sqrt{n(n+1)}} \tilde{Z}_{i,j} u_i v_j + \frac{1}{n(n+1)} u_i U_i v_j V_j - \frac{1}{2 n(n+1)} u_i^2 v_j^2
$$
and
\begin{align*}
	w^{\text{(pert)}}(\bbf{u},\bbf{v}) &= \sum_{i=1}^n \sqrt{s_n - s_{n+1}} a_u \tilde{z}^{(u)}_i u_i + (s_n - s_{n+1}) a_u^2 u_i U_i - \frac{(s_n-s_{n+1}) a_u^2}{2} u_i^2 \\
																						  &+ \sum_{i=1}^n \sqrt{s_n - s_{n+1}} a_v \tilde{z}^{(v)}_i v_i + (s_n - s_{n+1}) a_v^2 v_i V_i - \frac{(s_n-s_{n+1}) a_v^2}{2} v_i^2 \,.
\end{align*}
Then $H_n^{\text{(tot)}}(\bbf{u},\bbf{v}) = w(\bbf{u},\bbf{v}) + w^{\text{(pert)}}(\bbf{u},\bbf{v}) + H_n^{\text{(tot)}\prime}(\bbf{u},\bbf{v})$ in law. Define the perturbed free energy
$$
F_n^{\text{(pert)}} = \frac{1}{n} \E \left[
	\log \left(
		\sum_{\bbf{u},\bbf{v} \in S^n} P_0^{\otimes n}(\bbf{u},\bbf{v}) \exp(H_n^{\text{(tot)}}(\bbf{u},\bbf{v}))	
	\right)
\right]
$$
and $A_n^{\text{(pert)}} = (n+1) F_{n+1}^{\text{(pert)}} - n F_n^{\text{(pert)}}$ so that, with the convention $F_0^{\text{(pert)}} = 0$,
$
F_n^{\text{(pert)}} = \frac{1}{n} \sum_{k=0}^{n-1} A_k^{\text{(pert)}}.
$
Lemma~\ref{lem:free_energy_pert} guarantees that $\E_a | F_n - F_n^{\text{(pert)}} | \xrightarrow[n \to \infty]{} 0$, because $s_n = n^{-1/4} \to 0$.
We thus obtain
$$
\limsup_{n \to \infty} F_n = \limsup_{n \to \infty} \E_a F_n^{\text{(pert)}} \leq \limsup_{n \to \infty} \E_a A_n^{\text{(pert)}}.
$$
It remains therefore to compute the limit of $A_n^{\text{(pert)}}$.
\begin{align*}
	A_n^{\text{(pert)}} &= (n+1) F_{n+1}^{\text{(pert)}} - n F_n^{\text{(pert)}}
	\\
	&=
	\E \left[
		\log \left(
			\sum_{\bbf{u},\bbf{v} \in S^n} P_0^{\otimes n}(\bbf{u},\bbf{v}) \exp(H_{n+1}^{\text{(tot)}}(\bbf{u},\bbf{v}))	
		\right)
	\right]
	-
	\E \left[
		\log \left(
			\sum_{\bbf{u},\bbf{v} \in S^n} P_0^{\otimes n}(\bbf{u},\bbf{v}) \exp(H_n^{\text{(tot)}}(\bbf{u},\bbf{v}))	
		\right)
	\right]
	\\
	&=
	\E 
	\log \left(
		\sum_{\bbf{u},\bbf{v} \in S^n} P_0^{\otimes n}(\bbf{u},\bbf{v}) \left( \sum_{u',v'\in S} P_0(u',v') e^{y(u',v',\bbf{u},\bbf{v})} \right) \exp(H_{n}^{\text{(tot)}\prime}(\bbf{u},\bbf{v}))	
	\right)
	\\
	& \qquad -
	\E 
	\log \left(
		\sum_{\bbf{u},\bbf{v} \in S^n} P_0^{\otimes n}(\bbf{u},\bbf{v}) \exp\big(w(\bbf{u},\bbf{v}) + w^{\text{(pert)}}(\bbf{u},\bbf{v})\big) \exp(H_n^{\text{(tot)}\prime}(\bbf{u},\bbf{v}))	
	\right)
	\\
	&= \E \log \left\langle \sum_{u',v' \in S} P_0(u',v') \exp(y(u',v',\bbf{u},\bbf{v})) \right\rangle'_{\!\!n,a}
	- \E \log \left\langle \exp(w(\bbf{u},\bbf{v}) + w^{\text{(pert)}}(\bbf{u},\bbf{v})) \right\rangle'_{n,a}
	\\
	&= \E \log \left\langle \sum_{u',v' \in S} P_0(u',v') \exp(y(u',v',\bbf{u},\bbf{v})) \right\rangle'_{\!\!n,a}
	- \E \log \left\langle \exp(w(\bbf{u},\bbf{v})) \right\rangle'_{n,a} + o(1) \,,
\end{align*}
because the contribution of $w^{\text{(pert)}}$ is negligible, for the same reasons than in the proof of Lemma~\ref{lem:free_energy_pert}. 
Indeed, since $s_n = n^{-1/4}$, $s_{n+1} - s_n = o(n^{-1})$.
Proposition~\ref{prop:aizenman} follows then from the following lemma.

\begin{lemma}
	$$
	\E_a \left| A_n^{\text{(pert)}} - \mathcal{F}\left(\E \langle \bbf{u}^{(1)}.\bbf{u}^{(2)} \rangle_{n,a} , \E \langle \bbf{v}^{(1)}.\bbf{v}^{(2)} \rangle_{n,a} \right) \right|
	\xrightarrow[n \to \infty]{} 0 \,.
	$$
\end{lemma}

\begin{proof}
	We define
	\begin{align*}
		B_1 &=  \left\langle \sum_{u',v' \in S} P_0(u',v') \exp(y(u',v',\bbf{u},\bbf{v})) \right\rangle'_{n,a} 
		\quad \text{and} \quad
	B_2 = \sum_{u',v' \in S} P_0(u',v') \exp(h_0(u',v')) \,.
	\end{align*}
	Applying Lemma~\ref{lem:l2} with $\phi = 1$, we have $\E_a \E (B_1 - B_2 )^2 \xrightarrow[n \to \infty]{} 0$.
	\begin{lemma}
		$$
		\E_a | \E \log(B_1) - \E \log(B_2) | \xrightarrow[n \to \infty]{} 0 \,.
		$$
	\end{lemma}
	\begin{proof}
		One have $| \log B_1 - \log B_2 | \leq \max(B_1^{-1},B_2^{-1}) | B_1 - B_2|$. So that
		$$
		\E | \log B_1 - \log B_2 | \leq \left( \E [B_1^{-2}+B_2^{-2}]  \E \left[ (B_1 - B_2)^2 \right] \right)^{1/2} \,.
		$$
		Because of the bounded support assumption on $P_0$, $\E [B_1^{-2}+B_2^{-2}]$ is bounded by a constant $C_1$. Thus
		$$
		\E_a | \E \log(B_1) - \E \log(B_2) | 
		\leq \left( C_1 \E_a \E \left[ (B_1 - B_2)^2 \right] \right)^{1/2}
		\xrightarrow[n \to \infty]{} 0 \,.
		$$
	\end{proof}
	\\

	Let $Z_1,Z_2$ be independent standard Gaussian random variables, independent of everything else. 
	Then the processes $(h_0(u',v'))_{u',v' \in S}$ and
	$$
	\left( u'  \sqrt{\frac{ n Q_v'}{n+1}} Z_1 + u' U' \frac{n }{n+1} Q_v' - u'^2 \frac{n }{2(n+1)} Q_v'
		+ v'  \sqrt{\frac{ n Q_u'}{n+1}} Z_2 + v' V' \frac{n }{n+1} Q_u' - v'^2 \frac{n }{2(n+1)} Q_u'
	\right)_{u',v'\in S}
	$$
	have the same law. Indeed, conditionally on $(U',V')$ and $(\langle u_i \rangle'_{n,a}, \langle v_i \rangle'_{n,a})_{1 \leq i \leq n}$ both are Gaussian processes with the same covariance structure. Consequently
	$$
\E \log(B_2) = \E \left[ \psi_{P_U} \left(\frac{ n Q_v'}{n+1}\right) + \psi_{P_V}\left( \frac{ n Q_u'}{n+1} \right) \right].
	$$
	$\psi_{P_U}$ and $\psi_{P_V}$ respectively $\frac{1}{2}\E_{P_U}[U^2]$ and $\frac{1}{2} \E_{P_V}[V^2]$-Lipschitz (see Lemma~\ref{lem:general_convex} in Appendix~\ref{sec:scalar_channel_proofs}), so using~\eqref{eq:concentration_prime} and Corollary~\ref{cor:overlap_p} we get
	$$
	\E_a \left| \E \log(B_2) -  \left( \psi_{P_U} \left( \E \langle \bbf{v}^{(1)}.\bbf{v}^{(2)} \rangle_{n,a} \right) + \psi_{P_V}\left(  \E \langle \bbf{u}^{(1)}.\bbf{u}^{(2)} \rangle_{n,a} \right) \right) \right|
	\xrightarrow[n \to \infty]{} 0
	$$
	and therefore $\displaystyle \E_a \left| \E \log(B_1) -  \left( \psi_{P_U} \left( \E \langle \bbf{v}^{(1)}.\bbf{v}^{(2)} \rangle_{n,a} \right) + \psi_{P_V}\left(  \E \langle \bbf{u}^{(1)}.\bbf{u}^{(2)} \rangle_{n,a} \right) \right) \right| \xrightarrow[n \to \infty]{} 0 \,.$
	Using the same kind of arguments, one show that
	$$
	\E_a \left| \E \log\big\langle \exp(w(\bbf{u},\bbf{v})) \big\rangle'_{n,a} - \frac{1}{2} \E \langle \bbf{v}^{(1)}.\bbf{v}^{(2)} \rangle_{n,a} \E \langle \bbf{u}^{(1)}.\bbf{u}^{(2)} \rangle_{n,a} \right|
	\xrightarrow[n \to \infty]{} 0 \,.
	$$
	We conclude: $\E_a \left| A_n^{\text{(pert)}} - \mathcal{F}\left(\E \langle \bbf{u}^{(1)}.\bbf{u}^{(2)} \rangle_{n,a} , \E \langle \bbf{v}^{(1)}.\bbf{v}^{(2)} \rangle_{n,a} \right) \right| \xrightarrow[n \to \infty]{} 0 \,.$
\end{proof}

\subsection{The final part} \label{sec:final_part}

We conclude the proof of Theorem~\ref{th:rs_formula} in this section, using the results of the previous sections.
We still consider the observation system (\ref{eq:channel_m}-\ref{eq:channel_s}-\ref{eq:perturbation}) with $q_u=q_v=0$ (so $H_n^{(s)} =0$) and $t=1$.

Let $(n_k)_{k \in \N} \in \N^{\N}$ be an extraction along which the superior limit of $(F_n)_n$ is achieved.
$\E \langle \bbf{u}^{(1)}.\bbf{u}^{(2)} \rangle_{n,a}$ and $\E \langle \bbf{v}^{(1)}.\bbf{v}^{(2)} \rangle_{n,a}$ are $(a_u,a_v)$-measurable bounded random variables.
Without loss of generalities we can assume that $(\E \langle \bbf{u}^{(1)}.\bbf{u}^{(2)} \rangle_{n_k,a})_{k \in \N}$ and $(\E \langle \bbf{v}^{(1)}.\bbf{v}^{(2)} \rangle_{n_k,a})_{k \in \N}$ are converging in law along this subsequence (if not, Prokhorov's theorem allows us to find another extraction of $(n_k)$ along with these quantities converges).
Denote by $Q_u^{\infty}$ and $Q_v^{\infty}$ their respective limits. The functions $\mathcal{F}$, $F_{P_U}$, $F_{P_V}$ are continuous and $\E \langle \bbf{u}^{(1)}.\bbf{u}^{(2)} \rangle_{n,a}$ and $\E \langle \bbf{v}^{(1)}.\bbf{v}^{(2)} \rangle_{n,a}$ are bounded, thus by weak convergence, Proposition~\ref{prop:aizenman} and Corollary~\ref{cor:fixed_point} (applied with $q_u = q_v =0$ and $t=1$) give
\begin{align}
	&\E_a \left| Q_u^{\infty} - F_{P_U}(  Q_v^{\infty}) \right| =0 \,, \label{eq:final_lim1} \\
	&\E_a \left| Q_v^{\infty}- F_{P_V}( Q_u^{\infty}) \right| = 0\,, \label{eq:final_lim2} \\
	&\limsup_{n \to \infty} F_n \leq \E_a \mathcal{F}\left(Q_u^{\infty},Q_v^{\infty} \right)\,. \label{eq:final_lim3} 
\end{align}
Equations~\eqref{eq:final_lim1} and~\eqref{eq:final_lim2} give that $(Q_u^{\infty},Q_v^{\infty}) \in \Gamma$ with probability $1$. Therefore, we have
$$
\mathcal{F}\left(Q_u^{\infty},Q_v^{\infty} \right) \leq \sup_{(q_1,q_2) \in \Gamma} \mathcal{F}(q_1,q_2) \,,
$$
almost surely. We conclude, using equation~\eqref{eq:final_lim3} that $\limsup F_n \leq \sup_{\Gamma} \mathcal{F}$, which proves (combined with Proposition~\ref{prop:interpolation}) the first expression for the limit of $F_n$. Theorem~\ref{th:rs_formula} follows then from Proposition~\ref{prop:min_max} in Appendix~\ref{sec:min_max}.

\section{Proof of Theorem~\ref{th:rs_formula_multidim}} \label{sec:proof_rs_multidim}

This section is dedicated to the proof of Theorem~\ref{th:rs_formula_multidim}. It extends the arguments presented in Section~\ref{sec:proof_rs} to the multidimensional case. The ingredients of the proof are the same: we will therefore often refer to the unidimensional proof. As mentioned at the beginning of Section~\ref{sec:proof_rs}, we can restrict ourselves to the case where $P_U$ and $P_V$ have a finite support $S \subset \R^k$, and where $\lambda = 1$, $n=m$ ($\alpha = 1$). We will therefore remove the dependencies in $\lambda,\alpha$. We will write as before $P_0 = P_U \otimes P_V$.

In the multidimensional case the overlaps becomes $k \times k$ matrices. For $\bbf{x}^{(1)},\bbf{x}^{(2)} \in (\R^k)^N$ we write
$$
\bbf{x}^{(1)}.\bbf{x}^{(2)} = \frac{1}{N} \sum_{i=1}^N \bbf{x}_i^{(1)} (\bbf{x}_i^{(2)})^{\intercal}
\in M_{k,k}(\R) \,.
$$
$\| \cdot \|$ will now denote the norm over $M_{k,k}(\R)$ defined as $\| \bbf{A} \| = \sqrt{\Tr(\bbf{A}^{\intercal}\bbf{A})}$.
\subsection{Adding a small perturbation} \label{sec:small_perturbation}
The major difference with the proof presented in Section~\ref{sec:proof_rs} is the kind of perturbation we will add to our observation system, in order to obtain concentration results for the overlaps. Instead of adding low-signal Gaussian scalar channels (see~\eqref{eq:pert_scalar}), we will rather reveal each variable $(\bbf{U}_i, \bbf{V}_i)$ with small probability. Lemma 3.1 from~\cite{andrea2008estimating} shows that this kind of perturbation forces the correlations to decay. This approach has already been used in~\cite{lelarge2016fundamental} and~\cite{coja2016information} to obtain overlaps concentration.
\\

Let $\epsilon\in [0,1]$, and suppose we have access to the additional information, for $1 \leq i \leq n$
\begin{equation} \label{eq:extra_info}
	\bbf{Y}'_i =
	\begin{cases}
		(\bbf{U}_i,\bbf{V}_i) &\text{if } L_i = 1 \,, \\
		(*,*) &\text{if } L_i = 0 \,,
	\end{cases}
\end{equation}
where $L_i \iid \Ber(\epsilon)$ and $*$ is a value that does not belong to $S$. 
The posterior distribution of $(\bbf{U},\bbf{V})$ given $(\bbf{Y},\bbf{Y'})$ is now
\begin{equation} \label{eq:posterior_extra}
	\P((\bbf{U},\bbf{V})=(\bbf{u},\bbf{v})| \bbf{Y},\bbf{Y}') = \frac{1}{\mathcal{Z}_{n,\epsilon}} \left(\prod_{i | \bbf{Y}'_i \neq (*,*)}\bbf{1}((\bbf{u}_i,\bbf{v}_i)=\bbf{Y}'_i) \right) \left( \prod_{i | \bbf{Y}_i'=(*,*)} P_0(\bbf{u}_i,\bbf{v}_i) \right) e^{H_n(\bbf{u},\bbf{v})} \,,
\end{equation}
where $\mathcal{Z}_{n,\epsilon}$ is the appropriate normalization constant.
For $(\bbf{u}, \bbf{v}) \in S^n \times S^n$ we will use the following notations
\begin{align} 
	\bbf{\bar{u}} &= (\bar{u}_1, \dots, \bar{u}_n) = (L_1 \bbf{U}_1 + (1-L_1) \bbf{u}_1, \dots, L_n \bbf{U}_n + (1-L_n)\bbf{u}_n) \,,\label{eq:bar_1}\\
	\bbf{\bar{v}} &= (\bar{v}_1, \dots, \bar{v}_n) = (L_1 \bbf{V}_1 + (1-L_1) \bbf{v}_1, \dots, L_n \bbf{V}_n + (1-L_n)\bbf{v}_n) \,.\label{eq:bar_2}
\end{align}
$\bbf{\bar{u}}$ and $\bbf{\bar{v}}$ are thus obtained by replacing the coordinates of $\bbf{u}$ and $\bbf{v}$ that are revealed by $\bbf{Y}'$ by their revealed values. The notations $\bbf{\bar{u}}$ and $\bbf{\bar{v}}$ will allow us to obtain a very convenient expression for the free energy of the perturbed model which is defined as
$$
F_{n,\epsilon} = \frac{1}{n} \E \log \mathcal{Z}_{n,\epsilon} = \frac{1}{n} \E \Big[ \log \sum_{\bbf{u}, \bbf{v} \in S^n} P_0^{\otimes n}(\bbf{u},\bbf{v}) \exp(H_{n}(\bbf{\bar{u}},\bbf{\bar{v}}))\Big] .
$$
The following Proposition comes from~\cite{lelarge2016fundamental} (Proposition~22):
\begin{proposition} \label{prop:approximation_f_n_epsilon}
	For all $n \geq 1$ and all $\epsilon \in [0,1]$, we have
	$$
	|F_{n} - F_{n,\epsilon} | \leq H(P_U \otimes P_V) \epsilon.
	$$
\end{proposition}

We define now $\epsilon$ as a uniform random variable over $[0, 1]$, independently of every other random variable. We will note $\E_{\epsilon}$ the expectation with respect to $\epsilon$. For $n \geq 1$, we define also $\epsilon_n = n^{-1/2} \epsilon \sim \mathcal{U}[0, n^{-1/2}]$. Proposition~\ref{prop:approximation_f_n_epsilon} implies that
$$
\big| F_n - \E_{\epsilon} [F_{n, \epsilon_n}] \big| \xrightarrow[n \to \infty]{} 0.
$$
It remains therefore to compute the limit of the free energy averaged over small perturbations.

\subsection{Overlap concentration}

Let $\langle \cdot \rangle_{n,\epsilon}$ denote the expectation with respect to the posterior distribution~\eqref{eq:posterior_extra} of $(\bbf{U},\bbf{V})$ given $(\bbf{Y},\bbf{Y}')$. The Nishimori identity (Proposition~\ref{prop:nishimori}) will thus be valid under $\langle \cdot \rangle_{n,\epsilon}$. We recall that $\bbf{Y}'$ is defined in~\eqref{eq:extra_info}, where $L_i \iid \Ber(\epsilon_n)$ are independent random variables. 
\\

The following lemma comes from~\cite{andrea2008estimating} (Lemma~3.1). It shows that the extra information $\bbf{Y}'$ forces the correlations to decay.

\begin{lemma}
	$$
	n^{-1/2} \E_{\epsilon}  \left[ \frac{1}{n^2} \sum_{1 \leq i,j \leq n} I\big((\bbf{U}_i,\bbf{V}_i);(\bbf{U}_j,\bbf{V}_j) \ | \ \bbf{Y},\bbf{Y'}\big) \right] \leq \frac{2 H(P_0)}{n}  \,.
	$$
\end{lemma}

This implies that the overlap between two replicas, i.e.\ two independent samples $(\bbf{u}^{(1)},\bbf{v}^{(1)})$ and $(\bbf{u}^{(2)},\bbf{v}^{(2)})$ from the Gibbs distribution $\langle \cdot \rangle_{n,\epsilon}$, concentrates. Let us define
\begin{align}
	\bbf{Q}_u &= \Big\langle \frac{1}{n} \sum_{i=1}^n \bbf{u}^{(1)}_i (\bbf{u}^{(2)}_i)^{\intercal} \Big\rangle_{n,\epsilon} = \langle \bbf{u}^{(1)} . \bbf{u}^{(2)} \rangle_{n,\epsilon} \label{eq:def_Q_u} \\
	\bbf{Q}_v &= \Big\langle \frac{1}{n} \sum_{i=1}^n \bbf{v}^{(1)}_i (\bbf{v}^{(2)}_i)^{\intercal} \Big\rangle_{n,\epsilon} = \langle \bbf{v}^{(1)} . \bbf{v}^{(2)} \rangle_{n,\epsilon} \label{eq:def_Q_v}
\end{align}
$\bbf{Q}_u$ and $\bbf{Q}_v$ are two random variables depending only on $(Y_{i,j})_{1 \leq i,j\leq n}$ and $(\bbf{Y}_i')_{1 \leq i \leq n}$. Notice that $\bbf{Q}_u, \bbf{Q}_v \in S_k^+$. 

\begin{proposition}[Overlap concentration]  \label{prop:overlap_multidim}
	\begin{align*}
		\E_{\epsilon} \E \Big\langle \| \bbf{u}^{(1)} . \bbf{u}^{(2)} - \bbf{Q}_u \|^2 \Big\rangle_{n,\epsilon} \xrightarrow[n \to \infty]{} 0
		, \quad \text{ and } \quad
		\E_{\epsilon} \E \Big\langle \| \bbf{v}^{(1)} . \bbf{v}^{(2)} - \bbf{Q}_v \|^2 \Big\rangle_{n,\epsilon} \xrightarrow[n \to \infty]{} 0.
	\end{align*}
\end{proposition}
See~\cite{lelarge2016fundamental}, Proposition~49 for a proof.

\subsection{Aizenman-Sims-Starr scheme}

Using the concentration results of Proposition~\ref{prop:overlap_multidim} the proofs of Section~\ref{sec:aizenman} can be extended to the multidimensional case.

\begin{proposition} \label{prop:aizenman_multidim}
	\begin{equation} \label{eq:upper_bound_multidim}
		\limsup_{n \to \infty} F_n \leq \limsup_{n \to \infty} \E_{\epsilon} \E \mathcal{F}\left(\bbf{Q}_u, \bbf{Q}_v\right)  \,.
	\end{equation}
\end{proposition}

\subsection{Fixed point equations}

Let $\bbf{q}_u, \bbf{q}_v \in S_k^+$.
Suppose that we have access to the additional observations
\begin{align*}
	Y_i^{(u)} &= \bbf{q}_u^{1/2} \bbf{U}_i + \bbf{Z}_i^{(u)}, \quad \ \text{for} \  1 \leq i \leq n \,,
	\\
	Y_i^{(v)} &= \bbf{q}_v^{1/2} \bbf{V}_i + \bbf{Z}_i^{(v)}, \quad \ \text{for} \  1 \leq i \leq n \,,
\end{align*}
where $(\bbf{Z}_i^{(u)})_{1 \leq i \leq n}$ and $(\bbf{Z}_i^{(v)})_{1 \leq i \leq n}$ are i.i.d.\ $\mathcal{N}(0,\bbf{I}_k)$, independently of everything else.
Let $\bbf{Q_u}$ and $\bbf{Q}_v$ be defined as in equations~\eqref{eq:def_Q_u} and~\eqref{eq:def_Q_v}, where the Gibbs measure $\langle \cdot \rangle_{n,\epsilon}$ denotes the posterior distribution of $(U,V)$ given $Y$, $Y'$, $Y^{(u)}$ and $Y^{(v)}$. Notice that Proposition~\ref{prop:overlap_multidim} still hold for this Gibbs distribution (the proofs are the same). The arguments of Section~\ref{sec:tap} can be extended to the multidimensional case to obtain the multidimensional version of Corollary~\ref{cor:fixed_point}:

\begin{proposition} \label{prop:fixed_point_multidim}
	\begin{align*}
		\E_{\epsilon} \E \Big\| \bbf{Q}_u - F_{P_U}(\bbf{Q}_v + \bbf{q}_u ) \Big\| \xrightarrow[n \to \infty]{} 0 
		\quad \text{and} \quad
		\E_{\epsilon} \E \Big\| \bbf{Q}_v - F_{P_V}(\bbf{Q}_u + \bbf{q}_v) \Big\| \xrightarrow[n \to \infty]{} 0 \,.
	\end{align*}
\end{proposition}

\subsection{The lower bound: interpolation method}

\begin{proposition} \label{prop:interpolation_multidim}
	$$
	\liminf_{n \to \infty} F_n \geq \sup_{(\bbf{q}_1,\bbf{q}_2) \in \Gamma }\mathcal{F}(\bbf{q}_1,\bbf{q}_2) \,.
	$$
\end{proposition}

\begin{proof}
	Let $(\bbf{q}_1,\bbf{q}_2) \in \Gamma$.
	Define, for $t \in [0,1]$, the Hamiltonians
	\begin{align*}
		H_{n,t}(\bbf{u},\bbf{v}) &= \sum_{1 \leq i,j \leq n} \sqrt{\frac{ t}{n}} \bbf{u}_i^{\intercal} \bbf{v}_j \bbf{Z}_{i,j} + \frac{ t}{n} \bbf{u}_i^{\intercal} \bbf{v}_j \bbf{U}_i^{\intercal} \bbf{V}_j - \frac{ t}{2n} (\bbf{u}_i^{\intercal} \bbf{v}_j)^2 \,, \\
		H_{n,t}^{(s)}(\bbf{u},\bbf{v}) &= \sum_{i=1}^n \sqrt{ (1-t)} \bbf{Z}^{(u)\intercal}_i\bbf{q}_2^{1/2} \bbf{u}_i +  (1-t) \bbf{u}_i^{\intercal} \bbf{q}_2 \bbf{U}_i - \frac{ (1-t) \bbf{q}_2}{2} (\bbf{u}_i^{\intercal} \bbf{u}_i)^2 \\
																							 &+ \sum_{i=1}^n \sqrt{ (1-t)} \bbf{Z}^{(v)\intercal}_i \bbf{q}_1^{1/2} \bbf{v}_i +  (1-t) \bbf{v}_i^{\intercal} \bbf{q}_1 \bbf{V}_i - \frac{ (1-t) \bbf{q}_1}{2} (\bbf{v}_i^{\intercal} \bbf{v}_i)^2 \,,
	\end{align*}
	for $\bbf{u},\bbf{v} \in S^n$, and $H_{n,t}^{\text{(tot)}} = H_{n,t} + H_{n,t}^{(s)}$. Let $\langle \cdot \rangle_t$ be the Gibbs measure defined as
	$$
	\langle h(\bbf{u},\bbf{v}) \rangle_t = \frac{\sum_{\bbf{u},\bbf{v} \in S^n} P_0^{\otimes n}(\bbf{u},\bbf{v}) h(\bbf{\bar{u}},\bbf{\bar{v}}) \exp(H_{n,t}^{\text{(tot)}}(\bbf{\bar{u}},\bbf{\bar{v}}))}
	{\sum_{\bbf{u},\bbf{v} \in S^n} P_0^{\otimes n}(\bbf{u},\bbf{v}) \exp(H_{n,t}^{\text{(tot)}}(\bbf{\bar{u}},\bbf{\bar{v}}))}
	, \ \text{for any function } h \ \text{on } (S^n)^2,
	$$
	where we recall that the notations $\bbf{\bar{u}}$ and $\bbf{\bar{v}}$ are defined by (\ref{eq:bar_1}-\ref{eq:bar_2}).
	Define
	$$
	\phi: t \in [0,1] \mapsto \frac{1}{n} \E_{\epsilon} \E \left[ \log \left( \sum_{\bbf{u},\bbf{v} \in S^n} P_0^{\otimes n}(\bbf{u},\bbf{v}) \exp(H_{n,t}^{\text{(tot)}}(\bbf{\bar{u}},\bbf{\bar{v}})) \right) \right].
	$$
	Let $t \in (0,1)$ be fixed. Gaussian integration by parts and Nishimori identity lead to
	\begin{align}
		\phi'(t) &= 
		\E_{\epsilon} 
		\left[
			\frac{1}{2} \E \Big\langle 
				\Tr\big[(\bbf{u}^{(1)}.\bbf{u}^{(2)})^{\intercal}(\bbf{v}^{(1)}.\bbf{v}^{(2)}) - \bbf{q}_2 (\bbf{u}^{(1)}.\bbf{u}^{(2)}) - \bbf{q}_1 (\bbf{v}^{(1)}.\bbf{v}^{(2)}) \big]
			\Big\rangle_t
		\right] \nonumber
		\\
		&=
		\E_{\epsilon} 
		\left[
			\frac{1}{2} \E \Big\langle 
				((\bbf{u}^{(1)}.\bbf{u}^{(2)})^{\intercal} - \bbf{q}_1)(\bbf{v}^{(1)}.\bbf{v}^{(2)} - \bbf{q}_2)
		\Big\rangle_t \right] \nonumber
		- \frac{1}{2} \Tr [\bbf{q}_2 \bbf{q}_1] \\
		&=
		\frac{1}{2} \E_{\epsilon} \E\left[ 
		\Tr \big[ (\langle \bbf{u}^{(1)}.\bbf{u}^{(2)} \rangle_t - \bbf{q}_1)(\langle \bbf{v}^{(1)}.\bbf{v}^{(2)} \rangle_t - \bbf{q}_2) \big] \right]
		- \frac{1}{2} \Tr[\bbf{q}_2 \bbf{q}_1] + o_{n}(1) \,, \label{eq:phi_p_multidim}
	\end{align}
	where $o_n(1)$ denotes a quantity that goes to $0$ as $n \to \infty$, because of the concentration of the overlaps (Proposition~\ref{prop:overlap_multidim}). 
	By Proposition~\ref{prop:fixed_point_multidim}
	\begin{align*}
		&\E_{\epsilon} \E \Big\| \langle \bbf{u}^{(1)}.\bbf{u}^{(2)} \rangle_t - F_{P_U}( t \langle \bbf{v}^{(1)}.\bbf{v}^{(2)} \rangle_t +  (1-t)\bbf{q}_2) \Big\| \xrightarrow[n \to \infty]{} 0 \,,
	\end{align*}
	We have also $\bbf{q}_1 = F_{P_U}( \bbf{q}_2)$.
	Thus
	\begin{align*}
		& \E_{\epsilon} \E \left[  \Tr\left[ (\langle \bbf{u}^{(1)}.\bbf{u}^{(2)} \rangle_t - \bbf{q}_1)( \langle \bbf{v}^{(1)}.\bbf{v}^{(2)} \rangle_t - \bbf{q}_2) \right] \right] \\
		&\qquad \qquad= \E_{\epsilon} \E \left[  \Tr\left[ 
\left(F_{P_U}( t \langle \bbf{v}^{(1)}.\bbf{v}^{(2)}\rangle_t +  (1-t)\bbf{q}_2)- F_{P_U}( \bbf{q}_2)\right)
\left( \langle \bbf{v}^{(1)}.\bbf{v}^{(2)} \rangle_t - \bbf{q}_2\right) 
\right] \right] + o_n(1)
\\
& \qquad \qquad \geq o_n(1)
\end{align*}
because $\frac{1}{2}F_{P_U}$ is the gradient of the convex function $\psi_{P_U}$ (Lemma~\ref{lem:general_convex}).
Consequently, by Equation~\eqref{eq:phi_p_multidim}, $\liminf_{n \to \infty} \phi'(t) \geq - \frac{1}{2} \bbf{q}_2 \bbf{q}_1$. Using Fatou's lemma
\begin{equation}\label{eq:ineq_phi_p}
	\liminf_{n \to \infty} \big( \phi(1) - \phi(0) \big) = \liminf_{n \to \infty} \int_0^1 \phi'(t) dt \geq \int_0^1 \liminf_{n \to \infty} \phi'(t) dt \geq - \frac{1}{2} \bbf{q}_2 \bbf{q}_1 \,.
\end{equation}
We have $\phi(1) = \frac{1}{n} \E_{\epsilon} \E \log \left( \sum_{\bbf{u},\bbf{v} \in S^n} P_0^{\otimes n}(\bbf{u},\bbf{v}) e^{H_n^{\text{(tot)}}(\bbf{\bar{u}},\bbf{\bar{v}})} \right) = \E_{\epsilon} [F_{n,\epsilon}]$. 
Hence
$\liminf\limits_{n \to \infty} \phi(1) = \liminf\limits_{n \to \infty} F_n$. 
Analogously to Proposition~\ref{prop:approximation_f_n_epsilon}, the effect of the perturbation term inside $\phi(0)$ will be, in the limit, negligible:
\begin{align*}
	\limsup_{n \to \infty} \phi(0) 
	&= \limsup_{n \to \infty} \frac{1}{n} \E \log \Big( \sum_{\bbf{u},\bbf{v} \in S^n} P_0^{\otimes n}(\bbf{u},\bbf{v}) \exp(\sum_{i=1}^n \bbf{Z}^{(u)\intercal} \bbf{q}_2^{1/2} \bbf{u}_i +   \bbf{u}_i^{\intercal} \bbf{q}_2 \bbf{U}_i - \frac{  }{2} \bbf{u}_i^{\intercal} \bbf{q}_2 \bbf{u}_i
	\\
	&\qquad \qquad\qquad\qquad\qquad\qquad\qquad \qquad \qquad 
	+ \sum_{i=1}^n \bbf{Z}^{(v)}_i \bbf{q}_1^{1/2} \bbf{v}_i +   \bbf{v}_i^{\intercal} \bbf{q}_1 \bbf{V}_i - \frac{1}{2} \bbf{v}_i^{\intercal} \bbf{q}_1 \bbf{v}_i)\Big)
	\\
	&= \E \log\!\! \left( \!\sum_{\bbf{u},\bbf{v} \in S}\!\! P_0(\bbf{u},\bbf{v}) \exp(  \bbf{Z}_1^{\intercal} \bbf{q}_2^{1/2} \bbf{u} +   \bbf{u}^{\intercal} \bbf{q}_2 \bbf{U} - \frac{1}{2} \bbf{u}^{\intercal} \bbf{q}_2 \bbf{u}
	+ \bbf{Z}_2 \bbf{q}_1^{1/2} \bbf{v} +   \bbf{v}^{\intercal} \bbf{q}_1 \bbf{V} - \frac{ }{2} \bbf{v}^{\intercal} \bbf{q}_1 \bbf{v})\!\!\right)
	\\
	&= \mathcal{F}(\bbf{q}_1,\bbf{q}_2) + \frac{1}{2} \bbf{q}_2 \bbf{q}_1 \,.
\end{align*}
We conclude using equation~\eqref{eq:ineq_phi_p}: $\liminf\limits_{n \to \infty} F_n \geq \mathcal{F}(\bbf{q}_1,\bbf{q}_2)$.
\end{proof}

\subsection{The final part}

The remaining of the proof is exactly the same than in the unidimensional case (Section~\ref{sec:final_part}): the variables $\bbf{Q}_u$ and $\bbf{Q}_v$ converge along a subsequence to a point of $\Gamma$, because of Proposition~\ref{prop:fixed_point_multidim}. This proves the converse bound of Proposition~\ref{prop:interpolation_multidim} and thus Theorem~\ref{th:rs_formula_multidim} (again we use Proposition~\ref{prop:min_max} to obtain the ``max-min formula'').

\begin{appendices}

\section{The linear Gaussian channel} \label{sec:linear_channel}

In this section we will work with positive semi-definite matrices. We will denote by $S_k^+$ the set of $k \times k$ positive semi-definite matrices. Recall that $S_k^+$ is a convex cone. We will also use Loewner (partial) order $\preceq$ on $S_k^+$. For $\bbf{A},\bbf{B} \in S_k^+$,
$$
\bbf{A} \preceq \bbf{B} \iff \forall \bbf{x} \in \R^k, \ \bbf{x}^{\intercal} \bbf{A} \bbf{x} \leq \bbf{x}^{\intercal} \bbf{B} \bbf{x} \iff \bbf{B} - \bbf{A} \in S_k^+.
$$
We will also use the strict inequality $\prec$: $\bbf{A} \prec \bbf{B} \iff \forall \bbf{x} \in \R^k, \ \bbf{x}^{\intercal} \bbf{A} \bbf{x} < \bbf{x}^{\intercal} \bbf{B} \bbf{x}$. Note that when $k=1$, $\preceq$ and $\prec$ correspond to the usual ordering of $\R$.

\subsection{Properties of the linear Gaussian channel} \label{sec:scalar_channel_proofs}

Let $P_X$ be a probability distribution on $\R^k$ ($k \geq 1$) with finite second moment and $\bbf{X} \sim P_X$. Let $\bbf{Z} \sim \mathcal{N}(\bbf{0},\bbf{I}_k)$ be independent from $\bbf{X}$. Let $\bbf{q} \in S_k^+$ and suppose that we observe
$$
\bbf{Y} = \bbf{q}^{1/2} \bbf{X} + \bbf{Z} \,.
$$
We define the Gibbs measure $\langle \cdot \rangle_{\bbf{q}}$ as the expectation associated to the posterior distribution $\P(\bbf{X} | \bbf{Y})$, defined by
\begin{equation}\label{eq:def_gibbs_scalar}
	\big\langle f(\bbf{x}) \big\rangle_{\bbf{q}} 
	= \frac{\int_{\bbf{x} \in \R^k} dP_X(\bbf{x})f(\bbf{x}) \exp\big( \bbf{Z}^{\intercal} \bbf{q}^{1/2} \bbf{x} + \bbf{X}^{\intercal} \bbf{q} \bbf{x} - \frac{1}{2} \bbf{x}^{\intercal} \bbf{q} \bbf{x} \big)}{\int_{\bbf{x} \in \R^k} dP_X(\bbf{x}) \exp\big( \bbf{Z}^{\intercal} \bbf{q}^{1/2} \bbf{x} + \bbf{X}^{\intercal} \bbf{q} \bbf{x} - \frac{1}{2} \bbf{x}^{\intercal} \bbf{q} \bbf{x} \big)} \,,
\end{equation}
for any continuous bounded function $f$.
Let $\bbf{x}$ be distributed according to $\langle \cdot \rangle_{\bbf{q}}$ independently of everything else. We define the overlap function:
$$
F_{P_X}: \bbf{q} \in S_k^+ \mapsto 
\E \langle \bbf{x} \bbf{X}^{\intercal} \rangle_{\bbf{q}}
=\E \left[ \langle \bbf{x} \rangle_{\bbf{q}} \langle \bbf{x} \rangle_{\bbf{q}}^{\intercal} \right] \in S_k^+ \,.
$$
We also define the free energy function
\begin{align*}
	\psi_{P_X}: \bbf{q} \in S_k^+ \mapsto \E \log \Big(\int_{\bbf{x} \in \R^k} dP_X(\bbf{x}) \exp\big( \bbf{Z}^{\intercal} \bbf{q}^{1/2} \bbf{x} + \bbf{X}^{\intercal} \bbf{q} \bbf{x} - \frac{1}{2} \bbf{x}^{\intercal} \bbf{q} \bbf{x} \big)\Big) \,.
\end{align*}
It is not difficult to verify that both functions are continuous over $S_k^+$.

\begin{lemma}\label{lem:mmse_opt}
	For all measurable function $f: \R^k \mapsto \R^k$, we have
	$$
	\bbf{0} \preceq \E_{P_X} \big[\bbf{X} \bbf{X}^{\intercal}\big] - F_{P_X}(\bbf{q}) \preceq \E \Big[(\bbf{X} - f(\bbf{Y}))(\bbf{X} - f(\bbf{Y}))^{\intercal}\Big] \,.
	$$
\end{lemma}
\begin{proof}
	Let $\bbf{v} \in \R^k$. Define the random variable $X_{\bbf{v}} = \bbf{X}^{\intercal} \bbf{v}$. We have
	$$
	\bbf{v}^{\intercal} \big(\E_{P_X} \big[\bbf{X} \bbf{X}^{\intercal}\big] - F_{P_X}(\bbf{q}) \big) \bbf{v} = \E \big[X_{\bbf{v}}^2 \big] - \E \big[X_{\bbf{v}} \E[X_{\bbf{v}}| \bbf{Y}]\big] = \min_{\hat{\theta}} \E \Big[\big(X_{\bbf{v}} - \hat{\theta}(\bbf{Y})\big)^2\Big] \,,
	$$
	where the minimum is taken with respect all measurable function $\hat{\theta}: \R^k \to \R$. The lemma follows from the fact that $\bbf{v}^{\intercal}\E \Big[(\bbf{X} - f(\bbf{Y}))(\bbf{X} - f(\bbf{Y}))^{\intercal}\Big]\bbf{v} = \E \big[(X_{\bbf{v}} - \bbf{v}^{\intercal}f(\bbf{Y}))^2\big]$.
\end{proof}
\\

The next lemma states the main properties of the functions $\psi_{P_X}$ and $F_{P_X}$.

\begin{lemma} \label{lem:general_convex}
	\begin{enumerate}[(i)]
		\item \label{item:psi_convex} $\psi_{P_X}$ is convex, 
		\item \label{item:psi_strictly} If $\Cov(\bbf{X})$ is inversible, then $\psi_{P_X}$ is strictly convex.
		\item \label{item:psi_diff} $\psi_{P_X}$ is differentiable on $S_k^+$ and for $\bbf{q} \in S_k^+$
			$$
			\nabla \psi_{P_X} (\bbf{q}) = \frac{1}{2} F_{P_X}(\bbf{q})
			$$
		\item \label{item:F_inc} $F_{P_X}$ is non-decreasing in the sense that, if $\bbf{q}_1 \preceq \bbf{q}_2$, then $F_{P_X}(\bbf{q}_1) \preceq F_{P_X}(\bbf{q}_2)$. If $\Cov(\bbf{X}) \succ \bbf{0}$ and $\bbf{q}_1 \prec \bbf{q}_2$, then $F_{P_X}(\bbf{q}_1) \prec F_{P_X}(\bbf{q}_2)$.
		\item \label{item:range} For $\bbf{q} \in S_k^+$, $\bbf{0} \preceq F_{P_X}(\bbf{q}) \preceq \E[\bbf{X} \bbf{X}^{\intercal}]$ ($\prec \E[\bbf{X}\bbf{X}^{\intercal}]$ if $\Cov(\bbf{X})$ is inversible).
		\item \label{item:F_0} $F_{P_X}(\bbf{0})=\E[\bbf{X}] \E[\bbf{X}]^{\intercal}$.
		\item \label{item:lim_F} $F_{P_X}(\bbf{q}) \xrightarrow[\bbf{q} \to \infty]{} \E\left[\bbf{X} \bbf{X}^{\intercal} \right]$. ($\bbf{q}\to \infty$ mean here that \textbf{all} the eigenvalues of $\bbf{q}$ go to infinity).
	\end{enumerate}
\end{lemma}

\begin{proof}
	To prove~(\ref{item:psi_convex}) it suffices to show that $g:t \in [0,1] \mapsto \psi_{P_X}(t \bbf{q}_1 + (1-t)\bbf{q}_2)$ is convex, for all $\bbf{q}_1,\bbf{q}_2 \in S_k^{+}$.
	\\

	Let $\bbf{q}_1, \bbf{q}_2 \in S_k^{+}$. Let $\bbf{Z}_1, \bbf{Z}_2 \sim \mathcal{N}(0,\bbf{I}_k)$ two independent standard random variables, independent of any other random variable. Define $\bbf{q}_t =t \bbf{q}_1 + (1-t) \bbf{q}_2$. We have
	$$
	\bbf{q}_t^{1/2} \bbf{Z} = \sqrt{t} \bbf{q}_1^{1/2} \bbf{Z}_1 + \sqrt{1-t} \bbf{q}_2^{1/2} \bbf{Z}_2
	$$
	in law. Define for $t \in [0,1]$,
	$$
	g(t) = \psi_{P_X}(\bbf{q}_t) = \E \log \Big(\int dP_X(\bbf{x}) \exp\big( \sqrt{t} \bbf{Z}_1^{\intercal} \bbf{q}_1^{1/2} \bbf{x} + \sqrt{1-t} \bbf{Z}_2^{\intercal} \bbf{q}_1^{1/2} \bbf{x} + \bbf{X}^{\intercal} \bbf{q}_t \bbf{x} - \frac{1}{2} \bbf{x}^{\intercal} \bbf{q}_t \bbf{x} \big)\Big) \,.
	$$
	$g$ is continuous on $[0,1]$, differentiable on $(0,1)$. Let $t\in (0,1)$. 
	Define $\bbf{q} = \bbf{q}_1 - \bbf{q}_2$.
	We have
	\begin{align*}
		g'(t) &= 
		\E \Big\langle \frac{1}{2 \sqrt{t}} \bbf{Z}_1^{\intercal} \bbf{q}_1^{1/2} \bbf{x} - \frac{1}{2 \sqrt{1-t}} \bbf{Z}_2^{\intercal} \bbf{q}_2^{1/2} \bbf{x} + \bbf{X}^{\intercal} \bbf{q} \bbf{x} - \frac{1}{2} \bbf{x}^{\intercal} \bbf{q} \bbf{x} \Big\rangle_{\bbf{q}_t}
		\\
		&=
		\E \Big\langle \frac{1}{2} \bbf{x}^{\intercal} \bbf{q}_1 \bbf{x} - \frac{1}{2} \bbf{x}^{\intercal} \bbf{q}_2 \bbf{x} + \bbf{X}^{\intercal} \bbf{q} \bbf{x} - \frac{1}{2} \bbf{x}^{\intercal} \bbf{q} \bbf{x} \Big\rangle_{\bbf{q}_t} - \E \big\langle \frac{1}{2} \bbf{x}^{\intercal} \bbf{q}_1 \big\rangle_{\bbf{q}_t} \langle \bbf{x} \rangle_{\bbf{q}_t} + \E \big\langle \frac{1}{2} \bbf{x}^{\intercal} \bbf{q}_2 \big\rangle_{\bbf{q}_t} \langle \bbf{x} \rangle_{\bbf{q}_t}
		\\
		&=
		\frac{1}{2} \E \Big\langle \bbf{X}^{\intercal} \bbf{q} \bbf{x} \Big\rangle_{\bbf{q}_t} = \frac{1}{2} \Tr\big( \bbf{q} \E \langle \bbf{x} \bbf{X}^{\intercal} \rangle_{\bbf{q}_t} \big) \,,
	\end{align*}
	where we used successively Gaussian integration by parts and the Nishimori identity. This derivative is continuous in $0$ and $1$, so $g$ is also differentiable at those points. 
	This proves~(\ref{item:psi_diff}).
	Similar computations shows that for $t \in (0,1)$,
	\begin{align*}
		g''(t) = \frac{1}{2} \E \, \Tr\Big[  \big( \bbf{q} (\langle \bbf{x} \bbf{x}^{\intercal} \rangle_{\bbf{q}_t} - \langle \bbf{x} \rangle_{\bbf{q}_t} \langle \bbf{x}^{\intercal} \rangle_{\bbf{q}_t}) \big)^2 \Big] \geq 0 \,,
	\end{align*}
	by Lemma~\ref{lem:trace_convex} below. This proves~(\ref{item:psi_convex}). To prove~(\ref{item:psi_strictly}) is suffices to show that $g''(t) > 0$ when $\Cov(\bbf{X}) \succ \bbf{0}$ and $\bbf{q}_1 \neq \bbf{q_2}$. Suppose that $g''(t) = 0$. 
	Then $\Tr\Big[  \big( \bbf{q} (\langle \bbf{x} \bbf{x}^{\intercal} \rangle_{\bbf{q}_t} - \langle \bbf{x} \rangle_{\bbf{q}_t} \langle \bbf{x}^{\intercal} \rangle_{\bbf{q}_t}) \big)^2 \Big] = 0$ almost surely.
	\begin{lemma} \label{lem:as_succ}
		If $\Cov(\bbf{X}) \succ \bbf{0}$ then for all $\bbf{q} \in S_k^+$
		$$
		\langle \bbf{x} \bbf{x}^{\intercal} \rangle_{\bbf{q}} - \langle \bbf{x} \rangle_{\bbf{q}} \langle \bbf{x}^{\intercal} \rangle_{\bbf{q}} \succ \bbf{0}
		$$ almost surely.
	\end{lemma}
	\begin{proof}
		Let $\bbf{M} = \langle \bbf{x} \bbf{x}^{\intercal} \rangle_{\bbf{q}} - \langle \bbf{x} \rangle_{\bbf{q}} \langle \bbf{x}^{\intercal} \rangle_{\bbf{q}} \succeq \bbf{0}$. Suppose that there exists $\bbf{v} \in \R^k \setminus \{\bbf{0}\}$ such that $\bbf{M} \bbf{v} =\bbf{0}$. Then $\left\langle \| \bbf{v}^{\intercal}(\bbf{x} - \langle \bbf{x} \rangle_{\bbf{q}} ) \|^2 \right\rangle_{\bbf{q}} = 0$. Therefore $\bbf{v}^{\intercal} \bbf{x} = \bbf{v}^{\intercal} \langle \bbf{x} \rangle_{\bbf{q}}$, $\langle \cdot \rangle_{\bbf{q}}$-almost surely.
		However $P_X$  is almost-surely absolutely continuous with respect to $\langle \cdot \rangle_{\bbf{q}}$ (his Radon-Nikodym derivative is almost surely $>0$). This implies that $\bbf{v}^{\intercal} \bbf{X}$ is constant: $\bbf{v}^{\intercal} \Cov(\bbf{X}) \bbf{v} = 0$. We obtain a contradiction.
	\end{proof}
	\\

	Combining Lemma~\ref{lem:as_succ} and Lemma~\ref{lem:trace_convex} below, we obtain $\bbf{q} = 0$ which is absurd. This proves~(\ref{item:psi_strictly}).
	\\

	Let us now prove (\ref{item:F_inc}).
	Let $\bbf{q},\bbf{q}' \in S_k^+$ and $v\in \R^k \setminus \{\bbf{0}\}$. For $0 \leq t \leq 1$ we define $\bbf{q}_t = \bbf{q} + t\bbf{q}'$ and $h(t) = \bbf{v}^{\intercal} F_{P_X}(\bbf{q}_t) \bbf{v}$. 
	In order to prove (\ref{item:F_inc}) we have to show that $h$ is non-decreasing and is increasing in the case where $\Cov(\bbf{X}) \succ \bbf{0}$ and $\bbf{q}' \succ \bbf{0}$.
	Using Gaussian integration by parts and the Nishimori property, one can show that for $t \in (0,1)$,
	$$
	h'(t) = \frac{1}{2} \E \left\| \bbf{q}^{\prime 1/2}\left(\langle \bbf{x} \bbf{x}^{\intercal} \rangle_{\bbf{q}_t} - \langle \bbf{x} \rangle_{\bbf{q}_t} \langle \bbf{x}^{\intercal} \rangle_{\bbf{q}_t}\right) \bbf{v} \right\|^2 \geq 0 \,.
	$$
	Now, if $\bbf{q}' \succ \bbf{0}$, using Lemma~\ref{lem:as_succ} we see that $h'(t) >0$. This proves~(\ref{item:F_inc}).
	\\
(\ref{item:F_0}) is obvious.
	Notice that for $\bbf{q} \in S_k^+$
	$$
	\bbf{0} \preceq \E \left[ (\bbf{X} - \langle \bbf{x} \rangle_{\bbf{q}})(\bbf{X} - \langle \bbf{x} \rangle_{\bbf{q}})^{\intercal} \right] 
	= \E[\bbf{X} \bbf{X}^{\intercal}] - F_{P_X}(\bbf{q}) \,.
	$$
	This proves the first part of~(\ref{item:range}). The second part follows from~(\ref{item:F_inc}) and (\ref{item:lim_F}), that we prove now.
	Let $\bbf{q} \succ \bbf{0}$ and apply Lemma~\ref{lem:mmse_opt} with $f(\bbf{Y}) = \bbf{q}^{-1/2} \bbf{Y}$:
	$$
	\bbf{0} \preceq 
	\E[\bbf{X}\bbf{X}^{\intercal}] - F_{P_X}(\bbf{q})
	\preceq 
	\E \left[ (\bbf{X} - \bbf{q}^{-1/2} \bbf{Y})(\bbf{X} - \bbf{q}^{-1/2}\bbf{Y})^{\intercal} \right] 
	= \bbf{q}^{-1} \xrightarrow[\bbf{q} \to \infty]{} \bbf{0}.
	$$
\end{proof}

\begin{lemma} \label{lem:trace_convex}
	Let $\bbf{A}$ and $\bbf{B}$ be two symmetric matrices. Suppose that $\bbf{B}$ is semidefinite positive. Then
	$$
	\Tr\big[(\bbf{A}\bbf{B})^2\big] \geq 0 \,.
	$$
	Moreover, if $\bbf{B} \succ \bbf{0}$ we have equality if and only if $\bbf{A} = \bbf{0}$.
\end{lemma}

\begin{proof}
	$\bbf{B}$ is semidefinite positive, so it admits a square root $\bbf{B}^{1/2}$. Define $\bbf{C}=\bbf{B}^{1/2} \bbf{A} \bbf{B}^{1/2}$. Then
	$$
	\Tr\big[(\bbf{A}\bbf{B})^2\big] 
	= \Tr\big[\bbf{A}\bbf{B}\bbf{A}\bbf{B}\big] 
	= \Tr\big[\bbf{A}\bbf{B}^{1/2} \bbf{B}^{1/2} \bbf{A} \bbf{B}^{1/2} \bbf{B}^{1/2}\big] 
	= \Tr\big[\bbf{B}^{1/2} \bbf{A} \bbf{B}^{1/2} \bbf{B}^{1/2} \bbf{A} \bbf{B}^{1/2} \big] 
	= \Tr \big[\bbf{C} \bbf{C}^{\intercal} \big] \geq 0 \,.
	$$
\end{proof}

\subsection{Fixed points equations}\label{sec:gamma_non_empty}

\begin{proposition}\label{prop:gamma_non_empty}
	The set
	\begin{equation}
		\Gamma(\lambda,\alpha) = \left\{
			(\bbf{q}_u,\bbf{q}_v) \in (S_k^+)^2 \ \big| \
			\bbf{q}_u = F_{P_U}(\lambda \alpha \bbf{q}_v) \text{ and }
			\bbf{q}_v = F_{P_V}(\lambda \bbf{q}_u)
		\right\} 
	\end{equation}
	is non-empty.
\end{proposition}

\begin{proof}
 $F_{P_U}$ and $F_{P_V}$ take their values (by Lemma~\ref{lem:general_convex}~(\ref{item:range})) in 
$$
C = \big\{ \bbf{M} \in S_k^+ \ | \ \bbf{M} \preceq \E_{P_U}[\bbf{U} \bbf{U}^{\intercal}] + \E_{P_V}[\bbf{V} \bbf{V}^{\intercal}] \big\} \,,
$$
which is convex and compact.
The function $f: (\bbf{q}_u,\bbf{q}_v) \mapsto (F_{P_U}(\lambda \alpha \bbf{q}_v),F_{P_V}(\lambda \bbf{q}_u))$ is continuous from $C$ to $C$. Brouwer's Theorem gives the existence of a fixed point of $f$: $\Gamma(\lambda,\alpha) \neq \emptyset$.
\end{proof}

\subsection{The min-max formula} \label{sec:min_max}

Recall that $\Gamma(\lambda,\alpha)$ is defined by Definition~\ref{def:gamma1} (for $k=1$) and Definition~\ref{def:gammak} (for $k \geq 1$).

\begin{proposition} \label{prop:min_max}
	Suppose that $\Cov(\bbf{V}) \succ \bbf{0}$. Then
	\begin{equation}\label{eq:sup_inf}
		\sup_{(\bbf{q}_u,\bbf{q}_v) \in \Gamma(\lambda,\alpha)} \mathcal{F}(\lambda,\alpha,\bbf{q}_u,\bbf{q}_v)
		=
		\sup_{\bbf{q}_v \in S_k^+} \inf_{\bbf{q}_u \in S_k^+} \mathcal{F}(\lambda,\alpha,\bbf{q}_u,\bbf{q}_v)
	\end{equation}
	Moreover, these extrema are achieved over the same couples $(\bbf{q}_u,\bbf{q}_v) \in \Gamma(\lambda,\alpha)$.
\end{proposition}

\begin{proof}
	$\Gamma(\lambda,\alpha)$ is a compact set. Let $(\bbf{q}_u^*,\bbf{q}_v^*) \in \Gamma(\lambda,\alpha)$ that achieves the supremum of the left-hand side of~\eqref{eq:sup_inf}. The function $\bbf{q}_u \in S_k^+ \mapsto \mathcal{F}(\lambda,\alpha,\bbf{q}_u,\bbf{q}_v^*)$ is convex (by Lemma~\ref{lem:general_convex}) and his gradient at $\bbf{q}_u^*$ is equal to
	$$
	\frac{\lambda \alpha}{2}(F_{P_V}(\lambda \bbf{q}_u^*) - \bbf{q}_v^*)=0 \,,
	$$
	because $(\bbf{q}_u^*,\bbf{q}_v^*) \in \Gamma(\lambda,\alpha)$. Thus $\displaystyle \mathcal{F}(\lambda,\alpha,\bbf{q}_u^*,\bbf{q}_v^*) = \inf_{\bbf{q}_u \in S_k^+} \mathcal{F}(\lambda,\alpha,\bbf{q}_u,\bbf{q}_v^*) \leq \sup_{\bbf{q}_v \in S_k^+} \inf_{\bbf{q}_u \in S_k^+} \mathcal{F}(\lambda,\alpha,\bbf{q}_u,\bbf{q}_v)$.
	\\

	We will denote $\Sigma_V = \E_{P_V}[\bbf{V}\bbf{V}^{\intercal}]$. To prove the converse inequality, we will first show that on can restrict the supremum in $\bbf{q}_v$ on the compact set
	$$
	K_V = \left\{ \bbf{q}_v \in S_k^+ \ | \ \bbf{q}_v \preceq \Sigma_{V} \right\} \,.
	$$
	Indeed, if $\bbf{q}_v \in S_k^+ \setminus K_V$, then there exists $\bbf{v}\in \R^k \setminus \{0\}$ such that $\bbf{v}^{\intercal} \bbf{q}_v \bbf{v} > \bbf{v}^{\intercal} \Sigma_V \bbf{v}$. We define, for $t >0$
	$$
	\phi(t) = \mathcal{F}(\lambda,\alpha, t \bbf{v} \bbf{v}^{\intercal},\bbf{q}_v) \,.
	$$
	$\phi$ is differentiable over $\R_+^*$ and for $t>0$
	$$
	\phi'(t) = \frac{\lambda \alpha}{2} \Tr[F_{P_V}(\lambda t \bbf{v} \bbf{v}^{\intercal} ) \bbf{v}\bbf{v}^{\intercal} - \bbf{q}_v \bbf{v} \bbf{v}^{\intercal}]
	= \frac{\lambda \alpha}{2} \left(
		\bbf{v}^{\intercal} F_{P_V}(t \bbf{v} \bbf{v}^{\intercal} ) \bbf{v} - \bbf{v}^{\intercal} \bbf{q}_v \bbf{v}
	\right)
	\leq \frac{\lambda \alpha}{2} \bbf{v}^{\intercal} (\Sigma_V - \bbf{q}_v) \bbf{v} \,.
	$$
	Therefore, $\phi(t) \xrightarrow[ t \to +\infty ]{} - \infty$ and $\displaystyle \inf_{\bbf{q}_u \in S_k^+} \mathcal{F}(\lambda,\alpha,\bbf{q}_u,\bbf{q}_v) = -\infty$. We can thus restrict the supremum to $K_V$.
	\\

	For $\bbf{q}_v \in S_k^+$ we define $\phi_{\bbf{q}_v}: \bbf{q}_u \in S_k^+ \mapsto \mathcal{F}(\lambda,\alpha,\bbf{q}_u,\bbf{q}_v)$ and $f(\bbf{q}_v) = \inf_{\bbf{q}_u} \phi_{\bbf{q}_v}(\bbf{q}_u)$.
	Let $\stackrel{\circ}{K_V}$ denote the interior of $K_V$ in $S_k^+$, that is
	$$
	\stackrel{\circ}{K_V}
	= \left\{ \bbf{M} \in S_k^+ \ \middle| \ \bbf{M} \prec \Sigma_V \right\} \,.
	$$
	\begin{lemma} \label{lem:unique_minimizer}
		\begin{itemize}
			\item If $\bbf{q}_v \in \stackrel{\circ}{K_V}$, then $\displaystyle \inf_{\bbf{q}_u \in S_k^+} \mathcal{F}(\lambda,\alpha, \bbf{q}_u, \bbf{q}_v)$ is achieved at a unique $\bbf{q}_u^*(\bbf{q}_v) \in S_k^+$. Moreover,
				\begin{equation} \label{eq:gradient_phi}
					F_{P_V}(\lambda \bbf{q}_u^*(\bbf{q}_v)) \succeq \bbf{q}_v
					\qquad
					\text{and}
					\qquad
					\Tr\big[\bbf{q}_u^*(F_{P_V}(\lambda \bbf{q}_u^*(\bbf{q}_v)) - \bbf{q}_v)\big] = 0.
				\end{equation}
			\item  The function $\displaystyle f: \bbf{q}_v \in \stackrel{\circ}{K_V} \mapsto \inf_{\bbf{q}_u \in S_k^+} \mathcal{F}(\lambda,\alpha, \bbf{q}_u, \bbf{q}_v)$ is differentiable, with gradient given by
				\begin{equation} \label{eq:gradient_f}
					\nabla f (\bbf{q}_v) = \frac{\lambda \alpha}{2} ( F_{P_U}(\lambda \alpha \bbf{q}_v) - \bbf{q}_u^*(\bbf{q}_v)) \,.
				\end{equation}
		\end{itemize}
	\end{lemma}
	\begin{proof} Let $\bbf{q}_v \in \stackrel{\circ}{K_V}$ and define $\phi_{\bbf{q}_v}:  \bbf{q}_u \mapsto \mathcal{F}(\lambda,\alpha, \bbf{q}_u, \bbf{q}_v)$. $\Cov(\bbf{V}) \succ \bbf{0}$ thus by Lemma~\ref{lem:general_convex}, $\phi_{\bbf{q}_v}$ is strictly convex with gradient
		$$
		\nabla \phi_{\bbf{q}_v}(\bbf{q}_u) = \frac{\lambda}{2}(F_{P_V}(\lambda \bbf{q}_u) - \bbf{q}_v) \xrightarrow[\bbf{q}_u \to \infty]{} \frac{\lambda}{2}(\Sigma_V - \bbf{q}_v) \succ 0 \,.
		$$
		Consequently, $\phi_{\bbf{q}_v}(\bbf{q}_u) \xrightarrow[\bbf{q}_u \to \infty]{} + \infty$. $\phi_{\bbf{q}_v}$ admits therefore a unique minimizer $\bbf{q}_u^*(\bbf{q}_v)$.~\eqref{eq:gradient_phi} follows from the optimality conditions at $\bbf{q}_u^*(\bbf{q}_v)$.
		\\

		Let $\bbf{Q}_v \in \stackrel{\circ}{K_V}$. We are going to show that $f$ is differentiable on $\{ \bbf{q}_v \in S_k^+ \ | \ \bbf{q}_v \prec \bbf{Q}_v\}$. $\bbf{Q}_v \prec \Sigma_V$ so by Lemma~\ref{lem:general_convex} we can find $\bbf{Q}_u \in S_k^+$ such that for all $\bbf{q}_u \succeq \bbf{Q}_u$, $F_{P_V}(\lambda \bbf{q}_u) \succ \bbf{Q}_v$.
		Let now $\bbf{0} \preceq \bbf{q}_v \prec \bbf{Q}_v$. For all $\bbf{q}_u \succeq \bbf{Q}_u$
		$$
		\nabla \phi_{\bbf{q}_v}(\bbf{q}_u) = \frac{\lambda\alpha}{2}(F_{P_V}(\lambda \bbf{q}_u) - \bbf{q}_v) \succ \bbf{0} \,.
		$$
		Consequently, $\bbf{q}_u^* \preceq \bbf{Q}_u$. For $\bbf{q}_v \in  \{ \bbf{q}_v \in S_k^+ \ | \ \bbf{q}_v \prec \bbf{Q}_v\}$, we have shown that the infimum is achieved at a unique point of a compact set. Thus, by an ``envelope theorem'' (Corollary~4 from~\cite{milgrom2002envelope}), $f$ is differentiable on $\{ \bbf{q}_v \in S_k^+ \ | \ \bbf{q}_v \prec \bbf{Q}_v\}$ with gradient given by~\eqref{eq:gradient_f}. The lemma follows.
	\end{proof}
	\\

	Let now 
	$$
	D_V = 
	\left\{
	\bbf{q}_v \in S_k^+ \ \middle| \ \inf_{\bbf{q}_u \in S_k^+} \mathcal{F}(\lambda,\alpha, \bbf{q}_u, \bbf{q}_v) \text{ is finite} \right\} \,.
	$$
	From what we have seen until now, $\stackrel{\circ}{K_V} \subset D_V \subset K_V$. Notice that $f: \bbf{q}_v \mapsto \inf_{\bbf{q}_u \in S_k^+} \mathcal{F}(\lambda,\alpha, \bbf{q}_u, \bbf{q}_v)$ is continuous over $D_V$. Indeed for $\bbf{q}_v \in D_V$
	$$
	\inf_{\bbf{q}_u \in S_k^+} \mathcal{F}(\lambda,\alpha, \bbf{q}_u, \bbf{q}_v) = \psi_{P_U}(\lambda \alpha \bbf{q}_v) + \inf_{\bbf{q}_u \in S_k^+} \left\{ \alpha \psi_{P_V}(\lambda\bbf{q}_u) - \frac{\alpha\lambda}{2} \Tr[\bbf{q}_v \bbf{q}_u] \right\} \,,
	$$
	and the second term of the right-hand side is a concave function of $\bbf{q}_v$ (as an infimum of linear functions) and is thus continuous. 
	Moreover, one verify easily that if $f(\bbf{q}_v)=-\infty$ for some $\bbf{q}_v \in K_V$, then $f(\bbf{q}_v') \xrightarrow[\bbf{q}_v' \in D_V \to \bbf{q}_v]{} -\infty$. Consequently the supremum of $f$ is achieved at some $\bbf{q}_v^* \in D_V$.
	\\

	{\bfseries Case 1:} $\bbf{q}_v^* \in \stackrel{\circ}{K_V}$.
	\\
	By Lemma~\ref{lem:unique_minimizer} above, the minimum of $\phi_{\bbf{q}_v^*}$ is thus achieved at a unique $\bbf{q}_u^* \in S_k^+$. The optimality condition of $f$ at $\bbf{q}_v^*$ gives:
	$$
	\frac\alpha {\lambda}{2}(F_{P_U}(\lambda \alpha \bbf{q}_v^*) - \bbf{q}_u^*) = 
	\nabla f(\bbf{q}_v^*) \preceq 0
	\quad \text{and} \quad \Tr[\bbf{q}_v^* \nabla f(\bbf{q}_v^*)] = 0 \,.
	$$
	Compute now
	\begin{align*}
		\phi_{\bbf{q}_v^*}(F_{P_U}(\lambda \alpha \bbf{q}_v^*))
		&= \psi_{P_U}(\lambda \alpha \bbf{q}_v^*) + \alpha \psi_{P_V}(\lambda F_{P_U}(\lambda \bbf{q}_v^*)) - \frac{\lambda \alpha}{2} \Tr[\bbf{q}_v^* F_{P_U}(\lambda \bbf{q}_v^*)]
		\\
		&\leq \psi_{P_U}(\lambda \alpha \bbf{q}_v^*) + \alpha \psi_{P_V}(\lambda \bbf{q}_u^*)  - \frac{\lambda \alpha}{2} \Tr[\bbf{q}_v^* \bbf{q}_u^*] + \frac{\lambda \alpha}{2} \Tr[\bbf{q}_v^* (\bbf{q}_u^* - F_{P_U}(\lambda \alpha \bbf{q}_v^*))] 
		&& (\text{because} \ \bbf{q}_u^* \succeq F_{P_U}(\lambda \alpha \bbf{q}_v^*) \, )
		\\
		&= \phi_{\bbf{q}_v^*}( \bbf{q}_u^*) && (\text{using~\eqref{eq:gradient_phi}}).
	\end{align*}
	By Lemma~\ref{lem:unique_minimizer}, $\bbf{q}_u^*$ is the unique minimizer of $\phi_{\bbf{q}_v^*}$, therefore $\bbf{q}_u^* = F_{P_U}(\lambda \alpha \bbf{q}_v^*)$. By~\eqref{eq:gradient_phi} we have $F_{P_V}(\lambda \bbf{q}_u^*) \succeq \bbf{q}_v^*$. Suppose that $\bbf{q}_v^* \neq F_{P_V}(\lambda \bbf{q}_u^*)$. Then by monotonicity of $\psi_{P_U}$ (Lemma~\ref{lem:general_convex}) we have $\psi_{P_U}(\lambda \alpha\bbf{q}_v^*) < \psi_{P_U}(\lambda \alpha F_{P_V}(\lambda \bbf{q}_u^*))$. Thus
	\begin{align*}
		f(F_{P_V}(\lambda \bbf{q}_u^*)) 
		&> \psi_{P_U}(\lambda \alpha \bbf{q}_v^*) + \inf_{\bbf{q}_u \in S_k^+} \alpha \psi_{P_V}(\lambda \bbf{q}_u) - \frac{\alpha \lambda}{2} \Tr[F_{P_V}(\lambda \bbf{q}_u^*) \bbf{q}_u]
		\\
		&=\psi_{P_U}(\lambda \alpha \bbf{q}_v^*) + \inf_{\bbf{q}_u \in S_k^+} \alpha \psi_{P_V}(\lambda \bbf{q}_u) - \frac{\lambda\alpha}{2} \Tr[\bbf{q}_v^* \bbf{q}_u] && \text{(because of~\eqref{eq:gradient_phi})}
		\\
		&= f(\bbf{q}_v^*)\,,
	\end{align*}
	which is absurd because $\bbf{q}_v^*$ maximizes $f$. We conclude that $\bbf{q}_v^* = F_{P_V}(\lambda \bbf{q}_u^*)$ and $(\bbf{q}_u^*,\bbf{q}_v^*) \in \Gamma(\lambda,\alpha)$ and therefore
	$$
	\sup_{\bbf{q}_v \in S_k^+ } \inf_{\bbf{q}_u \in S_k^+} \mathcal{F}(\lambda,\alpha,\bbf{q}_u,\bbf{q}_v)
	= \mathcal{F}(\lambda,\alpha,\bbf{q}_u^*,\bbf{q}_v^*)
	\leq
	\sup_{(\bbf{q}_v,\bbf{q}_u) \in \Gamma(\lambda,\alpha)} \mathcal{F}(\lambda,\alpha,\bbf{q}_u,\bbf{q}_v) \,.
	$$

	\vspace{0.3cm}
	{\bfseries Case 2:} $\bbf{q}_v^* \in D_V \setminus \stackrel{\circ}{K_V}$.
	\\
	Let $(\delta_n)_n$ be an increasing positive sequence that converges to $1$. For $n \geq \N$ we define $\bbf{q}_v^{(n)} = \delta_n \bbf{q}_v^* \in \stackrel{\circ}{K_V}$.
	Let $\bbf{q}_u^{(n)} \in S_k^+$ be the minimum of $\phi_{\bbf{q}_v^{(n)}}$. By continuity of $f$ at $\bbf{q}_v^*$:
	$$
	\mathcal{F}(\lambda,\alpha,\bbf{q}_u^{(n)},\bbf{q}_v^{(n)}) \xrightarrow[n \to \infty]{} \sup_{\bbf{q}_v \in S_k^+ } \inf_{\bbf{q}_u \in S_k^+} \mathcal{F}(\lambda,\alpha,\bbf{q}_u,\bbf{q}_v)\,.
	$$
	Let $(\bbf{e}_1^{(n)}, \dots, \bbf{e}_k^{(n)}) \in (\R^k)^k$ be an orthonormal basis of eigenvectors of $\bbf{q}_u^{(n)}$ and $(\mu_1^{(n)}, \dots, \mu_k^{(n)}) \in \R_+^k$ be the associated eigenvalues. 
	Without loss of generalities, one can assume that 
	$ (\bbf{e}_1^{(n)},\dots, \bbf{e}_k^{(n)}) $ converges to a orthonormal basis $(\bbf{e}_1, \dots, \bbf{e}_k)$, and $(\mu_1^{(n)}, \dots, \mu_k^{(n)})$ converges to $(\mu_1, \dots, \mu_k)$ where $\mu_1 \geq \mu_2 \geq \dots \geq \mu_k$ are in $\R_+ \cup \{+\infty \}$. 
	We also denote $\bbf{R}^{(n)} = (\bbf{e}_1^{(n)} | \dots | \bbf{e}_k^{(n)}) \xrightarrow[n \to \infty]{} \bbf{R} = (\bbf{e}_1| \dots | \bbf{e}_k)$.
	Suppose that $\mu_1 < + \infty$. Then $\bbf{q}_u^{(n)} \xrightarrow[n \to \infty]{} \bbf{R} \bbf{Diag}(\mu_1, \dots, \mu_k) \bbf{R}^{\intercal}$ and $F_{P_V}(\lambda \bbf{q}_u^{(n)}) \xrightarrow[n \to \infty]{} F_{P_V}(\lambda \bbf{R} \bbf{Diag}(\mu_1, \dots, \mu_k) \bbf{R}^{\intercal})$. Equation~\eqref{eq:gradient_phi} and Lemma~\ref{lem:general_convex} give then
	$$
	\Sigma_V \succ F_{P_V}(\lambda \bbf{R} \bbf{Diag}(\mu_1, \dots, \mu_k) \bbf{R}^{\intercal}) \succeq \bbf{q}_v^* \,,
	$$
	which is absurd. Therefore, there exists $r \in \{1, \dots, k \}$ such that $\mu_1 = \dots = \mu_r =+ \infty$ and $\mu_{r+1}, \dots, \mu_k < \infty$.
	\begin{lemma}	
	For $1 \leq i \leq r$
	$$
	\bbf{e}_i^{(n)\intercal} F_{P_V}(\lambda \bbf{q}_u^{(n)}) \bbf{e}_i^{(n)}
	\xrightarrow[n \to \infty]{}
	\bbf{e}_i^{\intercal} \Sigma_V \bbf{e}_i > 0 \,.
	$$
	\end{lemma}	
	\begin{proof}
		By Lemma~\ref{lem:mmse_opt} we have
		\begin{equation} \label{eq:comm}
		\bbf{0} \preceq \Sigma_V - F_{P_V}(\lambda \bbf{q}_u^{(n)}) \preceq \E \Big[(\bbf{V} - h(\bbf{Y}))(\bbf{V} - h(\bbf{Y}))^{\intercal}\Big],
	\end{equation}
		for all measurable function $h$,
		where the last expectation is with respect $\bbf{V}$ and $\bbf{Y} = (\bbf{q}_u^{(n)})^{1/2} \bbf{V} + \bbf{Z}$, where $(\bbf{V},\bbf{Z}) \sim P_V \otimes \cN(\bbf{0},\bbf{I}_k)$.
		Let $i \in \{1, \dots, r\}$ and let us chose $h(\bbf{Y}) = \frac{1}{\sqrt{\mu_i^{(n)}}} \bbf{Y}$. Compute
		\begin{align*}
		\bbf{e}_i^{(n)\intercal} \E \Big[(\bbf{V} - h(\bbf{Y}))(\bbf{V} - h(\bbf{Y}))^{\intercal}\Big] \bbf{e}_i^{(n)}
		&=
		\E \Big[\big(\bbf{V}^{\intercal} \bbf{e}_i^{(n)} - (\mu_i^{(n)})^{-1/2} \bbf{V}^{\intercal} (\bbf{q}_u^{(n)})^{1/2} \bbf{e}_i^{(n)} -(\mu_i^{(n)})^{-1/2} \bbf{Z}^{\intercal} \bbf{e}_i^{(n)} \big)^2\Big]
		\\
		&= \frac{1}{\mu_i^{(n)}} \xrightarrow[n \to \infty]{} 0 \,,
		\end{align*}
		where we used the fact that $\bbf{e}_i^{(n)}$ is an eigenvectors of $(\bbf{q}_u^{(n)})^{1/2}$ associated with the eigenvalue $\sqrt{\mu_i^{(n)}}$. We conclude using Equation~\ref{eq:comm}.
	\end{proof}
\\

	Using the Lemma above
	$$
	\Tr \left[
		\bbf{q}_u^{(n)} F_{P_V}(\lambda \bbf{q}_u^{(n)})
	\right]
	=
	\sum_{i=1}^k \mu_{i}^{(n)} \left(\bbf{e}_i^{(n)\intercal} F_{P_V}(\lambda \bbf{q}_u^{(n)}) \bbf{e}_i^{(n)}\right)
	\xrightarrow[n \to \infty]{}
	+ \infty \,.
	$$
	By~\eqref{eq:gradient_phi} we have
	$$
	(1-\delta_n) \Tr[\bbf{q}_v^* \bbf{q}_u^{(n)}] = \Tr[\bbf{q}_v^{(n)} \bbf{q}_u^{(n)}] = 
	\Tr \left[
		\bbf{q}_u^{(n)} F_{P_V}(\lambda \bbf{q}_u^{(n)})
	\right],
	$$
	and we conclude that $\Tr[\bbf{q}_v^* \bbf{q}_u^{(n)}] \xrightarrow[n \to \infty]{} + \infty$. 
	Recall that $\displaystyle g: \bbf{q}_v \mapsto \inf_{\bbf{q}_u \in S_k^+} \psi_{P_V}(\lambda \bbf{q}_u) - \frac{\lambda}{2} \Tr[\bbf{q}_v \bbf{q}_u]$ is concave on its domain $D_V$. 
	Since $\bbf{q}_v^{(n)} \in \stackrel{\circ}{K_V}$, we have by Lemma~\ref{lem:unique_minimizer}, $\displaystyle \nabla g (\bbf{q}_v^{(n)}) = - \frac{\lambda}{2} \bbf{q}_u^{(n)}$.
	By concavity we have then
	$$
	g(\bbf{q}_v^*) \leq g(\bbf{q}_v^{(n)}) - \frac{\lambda}{2} \Tr[(\bbf{q}_v^* - \bbf{q}_v^{(n)}) \bbf{q}_u^{(n)}]
	= g(\bbf{q}_v^{(n)}) - \frac{\lambda}{2}\delta_n \Tr[\bbf{q}_v^* \bbf{q}_u^{(n)}] \,.
	$$
	Consequently
	$$
	f(\bbf{q}_v^{(n)}) = \alpha g(\bbf{q}_v^{(n)}) + \psi_{P_U}(\lambda \bbf{q}_v^{(n)})
	\geq
	\alpha g(\bbf{q}_v^*) + \frac{\alpha \lambda}{2}\delta_n \Tr[\bbf{q}_v^* \bbf{q}_u^{(n)}] + \psi_{P_U}(\lambda \alpha \bbf{q}_v^*(1 - \delta_n)) \,.
	$$
	$\psi_{P_U}$ has bounded gradient and is thus $L$-Lipschitz for some constant $L>0$. We have then
	\begin{align*}
		f(\bbf{q}_v^{(n)})
		&\geq
		\alpha g(\bbf{q}_v^*) + \frac{\alpha \lambda}{2}\delta_n \Tr[\bbf{q}_v^* \bbf{q}_u^{(n)}] + \psi_{P_U}(\lambda \alpha \bbf{q}_v^*) - L\lambda \alpha \|\bbf{q}_v^*\|\delta_n
		\\
		&\geq
		f(\bbf{q}_v^*) + \delta_n \lambda \alpha \left(\frac{1}{2} \Tr[\bbf{q}_v^* \bbf{q}_u^{(n)}]  - L\|\bbf{q}_v^*\| \right)
		\\
		&> f(\bbf{q}_v^*) \,,
	\end{align*}
	for $n$ large enough. This is absurd. We conclude that we can not have $\bbf{q}_v^* \in D_V \setminus \stackrel{\circ}{K_V}$.
\end{proof}

\section{Proofs of the decorrelation principles}

\subsection{Proof of Theorem~\ref{th:overlap_concentration}}\label{sec:proof_overlap_concentration}

Define for $\bbf{x} \in \R^n$
$$
U(\bbf{x}) 
= \frac{1}{n s_n} \frac{\partial}{\partial a} h_{n,a}(\bbf{x})
= \frac{1}{n} \sum_{i=1}^n \frac{1}{\sqrt{s_n}} Z_i x_i + 2 a x_i X_i - a x_i^2 \,.
$$
\begin{lemma}\label{lem:concentration_energy}
	Under the conditions of Theorem~\ref{th:overlap_concentration},
	$$
	\int_{1}^2 \E \Big\langle \big| U(\bbf{x}) - \E \langle U(\bbf{x}) \rangle_{n,a} \big| \Big\rangle_{\!n,a} da \xrightarrow[n \to \infty]{} 0 \,.
	$$
\end{lemma}
Before proving Lemma~\ref{lem:concentration_energy}, let us show how it implies Theorem~\ref{th:overlap_concentration}.

\begin{proof}[of Theorem~\ref{th:overlap_concentration}]
	By the bounded support assumption on $P$, the overlap between two replicas is bounded by $K^2$, thus
	\begin{equation}\label{eq:gg_trick}
		\left| \E \!\left\langle U(\bbf{x}^{(1)}) \, \bbf{x}^{(1)}\! .\bbf{x}^{(2)} \right\rangle_{\!n,a} 
		\!\!- \E\!\left\langle \bbf{x}^{(1)}\! .\bbf{x}^{(2)} \right\rangle_{\!n,a} \!\! \E\! \left\langle U(\bbf{x}^{(1)}) \right\rangle_{\!n,a} \right| 
		\leq K^2  \E\! \left\langle \big| U(\bbf{x}) - \E\! \left\langle U(\bbf{x}) \right\rangle_{\!n,a} \big| \right\rangle_{\!n,a}.
	\end{equation}
	Let us compute the left-hand side of~\eqref{eq:gg_trick}.
	By Gaussian integration by parts and using the Nishimori identity (Proposition~\ref{prop:nishimori}) we get $\E \big\langle U(\bbf{x}^{(1)}) \big\rangle_{\!n,a} = 2 a \, \E \big\langle \bbf{x}^{(1)} \! . \bbf{x}^{(2)} \big\rangle_{\!n,a}$. Therefore
	$$
	\E\left\langle \bbf{x}^{(1)}\! . \bbf{x}^{(2)} \right\rangle_{\!n,a} \E \left\langle U(\bbf{x}^{(1)}) \right\rangle_{\!n,a} 
	= 2 a \left(\E \left\langle \bbf{x}^{(1)} \! . \bbf{x}^{(2)} \right\rangle_{\!n,a}\right)^{\!2}.
	$$
	Using the same tools, we compute
	\begin{align*}
		\E&  \big\langle U(\bbf{x}^{(1)})(\bbf{x}^{(1)}\!.\bbf{x}^{(2)}) \big\rangle_{\!n,a} 
		\\
		&= 2a \E \big\langle (\bbf{x}^{(1)}\!.\bbf{X})(\bbf{x}^{(1)}\!.\bbf{x}^{(2)}) \big\rangle_{\!n,a} + \frac{1}{n \sqrt{s_n}} \sum_{i=1}^n \E Z_i \big\langle x^{(1)}_i (\bbf{x}^{(1)}\!.\bbf{x}^{(2)}) \big\rangle_{\!n,a} - \frac{a}{n} \sum_{i=1}^n \E \big\langle (x^{(1)}_i)^2 (\bbf{x}^{(1)}\!.\bbf{x}^{(2)}) \big\rangle_{\!n,a} \\
		&= 2a \E \big\langle (\bbf{x}^{(1)}\!.\bbf{X})(\bbf{x}^{(1)}\!.\bbf{x}^{(2)}) \big\rangle_{\!n,a} + a \E \big\langle (\bbf{x}^{(1)}\!.\bbf{x}^{(2)})^2 \big\rangle_{\!n,a} - a \E \big\langle (\bbf{x}^{(1)}\!.\bbf{x}^{(3)} + \bbf{x}^{(1)}\!.\bbf{x}^{(4)})(\bbf{x}^{(1)}\!.\bbf{x}^{(2)})\big\rangle_{\!n,a} \\
		&= 2a \E \big\langle (\bbf{x}^{(1)}\!.\bbf{x}^{(2)})^2 \big\rangle_{\!n,a} \,.
	\end{align*}
	Thus, by~\eqref{eq:gg_trick} we have for all $a \in [1,2]$
	$$
	\E \Big\langle \big(\bbf{x}^{(1)}\!.\bbf{x}^{(2)} - \E\big\langle \bbf{x}^{(1)}\!.\bbf{x}^{(2)} \big\rangle_{\!n,a}\big)^2 \Big\rangle_{\!n,a}
	\leq
	\frac{K^2}{2}  \E \left\langle \big| U(\bbf{x}) - \E \left\langle U(\bbf{x}) \right\rangle_{\!n,a} \! \big| \right\rangle_{\!n,a} ,
	$$
	and we conclude by integrating with respect to $a$ over $[1,2]$ and using Lemma~\ref{lem:concentration_energy}.
\end{proof}
\\

\begin{proof}[of Lemma~\ref{lem:concentration_energy}]
	$\phi$ is twice differentiable on $[1/2,3]$, and for $a \in [1/2,3]$
	\begin{align}
		\phi'(a) &= \langle U(\bbf{x}) \rangle_{n,a} \,, \\
		\phi''(a) &= n s_n \big(\langle U(\bbf{x})^2 \rangle_{n,a} -\langle U(\bbf{x}) \rangle_{n,a}^2\big)
		+ \frac{1}{n} \sum_{i=1}^n \Big\langle 2 x_i X_i - x_i^2 \Big\rangle_{n,a} \,. \label{eq:der_sec_phi}
	\end{align}
	Thus $\big\langle (U(\bbf{x}) - \langle U(\bbf{x})\rangle_{n,a})^2 \big\rangle_{n,a} \leq \frac{1}{n s_n} (\phi''(a) + 2 K^2)$ and by integration with respect to $a \in [1,2]$,
	\begin{align*}
		\int_{1}^2 \E \big\langle (U(\bbf{x}) - \langle U(\bbf{x})\rangle_{n,a})^2 \big\rangle_{n,a} da \leq \frac{1}{n s_n} \left(\E \phi'(2) - \E \phi'(1) + 2 K^2\right) = O(n^{-1} s_n^{-1}) \,,
	\end{align*}
	because $\E \phi'(a) = 2 a \E \langle \bbf{x} .\bbf{X} \rangle_{n,a}$.
	Hence $\int_{1}^2 \E \big\langle |U(\bbf{x}) - \langle U(\bbf{x})\rangle_{n,a}| \big\rangle_{\!n,a} da \xrightarrow[n \to \infty]{} 0$. It remains to show that $\int_{1}^2 \E \big| \langle U(\bbf{x}) \rangle_{n,a} - \E \langle U(\bbf{x})\rangle_{n,a}\big| da \xrightarrow[n \to \infty]{} 0$.
	\\

	We will use the following lemma on convex functions (from~\cite{panchenko2013SK}, Lemma 3.2).
	\begin{lemma}
		If $f$ and $g$ are two differentiable convex functions then, for any $b >0$
		$$
		|f'(a) - g'(a)| \leq g'(a+b) - g'(a-b) + \frac{d}{b} \,,
		$$
		where $d = |f(a+b) - g(a+b)| + |f(a-b) - g(a-b)| + |f(a) - g(a)|$.
	\end{lemma}
	We apply this lemma to $\lambda \mapsto \phi(\lambda) + \frac{3}{2} K^2 \lambda^2$ and $\lambda \mapsto \E \phi(\lambda) + \frac{3}{2} K^2 \lambda^2$ that are convex because of~\eqref{eq:der_sec_phi} and the bounded support assumption on $P$.
	Therefore, for all $a \in [1,2]$ and $b \in (0,1/2)$ we have
	\begin{equation} \label{eq:appli_lem_convex}
		\E |\phi'(a) - \E \phi'(a)| \leq \E \phi'(a+b) - \E \phi'(a-b) + 6 K^2 b + \frac{3 v_n(s_n)}{b} \,.
	\end{equation}
	Notice that for all $a \in [1/2,3], \ |\E \phi'(a) | = | 2a \E \langle \bbf{x}.\bbf{X} \rangle_{n,a} | \leq 6 K^2$. Therefore, by the mean value theorem
	\begin{align*}
		\int_1^2 \big(\E \phi'(a+b) - \E \phi'(a-b)\big) da 
		&= \big(\E \phi(b+2) - \E \phi(2-b)\big) + \big(\E \phi(1-b) - \E \phi(1+b)\big) \\
		&= \big(\E \phi(b+2) - \E \phi(b+1)\big) - \big(\E \phi(2-b) - \E \phi(1-b)\big) \\
		&\leq 24 K^2 b \,.
	\end{align*}
	Combining this with equation~\eqref{eq:appli_lem_convex}, we obtain
	\begin{equation} \label{eq:control_b}
		\forall b \in (0,1/2), \ \int_1^2 \E | \phi'(a) - \E \phi'(a) | da  \leq C \Big(b + \frac{v_n(s_n)}{b}\Big) \,,
	\end{equation}
	for some constant $C>0$ depending only on $K$.
	The minimum of the right-hand side is achieved for $b=\sqrt{v_n(s_n)} < 1/2$ for $n$ large enough. Then,~\eqref{eq:control_b} gives
	\begin{align*}
		\int_1^2 \E \big| \langle U(\bbf{x}) \rangle_{n,a} - \E \langle U(\bbf{x}) \rangle_{n,a} \big| da
		= \int_1^2 \E | \phi'(a) - \E \phi'(a) | da
		\leq C \sqrt{v_n(s_n)} \xrightarrow[n \to \infty]{} 0 \,.
	\end{align*}
\end{proof}

\subsection{Proof of Proposition~\ref{prop:overlap_uv}} \label{sec:proof_overlap_uv}
Define
$$
\phi: (a_u,a_v) \mapsto \frac{1}{n s_n} \log \left( \sum_{\bbf{u},\bbf{v} \in S^n} P_0^{\otimes n}(\bbf{u},\bbf{v}) e^{H^{\text{(tot)}}_n(\bbf{u},\bbf{v})} \right)
\ \text{ and } \
v_n(s_n) = \sup_{1/2 \leq a_u,a_v \leq 3} \E | \phi(a_u,a_v) - \E \phi(a_u,a_v) |.
$$
\begin{lemma}\label{lem:v_n_uv}
	$$
	v_n(s_n) = O(n^{-1/2} s_n^{-1}) \,.
	$$
\end{lemma}
\begin{proof}
	Let $a_u,a_v \in [1/2,3]$.
	We first work conditionally to $(\bbf{U},\bbf{V})$, i.e.\ suppose $\bbf{U}$ and $\bbf{V}$ to be fixed and consider the function
	$$
f: \big(\bbf{Z}, \bbf{Z}^{(u)}, \bbf{Z}^{(v)}, \bbf{z}^{(u)}, \bbf{z}^{(v)}\big)
\mapsto 
\phi(a_u,a_v) \,.
	$$
	It is not difficult to verify that
	$$
	\|\nabla f\|^2 \leq C n^{-1} s_n^{-2} \,,
	$$
for some constant $C>0$ that only depend on $q_u$, $q_v$ and $K$.
	Let $\E_z$ denote the expectation with respect to the Gaussian random variables $\big(\bbf{Z}, \bbf{Z}^{(u)}, \bbf{Z}^{(v)}, \bbf{z}^{(u)}, \bbf{z}^{(v)}\big)$.
	The Gaussian Poincaré inequality (see~\cite{boucheron2013concentration} Chapter 3) gives then 
	$$
	\E_z \Big[\big(\phi(a_u,a_v) - \E_z \phi(a_u,a_v)\big)^2\Big] \leq C n^{-1} s_n^{-2} \,.
	$$
	Now we are going to show that $\E_{z} [\phi(a_u,a_v)]$ concentrates around its expectation (with respect to $\bbf{U},\bbf{V}$). $E_{z} [ \phi(a_u,a_v)]$ can be seen as a function of $(\bbf{U},\bbf{V})$. We can easily verify that this function has ``the bounded differences property'' (see~\cite{boucheron2013concentration}, section 3.2) because $\bbf{U}$ and $\bbf{V}$ have bounded support. Indeed, for $1 \leq i \leq n$,
	$$
	\Big|
	\frac{\partial}{\partial U_i} \E_{z} [\phi(a_u,a_v)]
	\Big|
	=
	\Big|
	\frac{1}{n s_n}
	\E_{z} 
	\Big\langle
		s_n a_u^2 u_i+
	\sum_{j=1}^n
		\frac{t}{n} u_i v_j V_j
	\Big\rangle_{\!\! n,a} 
	\Big|
	\leq C' n^{-1} s_n^{-1}\,,
	$$
	for some constant $C'$ depending only on $t$, $q_u$, $q_v$ and $K$. We have the same inequality for the partial derivatives with respect to the $V_j$.  
	Then Corollary~3.2 from~\cite{boucheron2013concentration} (which is a consequence of the Efron-Stein inequality) gives
	$$
	\E ( \E_{z} [\phi(a_u,a_v)] - \E [\phi(a_u,a_v)])^2 \leq C'' n^{-1} s_n^{-2} \,,
	$$
	for some constant $C''$ depending only on $t$, $q_u$, $q_v$ and $K$.
	We conclude that there exists a constant $C'''$ such that for all $a_u,a_v \in [1/2,3]$, $\E | \phi(a_u,a_v) - \E \phi(a_u,a_v) | \leq C''' n^{-1/2} s_n^{-2}$.
\end{proof}
\\

\begin{proof}[of Proposition~\ref{prop:overlap_uv}]
	The choice of $s_n=n^{-1/4}$ and Lemma~\ref{lem:v_n_uv} above implies
	$$
	\begin{cases}
		v_n(s_n) \xrightarrow[n \to \infty]{} 0\,, \\
		n s_n \xrightarrow[n \to \infty]{} + \infty\,.
	\end{cases}
	$$
	We deduce then the proposition from the fact that the proof we gave of Theorem~\ref{th:overlap_concentration} remains valid if one consider the overlap over only the first half of the components of the replicas (with a perturbation involving only the first half of the components of $\bbf{X}$). 
\end{proof}
\end{appendices}

{%
	\bibliographystyle{plain}
	\bibliography{./references.bib}

\begin{thebibliography}{10}

\bibitem{aizenman2003extended}
Michael Aizenman, Robert Sims, and Shannon~L Starr.
\newblock Extended variational principle for the sherrington-kirkpatrick
  spin-glass model.
\newblock {\em Physical Review B}, 68(21):214403, 2003.

\bibitem{baik2005phase}
Jinho Baik, G{\'e}rard Ben~Arous, and Sandrine P{\'e}ch{\'e}.
\newblock Phase transition of the largest eigenvalue for nonnull complex sample
  covariance matrices.
\newblock {\em Annals of Probability}, pages 1643--1697, 2005.

\bibitem{banks2016information}
Jess Banks, Cristopher Moore, Roman Vershynin, Nicolas Verzelen, and Jiaming
  Xu.
\newblock Information-theoretic bounds and phase transitions in clustering,
  sparse pca, and submatrix localization.
\newblock In {\em Information Theory (ISIT), 2017 IEEE International Symposium
  on}, pages 1137--1141. IEEE, 2017.

\bibitem{barbier2016mutual}
Jean Barbier, Mohamad Dia, Nicolas Macris, Florent Krzakala, Thibault Lesieur,
  and Lenka Zdeborov{\'a}.
\newblock Mutual information for symmetric rank-one matrix estimation: A proof
  of the replica formula.
\newblock In {\em Advances in Neural Information Processing Systems}, pages
  424--432, 2016.

\bibitem{barra2010replica}
Adriano Barra, Giuseppe Genovese, and Francesco Guerra.
\newblock The replica symmetric approximation of the analogical neural network.
\newblock {\em Journal of Statistical Physics}, 140(4):784--796, 2010.

\bibitem{barra2011equilibrium}
Adriano Barra, Giuseppe Genovese, and Francesco Guerra.
\newblock Equilibrium statistical mechanics of bipartite spin systems.
\newblock {\em Journal of Physics A: Mathematical and Theoretical},
  44(24):245002, 2011.

\bibitem{benaych2012singular}
Florent Benaych-Georges and Raj~Rao Nadakuditi.
\newblock The singular values and vectors of low rank perturbations of large
  rectangular random matrices.
\newblock {\em Journal of Multivariate Analysis}, 111:120--135, 2012.

\bibitem{boucheron2013concentration}
St{\'e}phane Boucheron, G{\'a}bor Lugosi, and Pascal Massart.
\newblock {\em Concentration inequalities: A nonasymptotic theory of
  independence}.
\newblock Oxford university press, 2013.

\bibitem{coja2016information}
Amin Coja-Oghlan, Florent Krzakala, Will Perkins, and Lenka Zdeborova.
\newblock Information-theoretic thresholds from the cavity method.
\newblock {\em arXiv preprint arXiv:1611.00814}, 2016.

\bibitem{deshpande2014information}
Yash Deshpande and Andrea Montanari.
\newblock Information-theoretically optimal sparse pca.
\newblock In {\em 2014 IEEE International Symposium on Information Theory},
  pages 2197--2201. IEEE, 2014.

\bibitem{donoho2009message}
David~L Donoho, Arian Maleki, and Andrea Montanari.
\newblock Message-passing algorithms for compressed sensing.
\newblock {\em Proceedings of the National Academy of Sciences},
  106(45):18914--18919, 2009.

\bibitem{edwards1976eigenvalue}
SF~Edwards and Raymund~C Jones.
\newblock The eigenvalue spectrum of a large symmetric random matrix.
\newblock {\em Journal of Physics A: Mathematical and General}, 9(10):1595,
  1976.

\bibitem{feral2007largest}
Delphine F{\'e}ral and Sandrine P{\'e}ch{\'e}.
\newblock The largest eigenvalue of rank one deformation of large wigner
  matrices.
\newblock {\em Communications in mathematical physics}, 272(1):185--228, 2007.

\bibitem{ghirlanda1998general}
Stefano Ghirlanda and Francesco Guerra.
\newblock General properties of overlap probability distributions in disordered
  spin systems. towards parisi ultrametricity.
\newblock {\em Journal of Physics A: Mathematical and General}, 31(46):9149,
  1998.

\bibitem{guerra2003broken}
Francesco Guerra.
\newblock Broken replica symmetry bounds in the mean field spin glass model.
\newblock {\em Communications in mathematical physics}, 233(1):1--12, 2003.

\bibitem{guo2005mutual}
Dongning Guo, Shlomo Shamai, and Sergio Verd{\'u}.
\newblock Mutual information and minimum mean-square error in gaussian
  channels.
\newblock {\em IEEE Transactions on Information Theory}, 51(4):1261--1282,
  2005.

\bibitem{hopfield1982neural}
John~J Hopfield.
\newblock Neural networks and physical systems with emergent collective
  computational abilities.
\newblock {\em Proceedings of the national academy of sciences},
  79(8):2554--2558, 1982.

\bibitem{iba1999nishimori}
Yukito Iba.
\newblock The nishimori line and bayesian statistics.
\newblock {\em Journal of Physics A: Mathematical and General}, 32(21):3875,
  1999.

\bibitem{javanmard2013state}
Adel Javanmard and Andrea Montanari.
\newblock State evolution for general approximate message passing algorithms,
  with applications to spatial coupling.
\newblock {\em Information and Inference}, page iat004, 2013.

\bibitem{johnstone2001distribution}
Iain~M Johnstone.
\newblock On the distribution of the largest eigenvalue in principal components
  analysis.
\newblock {\em Annals of statistics}, pages 295--327, 2001.

\bibitem{korada2009exact}
Satish~Babu Korada and Nicolas Macris.
\newblock Exact solution of the gauge symmetric p-spin glass model on a
  complete graph.
\newblock {\em Journal of Statistical Physics}, 136(2):205--230, 2009.

\bibitem{korada2010tight}
Satish~Babu Korada and Nicolas Macris.
\newblock Tight bounds on the capacity of binary input random cdma systems.
\newblock {\em IEEE Transactions on Information Theory}, 56(11):5590--5613,
  2010.

\bibitem{krzakala2016mutual}
Florent Krzakala, Jiaming Xu, and Lenka Zdeborov{\'a}.
\newblock Mutual information in rank-one matrix estimation.
\newblock In {\em Information Theory Workshop (ITW), 2016 IEEE}, pages 71--75.
  IEEE, 2016.

\bibitem{lelarge2016fundamental}
Marc Lelarge and L{\'e}o Miolane.
\newblock Fundamental limits of symmetric low-rank matrix estimation.
\newblock {\em Probability Theory and Related Fields}, Apr 2018.

\bibitem{lesieur2016uv}
Thibault Lesieur, Caterina De~Bacco, Jess Banks, Florent Krzakala, Cris Moore,
  and Lenka Zdeborov{\'a}.
\newblock Phase transitions and optimal algorithms in high-dimensional gaussian
  mixture clustering.
\newblock {\em arXiv preprint arXiv:1610.02918}, 2016.

\bibitem{DBLP:conf/allerton/LesieurKZ15}
Thibault Lesieur, Florent Krzakala, and Lenka Zdeborov{\'{a}}.
\newblock {MMSE} of probabilistic low-rank matrix estimation: Universality with
  respect to the output channel.
\newblock In {\em 53rd Annual Allerton Conference on Communication, Control,
  and Computing, Allerton 2015, Allerton Park {\&} Retreat Center, Monticello,
  IL, USA, September 29 - October 2, 2015}, pages 680--687, 2015.

\bibitem{DBLP:conf/isit/LesieurKZ15}
Thibault Lesieur, Florent Krzakala, and Lenka Zdeborov{\'{a}}.
\newblock Phase transitions in sparse {PCA}.
\newblock In {\em {IEEE} International Symposium on Information Theory, {ISIT}
  2015, Hong Kong, China, June 14-19, 2015}, pages 1635--1639. {IEEE}, 2015.

\bibitem{lesieur2017constrained}
Thibault Lesieur, Florent Krzakala, and Lenka Zdeborová.
\newblock Constrained low-rank matrix estimation: phase transitions,
  approximate message passing and applications.
\newblock {\em Journal of Statistical Mechanics: Theory and Experiment},
  2017(7):073403, 2017.

\bibitem{matsushita2013low}
Ryosuke Matsushita and Toshiyuki Tanaka.
\newblock Low-rank matrix reconstruction and clustering via approximate message
  passing.
\newblock In {\em Advances in Neural Information Processing Systems}, pages
  917--925, 2013.

\bibitem{mezard1987spin}
Marc M{\'e}zard, Giorgio Parisi, and Miguel Virasoro.
\newblock {\em Spin glass theory and beyond: An Introduction to the Replica
  Method and Its Applications}, volume~9.
\newblock World Scientific Publishing Co Inc, 1987.

\bibitem{milgrom2002envelope}
Paul Milgrom and Ilya Segal.
\newblock Envelope theorems for arbitrary choice sets.
\newblock {\em Econometrica}, 70(2):583--601, 2002.

\bibitem{andrea2008estimating}
Andrea Montanari.
\newblock Estimating random variables from random sparse observations.
\newblock {\em European Transactions on Telecommunications}, 19(4):385--403,
  2008.

\bibitem{nishimori2001statistical}
Hidetoshi Nishimori.
\newblock {\em Statistical physics of spin glasses and information processing:
  an introduction}, volume 111.
\newblock Clarendon Press, 2001.

\bibitem{panchenko2013SK}
Dmitry Panchenko.
\newblock {\em The Sherrington-Kirkpatrick model}.
\newblock Springer Science \& Business Media, 2013.

\bibitem{perry2016statistical}
Amelia Perry, Alexander~S Wein, and Afonso~S Bandeira.
\newblock Statistical limits of spiked tensor models.
\newblock {\em arXiv preprint arXiv:1612.07728}, 2016.

\bibitem{perry2016optimality}
Amelia Perry, Alexander~S Wein, Afonso~S Bandeira, and Ankur Moitra.
\newblock Optimality and sub-optimality of pca for spiked random matrices and
  synchronization.
\newblock {\em arXiv preprint arXiv:1609.05573}, 2016.

\bibitem{rangan2012iterative}
Sundeep Rangan and Alyson~K Fletcher.
\newblock Iterative estimation of constrained rank-one matrices in noise.
\newblock In {\em Information Theory Proceedings (ISIT), 2012 IEEE
  International Symposium on}, pages 1246--1250. IEEE, 2012.

\bibitem{talagrand2010meanfield1}
Michel Talagrand.
\newblock {\em Mean field models for spin glasses: Volume I: Basic examples},
  volume~54.
\newblock Springer Science \& Business Media, 2010.

\bibitem{thouless1977solution}
David~J Thouless, Philip~W Anderson, and Robert~G Palmer.
\newblock Solution of'solvable model of a spin glass'.
\newblock {\em Philosophical Magazine}, 35(3):593--601, 1977.

\bibitem{zdeborova2016statistical}
Lenka Zdeborov{\'a} and Florent Krzakala.
\newblock Statistical physics of inference: Thresholds and algorithms.
\newblock {\em Advances in Physics}, 65(5):453--552, 2016.

\end{thebibliography}
}

\end{document}